\documentclass[11pt, letterpaper]{amsart}
\usepackage[left=1in,right=1in,bottom=0.5in,top=0.8in]{geometry}
\usepackage{amsfonts}
\usepackage{amsmath, amssymb}
\usepackage{graphicx}
\usepackage[font=small,labelfont=bf]{caption}
\usepackage{epstopdf}
\usepackage{xcolor}
\usepackage{amsthm}
\usepackage{float}
\usepackage{pgfplots}
\usepackage{listings}
\usepackage{longtable}
\usepackage{mathrsfs}
\usepackage{bbold}
\usepackage{comment}
\usepackage{esint}

\usepackage{mathtools}
\usepackage{scalerel}[2014/03/10]
\usepackage[usestackEOL]{stackengine}
\usepackage[utf8]{inputenc}
\usepackage[T1]{fontenc}
\usepackage{enumitem}
\DeclarePairedDelimiterX\set[1]\lbrace\rbrace{#1}

\usepackage{scalerel}[2014/03/10]
\usepackage[usestackEOL]{stackengine}

\usepackage[backref=page]{hyperref}
\hypersetup{
plainpages=false,
colorlinks,
linkcolor={cyan!90!black},
citecolor={magenta},
urlcolor={red!90!black},
bookmarksdepth=2,}

\hypersetup{
pdftitle={second-order degenerate parabolic equations with unbounded lower-order terms},
pdfsubject={Mathematics, PDE, Analysis},
pdfauthor={Khalid Baadi},
pdfkeywords={Parabolic systems, Cauchy problems, Green operators, variational methods, Gaussian estimates, Heat Kernel, Energy equality.}
}

\newtheorem{thm}{Theorem}[section]
\newtheorem{cor}[thm]{Corollary}
\newtheorem{prop}[thm]{Proposition}
\newtheorem{lem}[thm]{Lemma}

\theoremstyle{definition}
\newtheorem{defn}[thm]{Definition}

\theoremstyle{remark}
\newtheorem{rem}[thm]{Remark}
\newtheorem{rems}[thm]{Remarks}

\definecolor{energy}{RGB}{114,0,172}
\definecolor{freq}{RGB}{45,177,93}
\definecolor{spin}{RGB}{251,0,29}
\definecolor{signal}{RGB}{203,23,206}
\definecolor{circle}{RGB}{217,86,16}
\definecolor{average}{RGB}{203,23,206}

\definecolor{kb}{rgb}   {.6, 0, 0}

\newcommand{\llangle}{\langle\!\langle}
\newcommand{\rrangle}{\rangle\!\rangle}

\colorlet{shadecolor}{gray!20}
\pgfplotsset{compat=1.9}

\usepgflibrary{fpu}
\makeatletter
\let\c@equation\c@thm
\makeatother
\numberwithin{equation}{section}

\author{Khalid Baadi}
\address{Université Paris-Saclay, CNRS, Laboratoire de Mathématiques d’Orsay, 91405 Orsay, France}
\email{khalid.baadi@universite-paris-saclay.fr}
\date{October 1, 2026}
\title[Degenerate parabolic equations with unbounded lower-order terms]{On well-posedness for second-order degenerate parabolic equations with unbounded lower-order terms}

\DeclareMathOperator*{\supess}{ess\,sup}

\begin{document}

\begin{abstract}
In this paper, we establish the well-posedness of Cauchy problems for weak solutions to second-order degenerate parabolic equations with a non-smooth, time-dependent degenerate elliptic part that includes both bounded and unbounded lower-order terms. The unbounded lower-order terms are allowed to lie in mixed time-space Lebesgue or even Lorentz spaces. Our notion of weak solutions is formulated under minimal assumptions. We prove the existence and uniqueness of a fundamental solution, which coincides with the associated evolution family for the homogeneous problem (\textit{i.e.}, with zero source term) and provides a representation formula for all weak solutions. We also establish $L^2$ off-diagonal estimates for the fundamental solution and derive Gaussian upper bounds under the weak assumption of Moser's $L^2$–$L^\infty$ estimates for weak solutions. Our approach is purely variational and avoids any \textit{a priori} regularity assumptions on weak solutions or regularization via smooth approximations. Two key ingredients are norm inequalities for fractional powers of the degenerate Laplacian, and a set of embeddings that ensure time continuity of weak solutions, extending the classical Lions regularity theorem and accommodating a wide class of source terms.
\end{abstract}

\maketitle

\tableofcontents

\section{Introduction}\label{section 1}

The aim of this paper is to construct solutions to parabolic Cauchy problems of the form
\begin{align}\label{eq: pb principal}
    \left\{
    \begin{array}{ll}
        \partial_t u + \mathcal{B}u  =  f \quad \mathrm{in} \  \mathcal{D}'((0,\mathfrak{T})\times \mathbb{R}^n), \\
        u(0) = \psi \in L^2(\mathbb{R}^n,\mathrm{d}\omega),
    \end{array}\right.
\end{align} 
where $\mathcal{B}$ denotes the degenerate elliptic part, which is non-autonomous and includes both bounded and unbounded lower-order terms. More precisely, we consider
\begin{equation}\label{eq: B}
    \mathcal{B}u = -\omega^{-1} \mathrm{div}_x(A\nabla_x u) -\omega^{-1} \mathrm{div}_x(\omega \, a u) + b \cdot \nabla_x u + c u,
\end{equation}
where the coefficients $A$, $a$, $b$, and $c$ depend on both time and space variables $(t,x)$. The leading-order term involves a matrix-valued function $A = A(t,x)$ with complex measurable entries. The degeneracy of the operator is governed by a spatial weight $\omega = \omega(x)$, which is assumed to be time-independent and to belong to the Muckenhoupt class $A_2(\mathbb{R}^n)$. In this setting, we consider $\omega^{-1}A$ to be bounded, and we may also assume that it satisfies the classical uniform ellipticity condition.

The main and natural questions one may ask are the following:
\begin{itemize}
    \item What conditions should be imposed on the coefficients $a$, $b$, and $c$ to establish the well-posedness of \eqref{eq: pb principal} ? We would like to have conditions that are critical or subcritical with respect to parabolic dilation invariance.
    \item How may these coefficients depend on the weight $\omega$, and what additional assumptions, beyond the $A_2$ condition, might be required on $\omega$ ?
    \item In what sense is \eqref{eq: pb principal} well-posed ?
    \item What is the appropriate existence and uniqueness class, and for which classes of source terms can one solve \eqref{eq: pb principal} ?
    \item What is the regularity of weak solutions ?
    \item Is there a fundamental solution for the degenerate parabolic operator $\partial_t + \mathcal{B}$ that represents weak solutions ? 
    \item If so, does it satisfy bounds similar to those known in the classical literature when $\omega = 1$ ?
\end{itemize}

This paper provides us with an answer to all these questions. Before summarizing our results, it is useful to briefly review some classical and well-known results that address these questions in particular cases, such as the unweighted case (\textit{i.e.}, $\omega = 1$) or situations without unbounded lower-order terms. The existing literature on the subject is vast, so we will highlight only a few representative works to give context and motivation.

Let us begin with the unweighted case $\omega = 1$. When the coefficients are sufficiently regular, several classical methods are available to establish well-posedness and construct the fundamental solution. One of the most effective techniques combines a parametrix construction with the freezing point method \cite{friedman2008partial}. This approach simplifies the problem by locally freezing the coefficients, making them effectively space-independent, and yields explicit solutions represented by smooth kernels $\Gamma(t, x, s, y)$ with Gaussian decay. In the case where the coefficients are merely real, measurable, and possibly unbounded, the classical theory of Ladyzhenskaya, Solonnikov, and Ural'tseva \cite{ladyzhenskaia1968linear} provides a foundational framework. Along similar lines, and still in the case of real coefficients and an elliptic part of order two, Aronson \cite{aronson1967bounds, aronson1968non} constructed generalized fundamental solutions using Nash’s regularity theory \cite{nash1958continuity} and proved upper and lower Gaussian bounds. More recently, Auscher and Egert \cite{auscheregert2023universal} revisited these problems of constructing solutions and extended the results of \cite{ladyzhenskaia1968linear} to more general settings, allowing complex coefficients and even unbounded lower-order terms lying in mixed Lorentz spaces. Their work relies on a variational approach that unifies and generalizes many earlier methods. For additional related works, see also \cite{AMP2019,MR4387945}.

In the weighted case ($\omega \neq 1$) and in the pure second-order setting (\textit{i.e.}, $a = b = c = 0$), the existence of a fundamental solution was established by Cruz-Uribe and Rios \cite{cruz2014corrigendum}, under the additional assumptions that $A$ is real-valued, symmetric, and independent of time. In this case, the fundamental solution is constructed via the semigroup $(e^{-t\mathcal{B}})_{t > 0}$ associated with the operator $\mathcal{B} = -\omega^{-1} \mathrm{div}_x(A(x) \nabla_x)$. More recently, Ataei and Nyström \cite{ataei2024fundamental} extended this result to the case where $A = A(t,x)$ depends on time, still within the pure second-order framework and under the assumption that $A$ remains real-valued. Their construction relies on an approximation argument inspired by Kato’s abstract theory \cite{kato1961abstract}. Both works make use of H\"older continuity estimates for weak solutions to derive pointwise Gaussian upper bounds. More recently, the author \cite{baadi2025degenerate}, using a variational framework, reproved and generalized these results to allow for complex coefficients, and obtained upper pointwise Gaussian bounds using simpler arguments and under weaker assumptions than those required for H\"older continuity estimates of weak solutions. While all of these approaches can be readily extended to include bounded lower-order terms, the situation becomes significantly more delicate in the presence of unbounded ones, and the question of constructing weak solutions has not been addressed.  

Although it is a different question, regularity of local weak solutions has been studied. Ishige \cite{Ishige99} establishes Harnack inequalities for nonnegative weak solutions with real coefficients and some unbounded lower-order terms, extending previous works in the purely second-order setting \cite{MR772255,MR748366,Chiarenza85}. His assumptions rely on a strong reverse doubling condition on the weight (assumption (A5) in \cite{Ishige99}). We will consider different conditions, so regularity theory under our conditions will require some further work that we plan to do subsequently. 

\medskip

We now turn to the presentation of our results. We restrict ourselves to one representative result in the case where only the lower-order coefficients are unbounded and $\mathfrak{T} = \infty$, corresponding to the homogeneous version of our theory. Further extensions and variants will be discussed in the paper.

Fix $r, q \in [2, \infty)$ and assume that
\begin{equation*}
    \omega \in A_2(\mathbb{R}^n) \cap RH_{\frac{q}{2}}(\mathbb{R}^n)
    \quad \text{and} \quad
    \frac{1}{r} + \frac{n}{2q} = \frac{n}{4}.
\end{equation*}
Note that $r = 2$ is possible, in which case $n \ge 3$ and $q = 2^\star := \frac{2n}{n - 2}$, but we always have $q > 2$ since $\frac{1}{r} > 0$. For any open interval $I \subset \mathbb{R}$, we define the mixed space
\begin{equation*}
    \dot{\Sigma}^{r,q}(I):= \left\{ u \in L^1_{\mathrm{loc}}(I;L^2_\omega(\mathbb{R}^n)): u \in L^r(I;L^q_{\omega^{q/2}}(\mathbb{R}^n)) \quad \text{and} \quad \nabla_x u \in L^2(I;L^2_\omega(\mathbb{R}^n)^n) \right\},
\end{equation*}
and endow it with the norm
\begin{equation*}
    \|u \|_{\dot{\Sigma}^{r,q}(I)}:= \|u \|_{L^r(I;L^q_{\omega^{q/2}}(\mathbb{R}^n))}+ \| \nabla_x u \|_{L^2(I;L^2_\omega(\mathbb{R}^n)^n)}.
\end{equation*}
Set
\begin{equation}\label{eq: P_rq}
P_{r,q}:=  \| a \|_{L^{{\frac{2r}{r-2}}}(\mathbb{R};{L^{\frac{2q}{q-2}}(\mathbb{R}^n)^n})} + \| b \|_{L^{{\frac{2r}{r-2}}}(\mathbb{R};{L^{\frac{2q}{q-2}}(\mathbb{R}^n)^n})}+\| c \|_{L^{{\frac{r}{r-2}}}(\mathbb{R};{L^{\frac{q}{q-2}}(\mathbb{R}^n)})}. 
\end{equation}
Set $M:=\| \omega^{-1} A \|_{L^\infty(\mathbb{R}^{n+1})}$. Assuming only $M < \infty$ and $P_{r,q} < \infty$ allows us to define the operator
\begin{equation*}
\partial_t + \mathcal{B} : \dot{\Sigma}^{r,q}(I) \rightarrow \mathcal{D}'(I\times\mathbb{R}^{n}).
\end{equation*}
To construct solutions, we begin with the case of $I=\mathbb{R}$, under the hypothesis that the restriction
\begin{equation*}
    \partial_t + \mathcal{B} : \dot{V}_0 \rightarrow \dot{V}_0^\star
\end{equation*}
is invertible, where $\dot{V}_0$ is a variational space that can be viewed as  
\begin{equation*}
    \dot{V}_0 = L^2(\mathbb{R}; \dot{H}^1_\omega(\mathbb{R}^n)) \cap \dot{H}^{\tfrac{1}{2}}(\mathbb{R}; L^2_\omega(\mathbb{R}^n)).
\end{equation*}
The space $\dot{V}_0^\star$ is the anti-dual of $\dot{V}_0$, and this is what we call the "variational approach".

Let us formulate our representative result. We assume that $A$ is elliptic in the sense of Gårding; that is, there exists $\nu > 0$ such that for all $t \in \mathbb{R}$ and all $u \in H^1_\omega(\mathbb{R}^n)$,
\begin{equation*}
    \nu \int_{\mathbb{R}^n} |\nabla_x u(x)|^2 \, \omega(x) \mathrm{d}x  \le \int_{\mathbb{R}^n} \mathrm{Re}\left( A(t,x) \nabla_x u(x) \cdot \overline{\nabla_x u(x)}\right) \, \mathrm{d}x.
\end{equation*} 
We use fractional powers of the degenerate Laplacian $-\Delta_\omega= -\omega^{-1}\mathrm{div}_x(\omega \nabla_x)$. See Section \ref{ssection 2.2}.
\begin{thm}\label{thm: intro}
    There exists a constant $\varepsilon_0 = \varepsilon_0(M, \nu, [\omega]_{A_2}, [\omega]_{RH_{\frac{q}{2}}}, n, q) > 0$ such that, if $P_{r, q} \leq \varepsilon_0$, then the following assertions hold.
    \begin{enumerate}
    \item Let $\rho \in [2, \infty]$, and define $\beta = \frac{2}{\rho} \in [0, 1]$ and let $\rho'$ denotes the H\"older conjugate of $\rho$. Let $h \in L^{\rho'}((0, \infty);L^2_\omega(\mathbb{R}^n))$ and $\psi \in L^2_\omega(\mathbb{R}^n)$. There exists a unique $u \in \dot{\Sigma}^{r,q}((0,\infty))$ solution to the Cauchy problem 
    \begin{align*}
    \left\{
    \begin{array}{ll}
        \partial_t u + \mathcal{B}u  = (-\Delta_\omega)^{\beta/2}h \quad \mathrm{in} \  \mathcal{D}'((0,\infty)\times \mathbb{R}^n), \\
        u(t) \rightarrow \psi \  \mathrm{ in } \ \mathcal{D'}(\mathbb{R}^n) \ \mathrm{as} \ t \rightarrow 0^+.
    \end{array}\right.
    \end{align*} 
    Moreover, $u\in C_0([0,\infty);L^2_\omega(\mathbb{R}^n))$, with $u(0)=\psi$, $t \mapsto \| u(t)  \|^2_{2,\omega}$ is absolutely continuous on $[0,\infty)$ and we can write the energy equalities. Furthermore, there exists a constant $C \in (0,\infty)$, depending only on the structural constants $M$, $\nu$, $[\omega]_{A_2}$, $[\omega]_{RH_{\frac{q}{2}}}$, $n$ and $q$, such that
    \begin{align*}
            \sup_{t \ge 0} \| u(t) \|_{2,\omega}+ \| \nabla_x u\|_{L^2((0,\infty);L^2_\omega(\mathbb{R}^n)^n)} 
            \leq C  ( \left \| h \right \|_{L^{\rho'}((0,\infty);L^2_\omega(\mathbb{R}^n))}+  \| \psi  \|_{2,\omega}  ).
    \end{align*} 
    \item There exists a unique fundamental solution $\Gamma=(\Gamma(t,s))_{0\leq s \leq t <\infty }$  for $\partial_t+\mathcal{B}$. In particular, for all $t \ge 0$, we have the following representation formula for $u$, the solution in item (1):
    \begin{align*}
    u(t) = \Gamma(t,0)\psi  + \int_{0}^{t} \Gamma(t,\tau) (-\Delta_\omega)^{\beta/2}h(\tau)\ \mathrm d \tau,
    \end{align*}
    where the integral is weakly defined in $L^2_\omega(\mathbb{R}^n)$, and also strongly defined in the Bochner sense when $\rho = \infty$.
    \end{enumerate}
\end{thm}

Let us comment this result. First, it is perturbative in nature with respect to the pure second-order case and follows as a particular instance of Theorem \ref{thm: Pb Cauchy homogène}. The smallness of $P_{r,q}$ is sufficient to guarantee the invertibility and causality of the degenerate parabolic operator $\partial_t + \mathcal{B}$ on $\mathbb{R}$, but it is not necessary. For alternative invertibility conditions that do not require a smallness assumption, as well as for the definition of a fundamental solution, the interpretation of the equations, and further details, we refer the reader to Sections \ref{section 2} and \ref{section 3}. 

Second, note that the quantity $P_{r,q}$ does not depend on $\omega$: it is exactly the one taken in the unweighted theory \cite{ladyzhenskaia1968linear, auscheregert2023universal}. This may seem surprising at first. The reason is that we use the Hardy-Littlewood-Sobolev estimates in Theorem \ref{thm:fractionalHLS} for the degenerate Laplacian $-\Delta_\omega $, 
and this is precisely where the assumption $\omega \in RH_{\frac{q}{2}}(\mathbb{R}^n)$ comes into play. These estimates are a key ingredient of our theory. Third, despite the apparently abstract formulation of our source terms, it covers in fact the following concrete ones and their linear combinations:
\begin{itemize}
\item When $\rho = \infty$, we have $(\Delta_\omega)^{\beta/2} h = h \in L^1((0,\infty); L^2_\omega(\mathbb{R}^n))$. 
\item When $\rho = 2$, we can write, for any $F \in L^{2}((0,\infty); L^2_\omega(\mathbb{R}^n)^n)$, $- \omega^{-1} \mathrm{div}_x(\omega F) = (-\Delta_\omega)^{1/2} h$ for some $h \in L^{2}((0,\infty); L^2_\omega(\mathbb{R}^n))$. See Lemma \ref{lem: div}.
\item When $2 \le \rho < \infty$, set $\Tilde{r} = \rho$ and let $\Tilde{q} \in (2, \infty)$ satisfy $\frac{1}{\Tilde{r}} + \frac{n}{2 \Tilde{q}} = \frac{n}{4}$. If $\omega \in RH_{\Tilde{q}/2}(\mathbb{R}^n)$, then any $g \in L^{\Tilde{r}'}((0,\infty); L_{\omega^{\Tilde{q}'/2}}^{\Tilde{q}'}(\mathbb{R}^n))$ can be written as $g = (-\Delta_\omega)^{\beta/2} h$ for some $h \in L^{\Tilde{r}'}((0,\infty); L_\omega^2(\mathbb{R}^n))$, 
see Lemma \ref{lem:xSobolev}. If, in addition, $\rho \ge r$, the condition on $\omega$ in the last statement is fulfilled; see Point (2) of Remarks \ref{rems: RH}.
\end{itemize}

Note that it is also possible to consider lower-order terms lying in mixed Lorentz spaces. Namely, we may assume that 
\begin{equation*}
L_{r,q} := \| a \|_{L^{\frac{2r}{r-2},\infty}(\mathbb{R}; L^{\frac{2q}{q-2},\infty}(\mathbb{R}^n)^n)} + \| b \|_{L^{\frac{2r}{r-2},\infty}(\mathbb{R}; L^{\frac{2q}{q-2},\infty}(\mathbb{R}^n)^n)} + \| c \|_{L^{\frac{r}{r-2},\infty}(\mathbb{R}; L^{\frac{q}{q-2},\infty}(\mathbb{R}^n))} < \infty,
\end{equation*}
where $L^{p,q}$ denotes Lorentz spaces. An example of such coefficients is
\begin{equation*}
    a(t,x)=\frac{a_\infty(t,x)}{|t|^{n\frac{r-2}{2r}}|x|^{n\frac{q-2}{2q}}},\ b(t,x)=\frac{b_\infty(t,x)}{|t|^{n\frac{r-2}{2r}}|x|^{n\frac{q-2}{2q}}}, \ c(t,x)=\frac{c_\infty(t,x)}{|t|^{n\frac{r-2}{r}}|x|^{n\frac{q-2}{q}}} \ \ \ (t,x)\in \mathbb{R}^{1+n},
\end{equation*}
where $a_\infty$, $b_\infty$, and $c_\infty$ are measurable and bounded functions on $\mathbb{R}^{1+n}$ taking values in $\mathbb{C}^n$, $\mathbb{C}^n$, and $\mathbb{C}$, respectively. For instance, when $r=2$, this gives us $a(t,x)=\frac{a_\infty(t,x)}{|x|}$, $b(t,x)=\frac{b_\infty(t,x)}{|x|}$, $c(t,x)=\frac{c_\infty(t,x)}{|x|^2}$, the usual size for drifts and potentials. The assumption $L_{r,q}<\infty$, together with invertibility of $\partial_t+\mathcal{B}$ on a variational space, is sufficient to construct solutions of equations of the form $\partial_t u + \mathcal{B} u = f$ in $\mathcal{D}'(\mathbb{R} \times \mathbb{R}^n)$. These solutions belong to a suitable class $\dot{\mathcal{L}}^{r,q}(\mathbb{R})$, which is contained in $\dot{\Sigma}^{r,q}(\mathbb{R})$, and we allow for a wider class of source terms $f$ also measured by mixed Lorentz spaces. In the context of parabolic Cauchy problems, we require the stronger condition
\begin{equation*}
\ell_{r,q} := \| a \|_{L^{\frac{2r}{r-2}}((0,\infty); L^{\frac{2q}{q-2},\infty}(\mathbb{R}^n)^n)} + \| b \|_{L^{\frac{2r}{r-2}}((0,\infty); L^{\frac{2q}{q-2},\infty}(\mathbb{R}^n)^n)} + \| c \|_{L^{\frac{r}{r-2}}((0,\infty); L^{\frac{q}{q-2},\infty}(\mathbb{R}^n))} < \infty,
\end{equation*}
which remains a weaker assumption than $P_{r,q}<\infty$. This distinction is used to ensure causality results. We refer the reader to Section \ref{section 4} for further details.

Now, when each of the lower-order terms $a$, $b$, and $c$ is decomposed into the sum of a bounded part and an unbounded one (“unbounded’’ meaning belonging to the mixed Lebesgue or Lorentz spaces introduced above), this corresponds to the inhomogeneous version of our theory, which is treated in Section \ref{section 5}. This is the case, for any pair $(\Tilde{r},\Tilde{q})$ with $r\le \Tilde{r}$, $2<\Tilde{q}$ and $\frac{1}{\Tilde{r}} + \frac{n}{2 \Tilde{q}} > \frac{n}{4}$, when
\begin{equation*}
    a, b \in L^{\frac{2\Tilde{r}}{\Tilde{r}-2}}(\mathbb{R}; L^{\frac{2 \Tilde{q}}{\Tilde{q}-2}}(\mathbb{R}^n)^n)\quad \text{and}  \quad c \in L^{\frac{\Tilde{r}}{\Tilde{r}-2}}(\mathbb{R}; L^{\frac{\Tilde{q}}{\Tilde{q}-2}}(\mathbb{R}^n)),
\end{equation*}
with the additional condition $\frac{1}{\Tilde{r}} - \frac{1}{2\Tilde{q}} < \frac{1}{4}$ when $n=1$. See Remark \ref{rem: Subcritical exponents} for further details and the connection with subcritical exponents appearing in the previously cited literature.

In the final section, Section \ref{section 6}, we establish $L^2$ off-diagonal estimates for the fundamental solution. We also derive Gaussian upper bounds under the assumption that Moser’s $L^2$–$L^\infty$ estimates hold for local weak solutions.

It is important to note that, in both the homogeneous and inhomogeneous settings, we obtain weighted versions of the results in \cite{ladyzhenskaia1968linear,aronson1968non,auscheregert2023universal}, including the case $r=2$. The case $r=\infty$, covered by the classical theory \cite{ladyzhenskaia1968linear}, is excluded, as in the unweighted versions \cite{aronson1968non,auscheregert2023universal}. This is because, as in \cite{auscheregert2023universal}, our approach relies on embeddings of the variational space.

In \cite{auscheregert2023universal}, all the function spaces under consideration are embedded in the ambient space of distributions $\mathcal{D}'(\mathbb{R}^{1+n})$, and the analysis relies on the use of the Fourier transform in the $x$-variable. However, this approach is no longer applicable when dealing with weighted spaces and degenerate operators. A key feature of our approach is the use of distribution spaces adapted to the degenerate operators under consideration, following the framework introduced in \cite{auscherbaadi2024fundamental}. These spaces may appear abstract, as they are tailored to the functional analytic structure of the operator and do not necessarily sit inside $\mathcal{D}'(\mathbb{R}^{1+n})$ in any obvious way. Nonetheless, they provide the appropriate setting to establish existence results for weak solutions. It is important to stress, however, that this level of abstraction is used solely as a technical tool. All our results concerning well-posedness such as existence, uniqueness, and representation formulas are formulated in concrete terms, within explicit functional spaces that preserve a strong connection to the PDE framework.

\medskip

For further details, we refer the reader to the main text.

\subsubsection*{\textbf{Notation:}} Throughout this paper, we adopt the following notation.
\begin{enumerate}[label=$\blacklozenge$]
\item We fix an integer $n \ge 1$. 
\item We use the notation $\mathcal{D}(\Omega)$ for the space of smooth ($C^\infty$) and compactly supported test functions on an open set $\Omega$. Variables will be indicated at the time of use.
\item We use the sans-serif font $C_0$ to denote the space of continuous functions, valued in a normed vector space, that vanish at infinity, and the sans-serif font "loc" to indicate that the prescribed property holds on all compact subsets of the given set.
\item For a given complex normed space $E$, the space $E^\star$ denotes its anti-dual space.
\item By convention, the notation $C = C(a,b,\dots)$ for a constant means that $ C \in (0, \infty)$ and depends only on $(a,b,\dots)$.
\item For any $\rho \in [1, \infty]$, $\rho'$ denotes its H\"older conjugate.
\item $L^p_{\Tilde{\omega}}(\mathbb{R}^n)$ (resp. $L^{p,q}_{\Tilde{\omega}}(\mathbb{R}^n)$) denote Lebesgue (resp. Lorentz) spaces with respect to the measure $\Tilde{\omega}(x)\mathrm{d}x$. When $\Tilde{\omega}=1$, we simply write $L^p(\mathbb{R}^n)$ (resp. $L^{p,q}(\mathbb{R}^n)$). \item We use the same notation for the norms of Lebesgue and Lorentz spaces when dealing with functions taking values in $\mathbb{C}$, $\mathbb{C}^n$ or $M_n(\mathbb{C})$. For example, we write $\| a \|_{L^{\frac{2r}{r-2}}(I; L^{\frac{2q}{q-2}}(\mathbb{R}^n))}$, $\| \nabla_x u \|_{L^2(I; L^2_\omega(\mathbb{R}^n))}$ instead of $\| a \|_{L^{\frac{2r}{r-2}}(I; L^{\frac{2q}{q-2}}(\mathbb{R}^n)^n)}$, $\| \nabla_x u \|_{L^2(I; L^2_\omega(\mathbb{R}^n)^n)}$, etc. We also use the Cauchy-Schwarz inequality implicitly when applying H\"older inequalities to functions taking values in $\mathbb{C}^n$.
\end{enumerate}

\subsubsection*{\textbf{Acknowledgements}}
I am deeply grateful to my PhD thesis advisor, Professor Pascal Auscher, for suggesting this problem, for fruitful discussions, and for providing valuable suggestions that greatly improved earlier versions of this manuscript. I also thank the anonymous referee for their careful reading of the manuscript and for providing suggestions that enhanced its presentation.

\section{Basic Assumptions, Preliminaries, and Parabolic Tools}\label{section 2}

In preparation for the analysis of the well-posedness of parabolic equations in weighted settings, we begin by specifying the main assumptions on the weights and the underlying functional framework. We then present some preliminaries and parabolic tools that will be used repeatedly in the sequel.

\subsection{Muckenhoupt and reverse H\"older weight classes}

We refer to \cite[Ch. V]{Stein1993_HA}, \cite{garcia2011weighted}, and \cite[Ch. 7]{GrafakosBook2014} for classical facts on weights. We say that an almost everywhere (for Lebesgue measure) defined function $\omega : \mathbb{R}^n \rightarrow [0,\infty]$ is a weight if it is locally integrable on $\mathbb{R}^n$. The Muckenhoupt class $A_2(\mathbb{R}^n)$ is defined as the set of all weights satisfying
\begin{equation}\label{MuckWeight}
      [ \omega  ]_{A_2}:= \sup_{Q \subset \mathbb{R}^n  } \left ( \frac{1}{|Q|} \int_Q \omega(x) \, \mathrm{d}x   \right ) \left ( \frac{1}{|Q|} \int_Q \omega^{-1}(x) \, \mathrm{d}x   \right ) < \infty,
\end{equation}
where the supremum is taken with respect to all cubes $Q \subset \mathbb{R}^n$ with sides parallel to the axes and $|Q|$ is the Lebesgue measure of $Q$.

The reverse H\"older class $RH_q(\mathbb{R}^n)$, $q\in [1,\infty)$, is defined as the set of all weights $\omega$ verifying, for some constant $C$,
\begin{equation*}
    \left ( \frac{1}{|Q|} \int_Q \omega^q(x) \, \mathrm{d}x   \right )^{\frac{1}{q}} \leq  \frac{C}{|Q|} \int_Q \omega(x) \, \mathrm{d}x,
\end{equation*}
for all cubes $Q \subset \mathbb{R}^n$ with sides parallel to the axes. The infimum of such constants $C$ is what we denote by $[ \omega  ]_{RH_q}$. Note that $RH_1(\mathbb{R}^n)$ is the class of all weights.

In the following proposition, we summarize some properties of these classes of weights. 

\begin{prop}\label{prop:weights} We have the following properties.
    \begin{enumerate}
        \item $RH_q(\mathbb{R}^n) \subset RH_p(\mathbb{R}^n)$ for all $1 < p \leq q < \infty$.
        \item If $\omega \in RH_q(\mathbb{R}^n)$, $1<q<\infty$, then there exists $p \in (q,\infty)$ such that $\omega \in RH_p(\mathbb{R}^n)$.
        \item $A_2(\mathbb{R}^n) \subset \bigcup_{1< q< \infty} RH_q(\mathbb{R}^n) $.
    \end{enumerate}
\end{prop}

\subsection{Function spaces and the degenerate Laplacian}\label{ssection 2.2}

Throughout this paper, \textbf{we fix a weight $\omega$ belonging to the Muckenhoupt class $A_2(\mathbb{R}^n)$}. We introduce the measure $\mathrm d \omega := \omega(x) \mathrm d x$. If $E \subset \mathbb{R}^n$ a Lebesgue measurable set, we write $\omega(E)$ instead of $\int_E \mathrm d \omega$. 

It follows from \eqref{MuckWeight} that there exists a constant $D=D(n,[ \omega  ]_{A_2})$, called the doubling constant for $\omega$, such that
\begin{equation}\label{DoublingMuck}
    \omega(2Q)\leq D \omega(Q),
\end{equation}
for all cubes $Q \subset \mathbb{R}^n$ with sides parallel to the axes. We may replace cubes by Euclidean balls. For simplicity, we keep using the same notation and constants.

For every $p\ge 1$ and $K \subset \mathbb{R}^n$ a measurable set, we denote by $L^p_\omega(K)$ the space of all measurable functions $f:K \rightarrow \mathbb{C}$ such that $$\|f\|_{L^p_\omega(K)}:=\left ( \int_K |f|^p \, \mathrm d \omega  \right )^{1/p}<\infty.$$
In particular, $L^2_\omega(\mathbb{R}^n)$ is the Hilbert space of square-integrable functions on $\mathbb{R}^n$ with respect to $\mathrm{d}\omega$. We denote its norm by ${\lVert \cdot \rVert}_{2,\omega}$ and its inner product by $\langle \cdot , \cdot \rangle_{2,\omega}$.  The class $\mathcal{D}(\mathbb{R}^n)$ is dense in $L^2_{\omega}(\mathbb{R}^n)$ as $\mathrm d \omega$ is a Radon measure on $\mathbb{R}^n$. Moreover, by \eqref{MuckWeight}, we have 
\begin{equation*}
L^2_{\omega}(\mathbb{R}^n) \subset L^1_{\text{loc}}(\mathbb{R}^n).
\end{equation*}

We define $H^1_{\omega}(\mathbb{R}^n)$ as the space of functions $f \in L^2_\omega(\mathbb{R}^n)$ for which the distributional gradient $\nabla_x f$ belongs to $L^2_\omega(\mathbb{R}^n)^n$, and equip this space with the norm
$\left\| f \right\|_{H^1_\omega} := ( \left\| f \right\|_{2,\omega}^2 + \left\| \nabla_x f \right\|_{2,\omega}^2 )^{1/2}$ making it a Hilbert space. The class $\mathcal{D}(\mathbb{R}^n)$ is dense in $H^1_{\omega}(\mathbb{R}^n)$ and this follows from standard truncation and convolution techniques combined with the boundedness of the (unweighted) maximal operator on $L^2_\omega(\mathbb{R}^n)$. For a proof, see \cite[Thm. 2.5]{kilpelainen1994weighted}.

We define the degenerate Laplacian $-\Delta_\omega$ as the unbounded self-adjoint operator on $L^2_\omega(\mathbb{R}^n)$ associated to the positive symmetric sesquilinear form on $H^1_\omega(\mathbb{R}^n) \times H^1_\omega(\mathbb{R}^n)$ defined by
\begin{equation*}
    (u,v) \mapsto \int_{\mathbb{R}^n} \nabla_x u \cdot \overline{\nabla_x v}  \ \mathrm d \omega.
\end{equation*}
The operator $-\Delta_\omega$ is injective, since the measure $\mathrm{d}\omega$ has infinite mass, as it is doubling \eqref{DoublingMuck}. For all $\beta \in \mathbb{R}$, we define $(-\Delta_\omega)^\beta$ as the self-adjoint operator constructed via the functional calculus for sectorial operators, or equivalently, through the Borel functional calculus, since $-\Delta_\omega$ is self-adjoint. For further details, we refer to \cite{reed1980methods, McIntosh86, haase2006functional}. Let the domain of $(-\Delta_\omega)^\beta$ be denoted by $D((-\Delta_\omega)^\beta)$. For $\beta=0$, $(-\Delta_\omega)^\beta=I$.

The following result concerns norm inequalities for fractional powers of the degenerate Laplacian $-\Delta_\omega$. For the proof, see the case $p = 2$ in \cite[Theorem~1.1 and Proposition~3.24]{auscherbaadi2025hardy}.

\begin{thm}\label{thm:fractionalHLS}
Let $q \in [2,\infty)$. The following statements are equivalent.
\begin{enumerate}
\item $\omega \in RH_{\frac{q}{2}}(\mathbb{R}^n)$.
\item For $\alpha=n\left ( \frac{1}{2}-\frac{1}{q} \right )$, there exists a constant $C$ such that
\begin{equation}\label{eq:fractionalHLS}
\|  \sqrt{\omega} \, f \|_{L^q(\mathbb{R}^n)}= \| f \|_{L^q_{{\omega}^{q/2}}(\mathbb{R}^n)} \leq C \| (-\Delta_\omega)^{\alpha/2} f \|_{L^2_{\omega}(\mathbb{R}^n)}, \quad \forall f \in D((-\Delta_\omega)^{\alpha/2}).
\end{equation}
\end{enumerate}
Moreover, the constant $C$ in (2) depends only on $[\omega]_{A_2}$, $[\omega]_{RH_{\frac{q}{2}}}$, $n$, and $q$.
\end{thm}

\begin{rem}
A typical example of a weight $\omega \in A_2(\mathbb{R}^n) \cap RH_{\frac{q}{2}}(\mathbb{R}^n)$ with $2\le q$, is given by
\begin{equation*}
\omega(x) = |x|^\beta, \quad \forall x \in \mathbb{R}^n, \quad \text{with} \  -\frac{2n}{q} < \beta < n.
\end{equation*}
\end{rem}

\subsection{The distributional anti-duality bracket}
In this paper, we adopt the following definition of the distributional anti-duality bracket, which differs from the standard one. All equalities in $\mathcal{D}'$ are understood in this sense when tested against functions in $\mathcal{D}$.
\begin{defn}[The distributional bracket]\label{def: distributions}
Let $I \subset \mathbb{R}$ be an open interval.
\begin{enumerate}
    \item For $F \in L^{1}_{\mathrm{loc}}(I;L^2_\omega(\mathbb{R}^n)^n)$, we define $- \omega^{-1} \mathrm{div}_x(\omega F)\in \mathcal{D'}(I\times \mathbb{R}^n)$ by setting, for all $\varphi \in \mathcal{D}(I \times \mathbb{R}^n)$,
    \begin{equation*}
     \llangle - \omega^{-1} \mathrm{div}_x(\omega F) , \varphi \rrangle:=\iint_{I \times \mathbb{R}^n} F(t,x) \cdot \nabla_x \overline{\varphi}(t,x) \, \mathrm{d}\omega(x) \mathrm{d}t=\int_{I} \langle F(t), \nabla_x \varphi(t)\rangle_{2,\omega} \, \mathrm{d}t.
    \end{equation*}
    \item For $h \in L^{1}_{\mathrm{loc}}(I;L^2_\omega(\mathbb{R}^n))$ and $\beta \in [0,1]$, we define $(-\Delta_\omega)^{\beta/2}h \in \mathcal{D'}(I\times \mathbb{R}^n)$ by setting, for all $\varphi \in \mathcal{D}(I \times \mathbb{R}^n)$,
    \begin{equation*}
    \llangle (-\Delta_\omega)^{\beta/2}h , \varphi \rrangle:=\iint_{I \times \mathbb{R}^n} h(t,x)  (\overline{(-\Delta_\omega)^{\beta/2}\varphi}(t,\cdot))(x) \, \mathrm{d}\omega(x) \mathrm{d}t = \int_{I} \langle h(t), ((-\Delta_\omega)^{\beta/2}\varphi(t)\rangle_{2,\omega} \, \mathrm{d}t.
    \end{equation*}
    \item If $q_1 \in [2,\infty)$ and $q_2 \in [1, \infty)$ are such that $\omega^{q_1/2}$ is locally integrable on $\mathbb{R}^n$, and if $g \in L^1_{\mathrm{loc}}(I; L^{q_2}_{\omega^{q'_1/2}}(\mathbb{R}^n))$, then we define, for all $\varphi \in \mathcal{D}(I \times \mathbb{R}^n)$,
    \begin{equation*}
    \llangle g , \varphi \rrangle:=\iint_{I \times \mathbb{R}^n} g(t,x) \cdot \overline{\varphi}(t,x) \, \mathrm{d}\omega(x) \mathrm{d}t=\int_{I} \langle g(t) , \varphi(t)\rangle_{2,\omega} \, \mathrm{d}t.
    \end{equation*}
    \item (Time derivative) If $u \in L^{1}_{\mathrm{loc}}(I;L^2_\omega(\mathbb{R}^n))$, we define its distributional time derivative $\partial_t u \in \mathcal{D'}(I\times \mathbb{R}^n)$ by setting, for all $\varphi \in \mathcal{D}(I \times \mathbb{R}^n)$,
    \begin{equation*}
    \llangle \partial_t u , \varphi \rrangle:=\iint_{I \times \mathbb{R}^n} -u(t,x) \cdot \partial_t \overline{\varphi}(t,x) \, \mathrm{d}\omega(x) \mathrm{d}t=\int_{I} -\langle u(t), \partial_t \varphi(t) \rangle_{2,\omega} \, \mathrm{d}t.
    \end{equation*}
\end{enumerate}
For $\varphi$ with spatially compact support, the integral in (1) exists whenever $\varphi \in L^\infty(I;H^1_\omega(\mathbb{R}^n))$, in (2) whenever $\varphi \in L^\infty(I;D((-\Delta_\omega)^{\beta/2}))$, and in (3) whenever $\varphi \in L^\infty(I;L^{q'_2}_{\omega^{q_1/2}}(\mathbb{R}^n))$. Thus, in the same way, and with a slight abuse of notation, we keep writing the distributional anti-duality bracket with the test function $\varphi$ on the right-hand side whenever this is meaningful, even if $\varphi \notin \mathcal{D}(I \times \mathbb{R}^n)$.
\end{defn}

\begin{rem}
    In the above definition, we use sesquilinear brackets. However, we may transition to a distributional duality bracket without complex conjugation, since
    $$\overline{(-\Delta_\omega)^{\beta/2} \, \overline{\phi}} = (-\Delta_\omega)^{\beta/2} \phi, \quad \text{for all } \beta \in [0,1] \text{ and } \phi \in \mathcal{D}(\mathbb{R}^n).$$
\end{rem}

\subsection{Parabolic embeddings and integral identities}\label{Subsection: Tools}

In this section, we collect some relevant abstract tools taken from \cite{auscherbaadi2024fundamental}, to which we refer the reader for details. We do so to avoid repetition in both the homogeneous and the inhomogeneous versions of the theory.

Let $H$ be a separable complex Hilbert space, and let $S$ be a positive, self-adjoint, and injective operator on $H$. There exists a family $(D_{S,\alpha})_{\alpha \in \mathbb{R}}$ of completions of the domains $D(S^\alpha)$ with respect to the homogeneous norms $\|S^\alpha \cdot\|_H$. These spaces embed into a common ambient space $E_\infty$, the anti-dual of the Fr\'echet space $E_{-\infty} := \bigcap_{\alpha \in \mathbb{R}} D(S^\alpha)$. Moreover, $D_{S,\alpha} \cap H = D(S^\alpha)$ for all $\alpha \in \mathbb{R}$.

We refer to \cite{hytonen2016analysis} for the definition and properties of Banach-valued $L^p(I;B)$ spaces. We recall the hierarchy of intermediate homogeneous solution and source spaces, for $ \alpha \in [-1,1]$, 
\begin{align*}
    \dot{V}_\alpha &:=\left \{ u \in L^2(\mathbb{R};D_{S,1}) : D_t^{\frac{1-\alpha}{2}}u \in L^2(\mathbb{R};D_{S,\alpha})\right \},\\
    \dot{W}_\alpha &:= \left \{ D_t^{\frac{1+\alpha}{2}}g \ : \ g \in L^2(\mathbb{R};D_{S,\alpha}) \right \},
\end{align*}
with 
\begin{align*}
    \| u \|_{\dot{V}_\alpha}&:=\left ( \left \| u \right \|_{L^2(\mathbb{R};D_{S,1})}^2+   \| D_t^{\frac{1-\alpha}{2}}u  \|_{L^2(\mathbb{R}; D_{S,\alpha})}^2\right )^{1/2},
   \\ \| f\|_{\dot{W}_\alpha}&:=\| D_t^{-\frac{1+\alpha}{2}}f\|_{L^2(\mathbb{R};D_{S,\alpha})}.
\end{align*}
We have set $D_t^{\gamma} g := \mathcal{F}^{-1}\left( \left| \tau \right|^{\gamma} \mathcal{F}g \right)$ where $\mathcal{F}$ is the Fourier transform on $\mathbb{R}$.

\begin{prop}[Various homogeneous embeddings]\label{prop:embeddings} We have the following embeddings.
    \begin{enumerate}
        \item \textbf{(Hierarchy and anti-duals)} For all $\alpha, \alpha' \in [-1,1]$ with $\alpha \le \alpha'$, we have a continuous inclusion $\dot{V}_\alpha \subset \dot{V}_{\alpha'}$ and $\dot{V}_\alpha^\star = L^2(\mathbb{R}; D_{S,-1}) + \dot{W}_{-\alpha}$. 
        \item \textbf{(Extended Lions' embedding)} For $\alpha \in [-1,0)$, we have $\dot{V}_\alpha \hookrightarrow C_0(\mathbb{R};H)$.
        \item \textbf{(Hardy-Littlewood-Sobolev embedding)} Let $\alpha \in (0,1]$ and let $r=\frac{2}{\alpha} \in [2,\infty)$. Then,  we have $\dot{V}_{\alpha} \hookrightarrow L^{r}(\mathbb{R};D_{S,\alpha})$ and there is a constant $C=C(r)$ such that for all $u \in \dot{V}_{\alpha}$,  
        \begin{equation*} 
        \left\| u \right\|_{L^r(\mathbb{R};D_{S,\alpha})} \leq C \|D_t^{\frac{1-\alpha}{2}}u \|_{L^2(\mathbb{R};D_{S,\alpha})}.
        \end{equation*}
        Consequently, we have $ L^{r'}(\mathbb{R};D_{S,-\alpha}) \hookrightarrow \dot{W}_{-\alpha}$.
    \item \textbf{(Mixed norm embeddings)} For $r\in (2,\infty)$ and $\alpha= 2/ r \in (0,1)$, we have the embeddings
    \begin{align*}
        \dot{V}_1\cap L^{\infty}(\mathbb{R}; H)\hookrightarrow L^{r}(\mathbb{R};D_{S,\alpha}) \ \ \text{and} \ \ \dot{V}_{0} \hookrightarrow L^{r}(\mathbb{R};D_{S,\alpha}), 
    \end{align*}  
    with 
    \begin{align*}
        \left\| u \right\|_{L^r(\mathbb{R};D_{S,\alpha})} &\leq \|u \|_{L^2(\mathbb{R};D_{S,1})}^\alpha \|u \|_{L^\infty(\mathbb{R};H)}^{1-\alpha},\\
        \left\| u \right\|_{L^r(\mathbb{R};D_{S,\alpha})} &\leq C(r)  \|u \|_{L^2(\mathbb{R};D_{S,1})}^\alpha \|D_t^{{1}/{2}}u \|_{L^2(\mathbb{R};H)}^{1-\alpha}.
    \end{align*}
    Consequently,
    \begin{align*}
       L^{r'}(\mathbb{R};D_{S,-\alpha}) \hookrightarrow L^2(\mathbb{R}; D_{S,-1}) + L^1(\mathbb{R};H)\ \ \text{and} \ \ L^{r'}(\mathbb{R};D_{S,-\alpha}) \hookrightarrow L^2(\mathbb{R}; D_{S,-1}) + \dot{W}_0=\dot{V}^\star_0.
    \end{align*}
    \end{enumerate}
\end{prop}

\begin{prop}[Integral identities]\label{prop: Lions non borné}
Let $I$ be an open interval of $\mathbb{R}$. Assume that $u \in L^{2}(I;D_{S,1})$ if $I$ is unbounded, and that $u \in L^{2}(I;D_{S,1}) \cap L^{1}(I;H)$ if $I$ is bounded. Let $\rho \in (2,\infty]$. Assume that
\begin{equation*}
    \partial_t u =  f+g \quad \text{in} \ \mathcal{D}'(I;E_\infty),
\end{equation*}
with $f \in L^2(I;D_{S,-1})$ and $g \in L^{\rho'}(I;D_{S,-{\beta}})$, where ${\beta}={2}/{\rho} \in [0,1)$. Then $u \in C(\Bar{I},H)$ ($C_0(\Bar{I},H)$ if $I$ is unbounded), $ t \mapsto \left \| u(t) \right \|^2_H$ is absolutely continuous on $\Bar{I}$ and for all $ \sigma, \tau \in \Bar{I}$ such that $  \sigma < \tau$,
\begin{align}\label{eq:integralidentity}
\left \| u(\tau) \right \|^2_H-\left \| u(\sigma) \right \|^2_H = 2\mathrm{Re}\int_{\sigma}^{\tau} \langle f(t),u(t)\rangle_{H,-1}  + \langle g(t),u(t)\rangle_{H, -\beta}\  \mathrm d t.
\end{align}
\end{prop}

\section{Lower-order coefficients in mixed Lebesgue spaces}\label{section 3}

In this section, we develop our variational framework for solving second-order degenerate parabolic equations with unbounded lower-order terms in mixed Lebesgue spaces, and obtain existence, uniqueness, regularity, and representation for the equations and their associated Cauchy problems.

\subsection{The variational space and an identification}

On $L^2_\omega(\mathbb{R}^n)$, we define the following positive, self-adjoint, and injective operator:
\begin{equation*}
    S := (-\Delta_\omega)^{1/2}, \quad \text{with} \quad D(S) = D((-\Delta_\omega)^{1/2}) = H^1_\omega(\mathbb{R}^n).
\end{equation*}
We recall that $\mathcal{D}(\mathbb{R}^n)$ is dense in $D(S) = H^1_\omega(\mathbb{R}^n)$ with respect to the graph norm.

We adopt the definitions from Section \ref{Subsection: Tools}, with $S = (-\Delta_\omega)^{1/2}$. The space $\dot{V}_0$ will play a central role as the variational space throughout this work.

\medskip

We record the following useful identification result.
\begin{lem}[Linking $L^2(I; D_{S,-1})$ and $L^2(I; \dot{H}^{-1}_\omega(\mathbb{R}^n))$]\label{lem: div}
Let $I \subset \mathbb{R}$ be an open interval. For any $F \in L^2(I; L^2_\omega(\mathbb{R}^n)^n)$, there exists a unique 
$h \in L^2(I; L^2_\omega(\mathbb{R}^n))$ such that 
\begin{equation*}
    \| h \|_{L^2(I; L^2_\omega(\mathbb{R}^n))} \le \| F \|_{L^2(I; L^2_\omega(\mathbb{R}^n))},
\end{equation*}
and
\begin{equation*}
    \llangle - \omega^{-1} \mathrm{div}_x(\omega F), \varphi \rrangle 
= \int_I \langle h(t), S \varphi(t) \rangle_{2,\omega} \, \mathrm{d}t,
\quad \forall \varphi \in L^2(I; H^1_\omega(\mathbb{R}^n)) \supset \mathcal{D}(I \times \mathbb{R}^n).
\end{equation*}
Conversely, for any $h \in L^2(I; L^2_\omega(\mathbb{R}^n))$, there exists $F \in L^2(I; L^2_\omega(\mathbb{R}^n))$ such that 
\begin{equation*}
\| F \|_{L^2(I; L^2_\omega(\mathbb{R}^n))} \le \| h \|_{L^2(I; L^2_\omega(\mathbb{R}^n))},
\end{equation*}
and the same identity holds:
\begin{equation*}
\llangle - \omega^{-1} \mathrm{div}_x(\omega F), \varphi \rrangle 
= \int_I \langle h(t), S \varphi(t) \rangle_{2,\omega} \, \mathrm{d}t,\quad \forall \varphi \in L^2(I; H^1_\omega(\mathbb{R}^n)) \supset \mathcal{D}(I \times \mathbb{R}^n).
\end{equation*}
In both cases, we have $- \omega^{-1} \mathrm{div}_x(\omega F) = Sh = (-\Delta_\omega)^{1/2} h$ in the weak sense against functions $\varphi \in L^2(I; H^1_\omega(\mathbb{R}^n)) \supset \mathcal{D}(I \times \mathbb{R}^n)$.

\end{lem}
\begin{proof}
    The first result is straightforward by using the duality between $L^2(I; D_{S,1})$ and $L^2(I; D_{S,-1})$. The converse also follows by an application of the Hahn-Banach theorem. We omit the details.
\end{proof}

\subsection{Embeddings}

We adopt the following convention, since $L^1_{\mathrm{loc}}(I; L^2_\omega(\mathbb{R}^n))$ will serve as our main ambient space. For a given normed space $E$ of Banach-valued functions defined on $I$, we define
\begin{equation*}
    E^{\mathrm{loc}} := E \cap L^1_{\mathrm{loc}}(I; L^2_\omega(\mathbb{R}^n)),
\end{equation*} 
equipped with the same norm as $E$. This condition is to be understood in a qualitative sense. For example, for a measurable set $I \subset \mathbb{R}$, $r \in [1,\infty]$, and $\alpha \in \mathbb{R}$, noting that $D_{S,\alpha} \cap L^2_\omega(\mathbb{R}^n) = D(S^{\alpha})$, we have
\begin{equation*}
L^r(I; D_{S,\alpha})^{\mathrm{loc}} = \left\{ u \in L^1_{\mathrm{loc}}(I; L^2_\omega(\mathbb{R}^n)) : \ u(t) \in D(S^\alpha) \ \text{for a.e. } t \in I, \ \text{and} \ \| S^\alpha u \|_{L^r(I; L^2_\omega(\mathbb{R}^n))} < \infty \right\}.
\end{equation*}
In the case of $\dot{V}_\alpha$ for $-1 \le \alpha \le 1$, $I = \mathbb{R}$ automatically when considering $\dot{V}_\alpha^{\mathrm{loc}}$.

\medskip

We now recall the definition of mixed spaces. For pairs of exponents $(r, q) \in [1, \infty]^2$, intervals $I \subset \mathbb{R}$, open sets $\Omega \subset \mathbb{R}^n$, and a weight $\tilde{\omega}$ on $\Omega$, we write $L^r(I;L^q_{\Tilde{\omega}}(\Omega))$ for the mixed norm space of measurable functions $u: I \times \Omega \rightarrow \mathbb{C} $ with 
\begin{equation*}
    \| u \|_{L^r(I;L^q_{\Tilde{\omega}}(\Omega))}:= \left ( \int_I  \left ( \int_\Omega \left | u(t,x)\right |^q \, \mathrm d \Tilde{\omega}(x) \right )^{r/q}  \mathrm d t \right )^{1/r} < \infty,
\end{equation*}
with the usual modifications if either $r = \infty$ or $q = \infty$, or if the functions take values in $\mathbb{C}^n$.

For any open interval $I \subset \mathbb{R}$, we recall that we have set
\begin{equation*}
    \dot{\Sigma}^{r,q}(I)= \left\{ u \in L^1_{\mathrm{loc}}(I;L^2_\omega(\mathbb{R}^n)): u \in L^r(I;L^q_{\omega^{q/2}}(\mathbb{R}^n)) \quad \text{and} \quad \nabla_x u \in L^2(I;L^2_\omega(\mathbb{R}^n)^n) \right\},
\end{equation*}
where the gradient is taken in the sense of distributions, with norm
\begin{equation*}
    \|u \|_{\dot{\Sigma}^{r,q}(I)}:= \|u \|_{L^r(I;L^q_{\omega^{q/2}}(\mathbb{R}^n))}+ \| \nabla_x u \|_{L^2(I;L^2_\omega(\mathbb{R}^n))}.
\end{equation*}
Note that when $\omega^{q/2}$ is locally integrable on $\mathbb{R}^n$, then $\mathcal{D}(I\times \mathbb{R}^n)$ is dense in $L^r(I;L^q_{\omega^{q/2}}(\mathbb{R}^n))$. In particular, $L^r(I;L^q_{\omega^{q/2}}(\mathbb{R}^n))^{\mathrm{loc}}$ is a dense subspace of $L^r(I;L^q_{\omega^{q/2}}(\mathbb{R}^n))$. Moreover, by using the map $u \mapsto (u,\nabla_x u)$ together with Lemma \ref{lem: div}, the duality theory is easily understood by identifying isometrically $\dot{\Sigma}^{r,q}(I)$ to a subspace of $L^r(I;L^q_{\omega^{q/2}}(\mathbb{R}^n)) \times L^2(I;L^2_\omega(\mathbb{R}^n)^n) $. As a consequence, when $(r,q) \in [1,\infty)^2$, we have 
\begin{equation*}
    \dot{\Sigma}^{r,q}(I)^\star = L^2(I;D_{S,-1})+L^{r'}(I,L^{q'}_{\omega^{q'/2}}(\mathbb{R}^n)).
\end{equation*}
More precisely, for $\varphi \in \dot{\Sigma}^{r,q}(I)^\star$, there exists $(F,g) \in L^2(I;L^2_\omega(\mathbb{R}^n)^n) \times L^{r'}(I;L^{q'}_{\omega^{q'/2}}(\mathbb{R}^n)) $ such that 
\begin{align*}
    \varphi=-\omega^{-1} \mathrm{div}_x(\omega \, F)+g,
\end{align*}
that is
\begin{align*}
\ \varphi(u)=\int_I \left ( \langle F(t), \nabla_x u(t) \rangle_{2,\omega}+ \langle g(t), u(t)\rangle_{2,\omega} \right ) \, \mathrm{dt}, \quad \forall u \in \dot{\Sigma}^{r,q}(I),
\end{align*}
with 
\begin{align*}
    \|  \varphi \|_{\dot{\Sigma}^{r,q}(I)^\star} \simeq \inf_{\varphi =-\omega^{-1} \mathrm{div}_x(\omega F)+g} \| F \|_{L^2(I;L^2_\omega(\mathbb{R}^n))} +\| g \|_{L^{r'}(I,L^{q'}_{\omega^{q'/2}}(\mathbb{R}^n))} .
\end{align*}
\begin{rem}
    Here, the inner product $\langle \cdot, \cdot \rangle_{2,\omega}$ serves as the (natural) anti-duality pairing between the spaces $L^q_{\omega^{q/2}}(\mathbb{R}^n)$ and $L^{q'}_{\omega^{q'/2}}(\mathbb{R}^n)$.
\end{rem}

\begin{defn}
    A pair $(r, q)$ is admissible if $\omega \in RH_{\frac{q}{2}}(\mathbb{R}^n)$ and $\frac{1}{r} + \frac{n}{2q} = \frac{n}{4}$, with $2 \le r, q < \infty$. 
\end{defn}
\begin{rems}
\begin{enumerate}\label{rems: RH}
    \item If $\omega \in RH_{\frac{q}{2}}(\mathbb{R}^n)$, then $\omega^{q/2}$ is locally integrable on $\mathbb{R}^n$.
    \item If $(r, q)$ is admissible and $(\Tilde{r}, \Tilde{q})$ is another pair with $r \le \Tilde{r} < \infty$ and $\frac{1}{\Tilde{r}} + \frac{n}{2\Tilde{q}} = \frac{n}{4}$, then it is automatically admissible. Indeed, we have $\Tilde{q} \in (2, q]$ and $\omega \in RH_{\frac{q}{2}}(\mathbb{R}^n) \subset RH_{\frac{\Tilde{q}}{2}}(\mathbb{R}^n)$ by Point (1) of Proposition \ref{prop:weights}. 
\end{enumerate}
\end{rems}

This definition of admissible pairs is motivated by the forthcoming results.

\begin{lem}\label{lem:xSobolev}
    If $(r,q)$ is an admissible pair and $I \subset \mathbb{R}$ is an interval, then we have a continuous inclusion
    \begin{equation*}
        L^r(I;D_{S,\frac{2}{r}})^{\mathrm{loc}} \hookrightarrow L^r(I;L^q_{\omega^{q/2}}(\mathbb{R}^n))^{\mathrm{loc}}.
    \end{equation*}
    By duality, this yields the continuous embedding
    \begin{equation*}
        L^{r'}(I;L^{q'}_{\omega^{q'/2}}(\mathbb{R}^n)) \hookrightarrow L^{r'}(I;D_{S,-\frac{2}{r}})\footnote{This is not a genuine set-theoretic inclusion. The meaning of this embedding is clarified in the proof.}.
    \end{equation*}
\end{lem}
\begin{proof}
If $u \in L^r(I; D_{S, \frac{2}{r}})^{\mathrm{loc}}$, then $u(t) \in D(S^{\frac{2}{r}})$ for almost every $t \in I$. Moreover, the pair $(r, q)$ is admissible if and only if $\omega \in RH_{\frac{q}{2}}(\mathbb{R}^n)$ and $\frac{2}{r} = n \left( \frac{1}{2} - \frac{1}{q} \right)$. Therefore, using Theorem~\ref{thm:fractionalHLS}, there is a constant $C = C([ \omega  ]_{A_2},[ \omega  ]_{RH_{\frac{q}{2}}},n,q)$ such that for almost every $t \in I$,
\begin{equation*}
    \| u(t) \|_{L^q_{\omega^{q/2}}(\mathbb{R}^n)} \le C \| (S^{\frac{2}{r}} u)(t) \|_{L^2_\omega(\mathbb{R}^n)}.
\end{equation*}
The first continuous inclusion then follows by raising both sides of this inequality to the power $r$ and integrating with respect to $t \in I$.

To prove the dual embedding result, fix $g \in L^{r'}(I; L^{q'}_{\omega^{q'/2}}(\mathbb{R}^n))$. Then for all $\varphi \in L^r(I; D_{S,\frac{2}{r}})^{\mathrm{loc}}$, we have by H\"older inequality,
\begin{equation*}
|\llangle g, \varphi \rrangle| \le \| g \|_{L^{r'}(I; L^{q'}_{\omega^{q'/2}}(\mathbb{R}^n))}  \| \varphi \|_{L^r(I; L^q_{\omega^{q/2}}(\mathbb{R}^n))} 
\le C \| g \|_{L^{r'}(I; L^{q'}_{\omega^{q'/2}}(\mathbb{R}^n))}  \| \varphi \|_{L^r(I; D_{S,\frac{2}{r}})}.
\end{equation*}
Moreover, we have $\mathcal{D}(I \times \mathbb{R}^n) \subset L^r(I; D_{S,\frac{2}{r}})^{\mathrm{loc}}$, and it is a dense subspace of both $L^r(I; L^q_{\omega^{q/2}}(\mathbb{R}^n))$ and $L^{r}(I; D_{S,\frac{2}{r}})$. Thus, there exists a unique $h \in L^{r'}(I; L^2_\omega(\mathbb{R}^n))$ such that
\begin{equation*}
 \| h \|_{L^{r'}(I; L^2_\omega(\mathbb{R}^n))}  \le C \| g \|_{L^{r'}(I; L^{q'}_{\omega^{q'/2}}(\mathbb{R}^n))},
\end{equation*}
and
\begin{equation*}
\llangle g, \varphi \rrangle = \llangle h, S^{\frac{2}{r}} \varphi \rrangle, \quad \forall \varphi \in L^r(I; D_{S,\frac{2}{r}})^{\mathrm{loc}} \supset \mathcal{D}(I \times \mathbb{R}^n).
\end{equation*}
\end{proof}

\begin{cor}\label{cor:embeddings}
If $(r,q)$ is an admissible pair, then 
\begin{equation*}
\dot{V}_0^{\mathrm{loc}} \hookrightarrow \dot{V}_{\frac{2}{r}}^{\mathrm{loc}} \hookrightarrow L^r(\mathbb{R};D_{S,\frac{2}{r}})\cap L^2(\mathbb{R};D_{S,1})^{\mathrm{loc}} \hookrightarrow \dot{\Sigma}^{r,q}(\mathbb{R}).
\end{equation*}
By duality, this yields the continuous inclusions
\begin{equation*}
\dot{\Sigma}^{r,q}(\mathbb{R})^{\star}  \hookrightarrow L^2(\mathbb{R};D_{S,-1})+L^{r'}(\mathbb{R};D_{S,-\frac{2}{r}}) \hookrightarrow \dot{V}_{\frac{2}{r}}^{\star} \hookrightarrow \dot{V}_0 ^{\star}.
\end{equation*}
\end{cor}

\begin{proof}
    For the first embedding, holds because the inclusion $\dot{V}_0 \hookrightarrow \dot{V}_{\frac{2}{r}}$ is continuous. The second embedding is the Hardy–Littlewood–Sobolev embedding in the time variable, as stated in Point (3) of Proposition \ref{prop:embeddings}. The third embedding is precisely the result of Lemma \ref{lem:xSobolev}. The embeddings of the dual spaces follow straightforwardly, and we skip details.
\end{proof}

\begin{rem}
    Roughly speaking, what happens here is that we first apply the time Hardy–Littlewood –Sobolev embedding to obtain the inclusion 
    $\dot{H}^{\frac{1-\alpha}{2}}(\mathbb{R}; D_{S,\alpha}) \hookrightarrow L^r(\mathbb{R}; D_{S,\alpha})$, 
    with $\alpha = \frac{2}{r}$. Then, we apply the spatial Hardy–Littlewood–Sobolev embedding 
    $D_{S,\alpha} \left( = \dot{H}_\omega^\alpha(\mathbb{R}^n)\right) \hookrightarrow L^q_{\omega^{q/2}}(\mathbb{R}^n)$, for $\alpha = n \left( \frac{1}{2} - \frac{1}{q} \right)$. Combining these yields the mixed-norm space $L^r L^q_{\omega^{q/2}}$ as an endpoint embedding space. Solving for $\alpha$ in terms of $r$ and $q$ leads to the definition of $(r, q)$ as an admissible pair.

\end{rem}

\subsection{The degenerate parabolic operator \texorpdfstring{$ \partial_t  + \mathcal{B}$}{dtplusB}}

    The parabolic operator $\partial_t + \mathcal{B}$ on $\mathbb{R}\times \mathbb{R}^n$ has a non-autonomous elliptic part $\mathcal{B}$ with unbounded lower-order terms, as defined in \eqref{eq: B}. Its coefficients $A$, $a$, $b$, and $c$ depend on $(t,x)$. The first coefficient is a matrix-valued function $A: \mathbb{R}\times \mathbb{R}^n \rightarrow M_n(\mathbb{C})$ with complex measurable coefficients such that $\omega^{-1}A$ is bounded, so for all $t\in \mathbb{R}$ and $u,v \in H^1_\omega(\mathbb{R}^n)$,
    \begin{equation}\label{eq: borne ponctuelle A}
         |\langle \omega^{-1} A(t) \nabla_x u , \nabla_x v \rangle_{2,\omega}|  \le M \| \nabla_x u \|_{2,\omega} \| \nabla_x v \|_{2,\omega},
    \end{equation}
    with $M:=\| \omega^{-1} A \|_{L^\infty(\mathbb{R}^{n+1})}$. In particular, we have 
    \begin{equation}\label{eq: borne A}
         \int_\mathbb{R} |\langle \omega^{-1} A(t) \nabla_x u(t) , \nabla_x v(t) \rangle_{2,\omega}| \, \mathrm d t \le M \| \nabla_x u \|_{L^2(I;L^2_\omega (\mathbb{R}^n))} \| \nabla_x v \|_{L^2(I;L^2_\omega (\mathbb{R}^n))}, 
    \end{equation}
    for all distributions $u,v \in \mathcal{D}'(\mathbb{R}^{n+1})$ with $\nabla_xu, \nabla_x v \in L^2(\mathbb{R};L^2_\omega(\mathbb{R}^n)^n)$. For all $t\in \mathbb{R}$ and $u,v\in \mathcal{D}'(\mathbb{R}^n)$ with $\nabla_xu, \nabla_x v \in L^2_\omega(\mathbb{R}^n)^n$, we set 
    \begin{equation*}
        \mathcal{A}(t)(u,v):= \langle \omega^{-1} A(t) \nabla_x u(t) , \nabla_x v(t) \rangle_{2,\omega}.
    \end{equation*}
    For all distributions $u,v \in \mathcal{D}'(\mathbb{R}^{n+1})$ with $\nabla_xu, \nabla_x v \in L^2(\mathbb{R};L^2_\omega(\mathbb{R}^n)^n)$, we set the form
    \begin{equation*}
         \mathcal{A}(u,v):= \int_\mathbb{R} \mathcal{A}(t)(u(t),v(t)) \, \mathrm d t= \int_\mathbb{R} \langle \omega^{-1} A(t) \nabla_x u(t) , \nabla_x v(t) \rangle_{2,\omega} \, \mathrm d t. 
    \end{equation*}
    The lower-order coefficients $a$, $b$ are $n$-vectors of complex-valued, measurable functions on $\mathbb{R}^{1+n}$, and $c$ is a complex-valued, measurable functions on $\mathbb{R}^{1+n}$. Next, we introduce the sesquilinear pairings corresponding to the lower-order terms. For all $t\in \mathbb{R}$, we set 
    \begin{equation*}
        \beta(t)(u,v):= \langle a(t) u , \nabla_x v \rangle_{2,\omega}+ \langle b(t) \cdot \nabla_x u , v \rangle_{2,\omega}+\langle c(t) u , v \rangle_{2,\omega}.
    \end{equation*}
    We also introduce the notation
    \begin{equation*}
        \beta(u,v):= \int_{\mathbb{R}} \beta(t)(u(t),v(t)) \, \mathrm{d}t.
    \end{equation*}
    The formal complex adjoint of $\mathcal{B}$ is given by 
    $$\mathcal{B}^\star u=-\omega^{-1} \mathrm{div}_x(A^\star\nabla_x u) -\omega^{-1} \mathrm{div}_x(\omega \, \overline{b} u)+\overline{a} \cdot \nabla_x u + \overline{c} u.$$
    The quantity $P_{r,q}$ defined in \eqref{eq: P_rq}, linked to the lower-order terms, arises in the following two results.
    \begin{lem}\label{lem:conditions coeff}
        Let $(r,q)$ be an admissible pair and assume that $P_{r,q}<\infty$. For all $u,v \in \dot{\Sigma}^{r,q}$ and almost every $t\in \mathbb{R}$, we have 
        \begin{align*}
            |\beta(u,v)(t)| \le &\| a(t) \|_{L^{\frac{2q}{q-2}}(\mathbb{R}^n)} \|u(t) \|_{L^q_{\omega^{q/2}}(\mathbb{R}^n)} \| \nabla_x v(t) \|_{L^2_{\omega}(\mathbb{R}^n)}\\&+ \| b(t) \|_{L^{\frac{2q}{q-2}}(\mathbb{R}^n)} \|v(t) \|_{L^q_{\omega^{q/2}}(\mathbb{R}^n)} \| \nabla_x u(t) \|_{L^2_{\omega}(\mathbb{R}^n)}\\&+ \| c(t) \|_{L^{\frac{q}{q-2}}(\mathbb{R}^n)} \|u(t) \|_{L^q_{\omega^{q/2}}(\mathbb{R}^n)} \|v(t) \|_{L^q_{\omega^{q/2}}(\mathbb{R}^n)} .
        \end{align*}
        In particular, $t \mapsto \beta(u,v)(t) \in L^1(\mathbb{R})$ and we have 
        \begin{equation*}
            |\beta(u,v)| \le P_{r,q} \|u \|_{\dot{\Sigma}^{r,q}(\mathbb{R})} \|v \|_{\dot{\Sigma}^{r,q}(\mathbb{R})}.
        \end{equation*}
    \end{lem}
    \begin{proof}
        For the first inequality, we proceed as follows. Regarding the term involving $a$, we write:
        \begin{equation*}
            \int_{\mathbb{R}^n} |a(t,x)| |u(t,x)| |\nabla_x v(t,x)| \, \mathrm d \omega(x)= \int_{\mathbb{R}^n} |a(t,x)| \times |u(t,x)|\omega^{1/2}(x) \times |\nabla_x v(t,x)| \omega^{1/2}(x) \, \mathrm d x,
        \end{equation*}
        and the result follows from H\"older inequality, since
        $$\frac{1}{\frac{2q}{q-2}}+\frac{1}{q}+\frac{1}{2}=1.$$
        The term involving $b$ is treated in the same way. The same applies to the term involving $c$, since
        $$\frac{1}{\frac{q}{q-2}}+\frac{1}{q}+\frac{1}{q}=1.$$
        The second inequality is treated in exactly the same way, using H\"older's inequality, since
        $$\frac{1}{\frac{2r}{r-2}} + \frac{1}{r} + \frac{1}{2} = 1 \quad \text{and} \quad \frac{1}{\frac{r}{r-2}} + \frac{1}{r} + \frac{1}{r} = 1.$$
    \end{proof}
    Remark that at this stage, we have not used the $RH_{q/2}$ condition on $\omega$.
    By combining Lemma \ref{lem:conditions coeff} and Corollary \ref{cor:embeddings}, we arrive at the following result.
    \begin{cor}\label{cor: borne Beta}
    Let $(r, q)$ be an admissible pair and assume that $P_{r,q} < \infty$. Then there exists a constant $C = C([\omega]_{A_2}, [\omega]_{RH_{\frac{q}{2}}}, n, q) $ such that
    \begin{equation*}
        |\beta(u, v)| \le C P_{r,q} \|u\|_{\dot{V}_0} \|v\|_{\dot{V}_0},
    \end{equation*}
    for all $u, v \in \dot{V}_0^{\mathrm{loc}}$. In other words, $\beta :  \dot{V}_0^{\mathrm{loc}}  \times \dot{V}_0^{\mathrm{loc}} \rightarrow \mathbb{C}$ is a bounded sesquilinear form.
    \end{cor}

    \begin{defn}
        A pair $(r, q)$ is said to be \emph{compatible for lower-order terms} if it is admissible and satisfies $P_{r,q} < \infty$.
    \end{defn}
    
    \begin{rem}
    We could have chosen admissible pairs $(r_i, q_i)$ possibly different for each of the entries $a$, $b$, and $c$ in Lemma \ref{lem:conditions coeff}. In this case, we must assume that $\omega \in RH_{\frac{\max(q_i)}{2}}(\mathbb{R}^n)$. One would then work in the intersection of the spaces $\dot{\Sigma}^{r_i, q_i}$, and all our results remain valid, provided the embedding results are adjusted accordingly. To simplify the presentation, we assume they are the same.
    \end{rem}

\begin{defn}[Forward parabolic operator $\partial_t + \mathcal{B}$ applied to $u \in \dot{\Sigma}^{r,q}(I)$]
    Let $(r,q)$ be a compatible pair for lower-order terms, and let $I \subset \mathbb{R}$ be an open interval. For any $u \in \dot{\Sigma}^{r,q}(I)$ and all test functions $\varphi \in \mathcal{D}(I \times \mathbb{R}^n)$, we define
    \begin{align*}
        \llangle (\partial_t + \mathcal{B})u, \varphi \rrangle
        &:= - \llangle u, \partial_t \varphi \rrangle + \llangle \mathcal{B}u, \varphi \rrangle \\
        &:= \int_I - \langle u(t), \partial_t \varphi(t) \rangle_{2,\omega} \, \mathrm{d}t
        + \int_I \langle A(t) \nabla_x u(t), \nabla_x \varphi(t) \rangle_{2,\omega} 
        + \beta(t)(u(t), \varphi(t)) \, \mathrm{d}t.
    \end{align*}
\end{defn}
\begin{rem}\label{rem: beta et A}
From \eqref{eq: borne A}, we can define the bounded operator, with bound $M$, 
$$\mathcal{A}: L^2(\mathbb{R};D_{S,1}) \rightarrow L^2(\mathbb{R};D_{S,-1}),$$
by setting, for all $u,v \in L^2(\mathbb{R};D_{S,1})$,
$$\llangle \mathcal{A}u, v \rrangle = \mathcal{A}(u,v) .$$
Similarly, if $(r,q)$ is a compatible pair for lower-order terms, then we define the bounded operator, with bound $P_{r,q}$,
$$\beta: \dot{\Sigma}^{r,q}(\mathbb{R}) \rightarrow \dot{\Sigma}^{r,q}(\mathbb{R})^\star= L^2(\mathbb{R};D_{S,-1})+L^{r'}(\mathbb{R},L^{q'}_{\omega^{q'/2}}(\mathbb{R}^n)), $$
by setting, for all $u,v \in \dot{\Sigma}^{r,q}(\mathbb{R})$,
$$\llangle \beta u, v \rrangle = \beta(u,v).$$
Note that we have $\mathcal{B}u= \mathcal{A}u+\beta u$, in $\mathcal{D}'(\mathbb{R}^{1+n})$, for all $u \in \dot{\Sigma}^{r,q}(\mathbb{R})$ and $\mathcal{B}u\in \dot{\Sigma}^{r,q}(\mathbb{R})^\star$. This remains valid when $\mathbb{R}$ is replaced by any open interval $I \subset \mathbb{R}$.
\end{rem}

    We conclude this section with the following result, which states that in suitable settings, the abstract test functions valued in $E_{-\infty}$ (adapted to $S$) can be replaced by standard test functions. This is a key ingredient, as the existence results will be established in abstract spaces.

\begin{prop}[From the concrete to the abstract and vice-versa]\label{prop : passage au concret}
Let $(r,q)$ be a compatible pair for lower-order terms, and let $I \subset \mathbb{R}$ be an open interval. Suppose that $u \in \dot{\Sigma}^{r,q}(I)$. For $f$ in a suitable class of source terms, the following holds: $u$ is a solution to the equation
\begin{equation*}
  \partial_t u + \mathcal{B}u = f 
  \quad \text{in } \mathcal{D}'(I \times \mathbb{R}^n),
\end{equation*}
if and only if $u$ is a solution to the equation
\begin{equation*}
  \partial_t u + \mathcal{B}u = f 
  \quad \text{in } \mathcal{D}'(I; E_\infty),
\end{equation*}
where the latter distributional formulation is interpreted via the anti-duality bracket defined previously,  
against test functions $\varphi \in \mathcal{D}(I; E_{-\infty})$.

\medskip

The same equivalence holds for Cauchy problems with initial data in $L^2_\omega(\mathbb{R}^n)$.

\medskip

Such source terms include
\begin{equation*}
     f = - \omega^{-1} \mathrm{div}_x(\omega F_1) + F_2 
      + (-\Delta_\omega)^{\beta/2} h_1 + h_2 + h_3,
\end{equation*}
where $F_1 \in L^2(I; L^2_\omega(\mathbb{R}^n)^n)$, $F_2 \in L^2(I; L^2_\omega(\mathbb{R}^n))$, $h_1, h_2 \in L^{\rho'}(I; L^2_\omega(\mathbb{R}^n))$ where $\rho \in (2, \infty)$ and $\beta = \frac{2}{\rho} \in (0,1)$ and $h_3 \in L^1(I; L^2_\omega(\mathbb{R}^n))$. 

\end{prop}
\begin{proof}
This follows directly from \cite[Lemma 8.2]{auscherbaadi2024fundamental}, since $\mathcal{D}(\mathbb{R}^n)$ is dense in $D(S) = H^1_\omega(\mathbb{R}^n)$ with respect to the graph norm. Note that $u \in \dot{\Sigma}^{r,q}(I) \subset L^1_{\mathrm{loc}}(I; L^2_\omega(\mathbb{R}^n))$ by assumption, and we also use Lemma \ref{lem: div}, and Lemma \ref{lem:xSobolev} for the norms involving $\|\cdot \|_{L^r(I;L^q_{\omega^{q/2}}(\mathbb{R}^n))}$. The details are omitted.
\end{proof}

\subsection{The variational approach}

We seek weak solutions $u \in \dot{\Sigma}^{r,q}(I)$ to the equation $\partial_t u + \mathcal{B}u = f$ for suitable source terms. We begin with the case $I = \mathbb{R}$ by introducing a variational parabolic operator, allowing the use of the Fourier transform in the time variable. We will later restrict to unbounded or bounded time intervals in the context of Cauchy problems.

We write $H_t$ for the Hilbert transform in the time variable, defined via its symbol $i \tau / |\tau|$. For details, see \cite[Section~6.1]{auscherbaadi2024fundamental}. We define a bounded sesquilinear form 
 $$B_{\dot{V}_0}: \dot{V}_0^{\mathrm{loc}} \times \dot{V}_0^{\mathrm{loc}} \rightarrow \mathbb{C}$$
 by  
\begin{align*}
      B_{\dot{V}_0}(u,v):&=\int_\mathbb{R}\langle H_t D_t^{1/2}u(t), D_t^{1/2}v(t) \rangle_{2,\omega} +  \langle A(t) \nabla_x u(t), \nabla_x \varphi(t) \rangle_{2,\omega} 
        + \beta(u, v)(t)  \, \mathrm{d}t
        \\&= \int_\mathbb{R}\langle H_t D_t^{1/2}u(t), D_t^{1/2}v(t) \rangle_{2,\omega} \, \mathrm{d}t + \mathcal{A}(u, v) + \beta(u, v),
\end{align*}
for all $u,v \in \dot{V}_0^{\mathrm{loc}}$. We define $\mathcal{H} \colon \dot{V}_0^{\mathrm{loc}} \rightarrow ( \dot{V}_0^{\mathrm{loc}} )^\star$ by 
\begin{equation*} 
\llangle \mathcal{H}u, v \rrangle_{\dot{V}_0^\star, \dot{V}_0}:= B_{\dot{V}_0}(u,v), \ u,v \in \dot{V}_0^{\mathrm{loc}}.
\end{equation*}
We have 
\begin{equation*}
   \left ( \partial_t+\mathcal{B} \right )_{\scriptscriptstyle{\vert \dot{V}_0^{\mathrm{loc}}}} = \mathcal{H} \ , \ \ \left ( -\partial_t+\mathcal{B}^\star \right )_{\scriptscriptstyle{\vert \dot{V}_0^{\mathrm{loc}}}} = \mathcal{H}^\star, 
\end{equation*}
where $\mathcal{H}^\star : \dot{V}_0^{\mathrm{loc}} \rightarrow  ({\dot{V}_0^{\mathrm{loc}}})^\star$ is the adjoint of $\mathcal{H}$. Indeed, we have the equality 
$$\int_{\mathbb{R}}\langle H_t D_t^{1/2}u(t), D_t^{1/2}v(t) \rangle_{2,\omega} \, \mathrm{d}t=- \int_{\mathbb{R}} \langle u(t), \partial_t v(t) \rangle_{2,\omega} \, \mathrm{d}t $$
when $u \in \dot{V}_0^{\mathrm{loc}}$ and $v\in \mathcal{D}(\mathbb{R}^{1+n})$. Thus, we may refer to $\mathcal{H}$ (more precisely, to its unique extension from $\dot{V}_0$ to $\dot{V}_0^\star$) as the variational parabolic operator associated with $\mathcal{B}$, since it arises from the sesquilinear form $B_{\dot{V}_0}$ (more precisely, from its unique extension to $\dot{V}_0 \times \dot{V}_0$), and the space $\dot{V}_0$ serves as the underlying variational space.

\subsection{Invertibility and causality} In this section, we give sufficient conditions for invertibility and causality, both of which will be defined rigorously. The first two results are of perturbative nature with respect to the pure second-order case, \textit{i.e.}, they hold when $P_{r,q}$ is small enough, while the last one only requires a lower bound.

We start with results of a perturbative nature. We recall that $A$ is said to be elliptic in the sense of Gårding if there exists $\nu > 0$ such that for all $t \in \mathbb{R}$ and all $u \in H^1_\omega(\mathbb{R}^n)$,
\begin{equation}\label{eq: ellipticité}
    \nu \| \nabla_x u \|^{2}_{2,\omega} \le \mathrm{Re} \left( \langle \omega^{-1} A(t) \nabla_x u, \nabla_x u \rangle_{2,\omega} \right).
\end{equation}
A sufficient condition to obtain this ellipticity is to assume that $A$ is uniformly elliptic, that is 
\begin{equation*}
    \nu \left | \xi \right |^2\omega(x)\leq \mathrm{Re}(A(t,x)\xi\cdot \overline{\xi}),
\end{equation*}
for all $ \xi \in \mathbb{C}^n$ and $(t,x) \in \mathbb{R}\times \mathbb{R}^n$.

\begin{thm}[Invertibility]\label{thm: Invertibility}
    Let $(r, q)$ be a compatible pair for lower-order terms. Then the operator $\mathcal{H}$ extends uniquely to an operator $\widetilde{\mathcal{H}} \colon \dot{V}_0 \rightarrow \dot{V}_0^\star$. Moreover, if \eqref{eq: ellipticité} holds, there exists a constant $\varepsilon_0 = \varepsilon_0(M, \nu, [\omega]_{A_2}, [\omega]_{RH_{\frac{q}{2}}}, n, q) $ such that if $P_{r, q} \leq \varepsilon_0$, then the operator $\widetilde{\mathcal{H}}$ is invertible, with a bound of $\widetilde{\mathcal{H}}^{-1}$ depending only on $M$ and $\nu$.
\end{thm}
\begin{proof}
    We start with the extension result. We first recall that $\dot{V}_0^{\mathrm{loc}}$ is dense in $\dot{V}_0$ for this latter space topology. In particular, up to unique abstract extensions, we have $( \dot{V}_0^{\mathrm{loc}})^\star= \dot{V}_0^\star$ and the uniqueness of the extension of $\mathcal{H}$ is obvious. Hence, we only need to prove that $B_{\dot{V}_0}$ extends to a sesquilinear form on $\dot{V}_0 \times \dot{V}_0$. To do so, we notice that the sesquilinear pairing in $B_{\dot{V}_0}$ involving the half-order time derivative $D_t^{1/2}$ and the Hilbert transform $H_t$ is well-defined on $\dot{V}_0 \times \dot{V}_0$, so no need to extend it. Using the bound \eqref{eq: borne A} and Corollary \ref{cor: borne Beta}, we see that $\mathcal{A}$ and $\beta$ extend respectively to sesquilinear forms $\Tilde{\mathcal{A}}$ and $\Tilde{\beta}$ on $\dot{V}_0 \times \dot{V}_0$ with bounds $M$ and $C P_{r, q}$, respectively, where $C = C([\omega]_{A_2}, [\omega]_{RH_{\frac{q}{2}}}, n, q) $ is a constant. Thus, $\mathcal{H}$ extends to $\widetilde{\mathcal{H}} \colon \dot{V}_0 \rightarrow \dot{V}_0^\star$, defined by 
    \begin{equation*}
        \llangle \widetilde{\mathcal{H}} u, v \rrangle_{\dot{V}_0^\star, \dot{V}_0}:= \Tilde{B}_{\dot{V}_0}(u,v):=\int_\mathbb{R}\langle H_t D_t^{1/2}u(t), D_t^{1/2}v(t) \rangle_{2,\omega} \, \mathrm{d}t + \Tilde{\mathcal{A}}(u, v) + \Tilde{\beta}(u, v), \ u,v \in \dot{V}_0.
    \end{equation*}
    Now, the invertibility result under a smallness assumption on the lower-order terms follows by a perturbation of the case of pure second-order case. The invertibility in this latter case is essentially due to Kaplan \cite{kaplan1966abstract}. For completeness, we provide the full details of this perturbative result.
    
    By the Plancherel theorem and the fact that the Hilbert transform $H_t$ commutes with $D_t^{1/2}$ and $S$,  it is a bijective isometry on $\dot{V}_0$. As  it is skew-adjoint, for all $\delta \in \mathbb{R}$,  $1+\delta H_t$ is an isomorphism on $\dot{V}_0$ and $\|(1+\delta H_t)u\|_{\dot{V}_0}^2= (1+\delta^2)\|u\|^2_{\dot{V}_0}$. The same equality holds on $\dot{V}_0^\star$.
    
    Let $\delta>0$ to be chosen later. The modified sesquilinear form $\Tilde{B}_{\dot{V}_0}(\cdot,(1+\delta H_t)\cdot)$ is bounded on $\dot{V}_0\times \dot{V}_0$ and, for all $u \in \dot{V}_0$, 
    \begin{align*}
    \mathrm{Re} \, \Tilde{B}_{\dot{V}_0}(u, (1+\delta H_t)u) &=  \mathrm{Re} \left ( \int_\mathbb{R}\langle H_t D_t^{1/2}u, D_t^{1/2}(1+\delta H_t)u \rangle_{2,\omega} \, \mathrm dt  \right )+ \mathrm{Re} \, \Tilde{\mathcal{A}}(u,(1+\delta H_t)u)
    \\& \hspace{4cm}+ \mathrm{Re} \,\Tilde{\beta}(u,(1+\delta H_t)u)
    \\
    & = \delta  \| D_t^{1/2}u  \|^2_{L^2(\mathbb{R};L^2_\omega(\mathbb{R}^n))} +\mathrm{Re} \, \Tilde{\mathcal{A}}(u,u) + \delta \, \mathrm{Re} \, \Tilde{\mathcal{A}}(u,H_t u)+  \mathrm{Re} \,\Tilde{\beta}(u,(1+\delta H_t)u),
    \end{align*}
    where we have used that $H_t$ is skew-adjoint, hence 
    \begin{align*}
    \mathrm{Re}\int_\mathbb{R}\langle H_t D_t^{1/2}u(t), D_t^{1/2}u(t) \rangle_{2,\omega} \ \mathrm dt =0.
    \end{align*}
    Using \eqref{eq: ellipticité}, \eqref{eq: borne ponctuelle A} (more precisely, their extensions to $D_{S,1}$) together with Corollary \ref{cor: borne Beta}, we obtain
    \begin{align*}
    \mathrm{Re} \, \Tilde{B}_{\dot{V}_0}(u, (1+\delta H_t)u) \geq \delta  \| D_t^{1/2}u  \|^2_{L^2(\mathbb{R};L^2_\omega(\mathbb{R}^n))}+(\nu -\delta M)\left \| u \right \|_{L^2(\mathbb{R};D_{S,1})}^2 - C P_{r,q} \sqrt{1+\delta^2} \|u \|_{\dot{V}_0}^2 .
    \end{align*}
    Choosing first $\delta = \frac{\nu}{1+M}$ and then $\varepsilon_0 > 0$ such that $C \varepsilon_0 \sqrt{1 + \delta^2} = \delta / 2$, then $P_{r,q} \le \varepsilon_0$ implies that
    \begin{align}\label{eq: LaxM H}
     \mathrm{Re}\, \llangle \widetilde{\mathcal{H}}u , (1+\delta H_t)u) \rrangle_{V_0^\star,V_0}=\mathrm{Re} (B_{V_0}(u, (1+\delta H_t)u)) \geq \frac{\delta}{2} \|u \|_{\dot{V}_0}^2 , \ \forall u \in \dot{V}_0.
    \end{align}
    Fix $f \in \dot{V}_0^\star$. The Lax-Milgram lemma implies that  there exists a unique $u \in \dot{V}_0$ such that $$\Tilde{B}_{\dot{V}_0}(u, (1+\delta H_t)\cdot)=  (1+\delta H_t)^\star \circ f.$$ Furthermore, we have the estimate 
    \begin{align*}
    \left \| u \right \|_{\dot{V}_0} \leq \frac{2}{\delta}\left \| (1+\delta H_t)^\star \circ f \right \|_{\dot{V}_0^\star}.
    \end{align*}
    Using the fact that $(1+\delta H_t)^\star$ is an isomorphism on $\dot V_0^\star$ with operator norm equal to $\sqrt{1+\delta^2}$, we have that for each $f \in \dot V_0^\star$ there exists a unique $u \in \dot V_0$ such that $\Tilde{B}_{\dot{V}_0}(u,\cdot)=f$, \textit{i.e.} $\widetilde{\mathcal{H}}u=f$ with 
    \begin{align*}
    \left \| u \right \|_{\dot{V}_0} \leq \frac{2}{\delta}\times \sqrt{ 1+\delta^2 } \left \| f \right \|_{\dot{V}_0^\star}. 
    \end{align*}
\end{proof}

\begin{rem}\label{rem: TildeB}
Note that the elliptic part $\mathcal{B}$, initially defined as a bounded operator $\mathcal{B}: \dot{\Sigma}^{r,q}(\mathbb{R}) \rightarrow \dot{\Sigma}^{r,q}(\mathbb{R})^\star = L^2(\mathbb{R};D_{S,-1}) + L^{r'}(\mathbb{R};L^{q'}_{\omega^{q'/2}}(\mathbb{R}^n))$, when restricted to $L^r(\mathbb{R};D_{S,\frac{2}{r}}) \cap L^2(\mathbb{R};D_{S,1})^{\mathrm{loc}}$ by Corollary \ref{cor:embeddings}, extends uniquely to a bounded operator
\begin{equation*}
\Tilde{\mathcal{B}}: L^r(\mathbb{R};D_{S,\frac{2}{r}}) \cap L^2(\mathbb{R};D_{S,1}) \rightarrow \left(L^r(\mathbb{R};D_{S,\frac{2}{r}}) \cap L^2(\mathbb{R};D_{S,1})\right)^\star = L^2(\mathbb{R};D_{S,-1}) + L^{r'}(\mathbb{R};D_{S,-\frac{2}{r}}).
\end{equation*}
In particular, we have $$\Tilde{\mathcal{B}}=\Tilde{\mathcal{A}}+\Tilde{\beta} \quad \text{and} \quad \left( \partial_t + \Tilde{\mathcal{B}} \right)_{\scriptscriptstyle{\vert \dot{V}_0}} = \widetilde{\mathcal{H}}.$$
\end{rem}

We now state the following causality result of perturbative nature.

\begin{thm}[Causality]\label{thm: Causality}
    Let $(r,q)$ be a compatible pair for lower-order terms. Assume that \eqref{eq: ellipticité} holds. There exists a constant $\varepsilon_0=\varepsilon_0(M,\nu,[\omega]_{A_2}, [\omega]_{RH_{\frac{q}{2}}}, n, q)$ such that $P_{r,q} \le \varepsilon_0$ implies that $\partial_t+\mathcal{B}$ is causal in the following sense: if $I$ is an open interval that contains a neighborhood of $-\infty$ and $u \in \dot{\Sigma}^{r,q}(I)$ satisfies $\partial_t u + \mathcal{B}u = 0$ in $\mathcal{D}'(I \times \mathbb{R}^n)$, then $u = 0$.

\end{thm}
\begin{proof}
    If $u \in \dot{\Sigma}^{r,q}(I)$ satisfies $\partial_t u + \mathcal{B}u = 0$ in $\mathcal{D}'(I \times \mathbb{R}^n)$, then $\partial_t u = -\mathcal{B}u=-\mathcal{A}u-\beta u$ in $\mathcal{D}'(I \times \mathbb{R}^n)$. We have $u\in L^2(I;D_{S,1})^{\mathrm{loc}}$ and $\mathcal{A}u+\beta u \in L^2(I;D_{S,-1})+L^{r'}(I;L^{q'}_{\omega^{q'/2}}(\mathbb{R}^n)) \subset L^2(I;D_{S,-1}) + L^{r'}(I;D_{S,-\frac{2}{r}})$. By combining Propositions \ref{prop : passage au concret} and \ref{prop: Lions non borné}, we have $u \in C_0(\overline{I}; L^2_\omega(\mathbb{R}^n))$ and, for all $\sigma, \tau \in \overline{I}$ such that $\sigma < \tau$,
    \begin{align*}
        \left \| u(\tau) \right \|^2_{2,\omega}-\left \| u(\sigma) \right \|^2_{2,\omega} = -2\mathrm{Re}\int_{\sigma}^{\tau} \langle \omega^{-1}A(t)\nabla_x u (t), \nabla_x u (t) \rangle_{2,\omega}  + \beta(u,u)(t) \, \mathrm d t.
    \end{align*}
    We choose $\tau \in \overline{I}$ such that $\left \| u(\tau) \right \|^2_{2,\omega}=\sup_{t\in \overline{I}} \left \| u(t) \right \|^2_{2,\omega}=:\underline{S}$ and set $\underline{I} = \int_{-\infty}^{\tau} \left \| \nabla_x u(t) \right \|^2_{2,\omega} \, \mathrm d t$.  Letting $\sigma \to -\infty$ in the above equality and using \eqref{eq: ellipticité}, we deduce that 
    \begin{equation}\label{eq:Causality}
        \underline{S} \le -2\nu \, \underline{I} + 2 \int_{-\infty}^{\tau} |\beta(u,u)(t)| \, \mathrm d t.
    \end{equation}
    Now, by combining Lemma \ref{lem:xSobolev} with the first embedding in (4) of Proposition \ref{prop:embeddings}, we obtain the existence of a constant $C = C([\omega]_{A_2}, [\omega]_{RH_{\frac{q}{2}}}, n, q)$ such that
    \begin{align*}
        \| u \|_{L^r((-\infty,\tau);L^q_{\omega^{q/2}}(\mathbb{R}^n))} \leq C  \| u \|_{L^r((-\infty,\tau);D_{S,\frac{2}{r}})} \leq C  \, \underline{S}^{\frac{1}{2}-\frac{1}{r}} \underline{I}^{\frac{1}{r}}.
    \end{align*}
    Combining this inequality together with an estimate as in the proof of Lemma \ref{lem:conditions coeff}, we obtain 
    \begin{equation*}
        \int_{-\infty}^{\tau} |\beta(u,u)(t)| \, \mathrm d t \leq C P_{r,q}  \left ( \underline{S}^{\frac{1}{2}-\frac{1}{r}} \underline{I}^{\frac{1}{2}+\frac{1}{r}}+\underline{S}^{1-\frac{2}{r}} \underline{I}^{\frac{2}{r}} \right ).
    \end{equation*}
    Going back to \eqref{eq:Causality}, we have 
    \begin{equation*}
         \underline{S} \le -2\nu \, \underline{I} + 2 C P_{r,q}  \left ( \underline{S}^{\frac{1}{2}-\frac{1}{r}} \underline{I}^{\frac{1}{2}+\frac{1}{r}}+\underline{S}^{1-\frac{2}{r}} \underline{I}^{\frac{2}{r}} \right ).
    \end{equation*}
    Since $r \in [2,\infty)$, we make use of Young's convexity inequality to deduce the inequality
    \begin{equation*}
        \underline{S} \le -2\nu \, \underline{I} + 2 C P_{r,q} \left ( (\frac{3}{2}-\frac{3}{r})\underline{S} + (\frac{1}{2}+\frac{3}{r})\underline{I} \right ).
    \end{equation*}
    Now, if $P_{r,q} <\min \left(\frac{\nu}{C(\frac{1}{2}+\frac{3}{r})}, \frac{1}{2C(\frac{3}{2}-\frac{3}{r})}  \right)=:2\varepsilon_0$, we conclude that $\underline{S} = 0$, and hence $u = 0$.
\end{proof}

Finally, we state the following invertibility and causality result based on lower bounds.

\begin{thm}[Invertibility and causality through lower bounds]\label{thm : Causality and invertibility}
Let $(r,q)$ be a compatible pair for lower-order terms. 
\begin{enumerate}
    \item Assume that there exists a constant $c>0$ such that 
    \begin{equation*}
        \mathrm{Re} \, \llangle \mathcal{B}u, u \rrangle \ge c \| \nabla_x u \|^2_{L^2(\mathbb{R};L^2_\omega(\mathbb{R}^n))}\quad \text{for all} \ u \in \dot{V}_0^{\mathrm{loc}}.
    \end{equation*}
    Then $\widetilde{\mathcal{H}}$ is invertible.
    \item Assume that
    \begin{equation*}
       \mathrm{Re} \,( \mathcal{A}(t)(u,v)+ \beta(t)(u, u)) \ge 0 \quad \text{for all} \ t\in \mathbb{R} \ \ \text{and} \ \ u \in H^1_{\omega}(\mathbb{R}^n).
    \end{equation*}
    Then $\partial_t+\mathcal{B}$ is causal in the sense of Theorem \ref{thm: Causality}.
\end{enumerate}
    
\end{thm}

\begin{proof}
    We follow \cite[Theorem 2.42]{auscheregert2023universal}. We start with the proof of (1). Using the assumption (more precisely, its extension to $\dot{V}_0$) and proceeding as in the proof of Theorem \ref{thm: Invertibility}, we have
    \begin{align*}
         \mathrm{Re} \, \llangle \widetilde{\mathcal{H}} u, (1+\delta H_t)u \rrangle_{\dot{V}_0^\star, \dot{V}_0} &\ge \delta  \| D_t^{1/2}u  \|^2_{L^2(\mathbb{R};L^2_\omega(\mathbb{R}^n))} + c  \| u \|^2_{L^2(\mathbb{R};D_{S,1})} -\delta M \|  u \|_{L^2(\mathbb{R};D_{S,1})} \|  H_t u \|_{L^2(\mathbb{R};D_{S,1})}\\& \hspace{4cm}-\delta P_{r,q} \|u \|_{\dot{\Sigma}^{r,q}} \|H_t u \|_{\dot{\Sigma}^{r,q}}
         \\& \ge \delta  \| D_t^{1/2}u  \|^2_{L^2(\mathbb{R};L^2_\omega(\mathbb{R}^n))} + (c-\delta \Tilde{M})  \| u \|^2_{L^2(\mathbb{R};D_{S,1})} \\& \hspace{0.5cm} - \delta P_{r,q} \left ( \|u \|_{L^r(\mathbb{R};L^q_{\omega^{q/2}}(\mathbb{R}^n))}+ \|H_t u \|_{L^r(\mathbb{R};L^q_{\omega^{q/2}}(\mathbb{R}^n))} \right )^2,
    \end{align*}
    with $\Tilde{M}:=M+2P_{r,q}$, where in the second inequality we used that $H_t$ is an isometry on $L^2(\mathbb{R};D_{S,1})$, and trivial lower bounds together with elementary estimates such as $2ab\le a^2+b^2$ for $a,b\ge0$. By combining Lemma \ref{lem:xSobolev} with the second embedding in (4) of Proposition \ref{prop:embeddings}, there are constants $C_1 = C_1([\omega]_{A_2}, [\omega]_{RH_{\frac{q}{2}}}, n, q)$ and $C_2 = C_2(r)$ such that
    \begin{equation}\label{eq: mixed}
        \| v \|_{L^r(\mathbb{R};L^q_{\omega^{q/2}}(\mathbb{R}^n))} \leq C_1  \| v \|_{L^r(\mathbb{R};D_{S,\frac{2}{r}})} \leq C_1  C_2 \|v \|_{L^2(\mathbb{R};D_{S,1})}^\alpha \|D_t^{{1}/{2}}v \|_{L^2(\mathbb{R};L^2_\omega(\mathbb{R}^n))}^{1-\alpha},
    \end{equation}
    for all $v \in \dot{V}_0$, with $\alpha=\frac{2}{r}$. We set $C:=C_1C_2$. 
    
    If $r=2$, then there is no contribution of $D_t^{{1}/{2}}$ in \eqref{eq: mixed} and by combining this inequality with the fact that $H_t$ is an isometry on $L^2(\mathbb{R};D_{S,1})$, we obtain 
    \begin{align*}
        \mathrm{Re} \, \llangle \widetilde{\mathcal{H}} u, (1+\delta H_t)u \rrangle_{\dot{V}_0^\star, \dot{V}_0} \ge \delta  \| D_t^{1/2}u  \|^2_{L^2(\mathbb{R};L^2_\omega(\mathbb{R}^n))} + (c-\delta( \Tilde{M}+4C^2 P_{r,q}))  \| u \|^2_{L^2(\mathbb{R};D_{S,1})},
    \end{align*}
    and we choose $\delta>0$ such that $c-\delta( \Tilde{M}+4C^2 P_{r,q})>0$ and we obtain the invertibility of $\widetilde{\mathcal{H}}$. 
    
    If $r>2$, \textit{i.e.} $\alpha \in (0,1)$, then for $\varepsilon>0$ to be fixed later, by convexity, from \eqref{eq: mixed} we obtain
    \begin{equation*}
        \| v \|_{L^r(\mathbb{R};L^q_{\omega^{q/2}}(\mathbb{R}^n))} \leq C   \left ( \alpha  \varepsilon^{-\frac{1}{\alpha}} \|v \|_{L^2(\mathbb{R};D_{S,1})}+ (1-\alpha) \varepsilon^{\frac{1}{1-\alpha}}   \|D_t^{{1}/{2}}v \|_{L^2(\mathbb{R};L^2_\omega(\mathbb{R}^n))} \right ), \quad \text{for all} \ v \in \dot{V}_0^{\mathrm{loc}}.
    \end{equation*}
    Using this inequality with the fact that $H_t$ is an isometry on $\dot{V}_0$, we obtain 
    \begin{align*}
       \mathrm{Re} \,  \llangle \widetilde{\mathcal{H}} u, (1+\delta H_t)u \rrangle_{\dot{V}_0^\star, \dot{V}_0} \ge &\delta \left ( 1- 8C^2 P_{r,q}(1-\alpha)^2 \varepsilon^{\frac{2}{1-\alpha}} \right )   \| D_t^{1/2}u  \|^2_{L^2(\mathbb{R};L^2_\omega(\mathbb{R}^n))} \\&+ \left ( c-\delta \Tilde{M}-8\delta C^2 P_{r,q} \alpha^2 \varepsilon^{-\frac{2}{\alpha}} \right )  \| u \|^2_{L^2(\mathbb{R};D_{S,1})}.
    \end{align*}
    We begin by choosing $\varepsilon > 0$ such that $1 - 8 C^2 P_{r,q}(1 - \alpha)^2 \varepsilon^{\frac{2}{1 - \alpha}} > 0$, and then choose $\delta > 0$ such that $c - \delta\left( \Tilde{M} +  8C^2 P_{r,q} \alpha^2 \varepsilon^{-\frac{2}{\alpha}} \right) > 0$. Under these conditions, we conclude the invertibility of $\widetilde{\mathcal{H}}$.

    For the proof of (2), we argue as in the proof of Theorem \ref{thm: Causality}. By the energy equality, we obtain
    \begin{align*}
    \left\| u(\tau) \right\|_{2,\omega}^2 
    = -2\,\mathrm{Re}\int_{-\infty}^{\tau} \langle \mathcal{B}u(t), u(t) \rangle \, \mathrm{d}t.
    \end{align*}
    We then conclude directly from the assumption of (2) that $u=0$.
\end{proof}

\subsection{Existence and uniqueness results} In this section, we establish the well-posedness of parabolic equations of the form $\partial_t u + \mathcal{B}u = f$ on $\mathbb{R}^{1+n} = \mathbb{R} \times \mathbb{R}^n$, where the source term $f$ belongs to various function spaces. We begin by stating the following assumptions, for which we have already provided sufficient conditions in the previous section.
\begin{align*}
    &\textbf{(I)} \quad \widetilde{\mathcal{H}} \ \text{is invertible;} \\
    &\textbf{(C)} \quad \partial_t + \mathcal{B} \ \text{is causal in the sense of Theorem \ref{thm: Causality}.}
\end{align*}

\textbf{Until the end of this Section \ref{section 3}, we assume that \textbf{(I)} holds and will not repeat this assumption in the statements.} We set $\mathfrak{I}:= \| (\widetilde{\mathcal{H}} )^{-1} \|_{\dot{V}_0^\star \rightarrow \dot{V}_0}$. The assumption \textbf{(C)} will be invoked only when it is used. \textbf{We also fix a pair $(r, q)$ that is compatible for lower-order terms.} Note that this includes the assumption $\omega \in RH_{\frac{q}{2}}(\mathbb{R}^n)$. 

\begin{rems}
\begin{enumerate}
\item Assumption \textbf{(I)} is equivalent to the statement that $\widetilde{\mathcal{H}^\star}$, the unique extension of $\mathcal{H}^\star$ to $\dot{V}_0$, is invertible, and we have $\mathfrak{I} = \| (\widetilde{\mathcal{H}^\star})^{-1} \|_{\dot{V}_0^\star \rightarrow \dot{V}_0}$.
\item Assumption \textbf{(C)} is equivalent to the statement that $-\partial_t + \mathcal{B}^\star$ is anti-causal (when $I$ is assumed to be a neighborhood of $+\infty$).
\end{enumerate}
Thus, all the results presented here also apply to backward equations.
\end{rems}

\subsubsection{Uniqueness} 

We establish the following uniqueness result, which is the only instance in which we move from the concrete setting to the abstract one.

\begin{prop}[Uniqueness]\label{thm: Uniqueness}
If $u\in \dot{\Sigma}^{r,q}(\mathbb{R})$ is a weak solution of $\partial_t u +\mathcal{B}u=0$, then $u=0$.
\end{prop}
\begin{proof}
Using the equation, we have $\partial_t u - \Delta_\omega u = -\Delta_\omega u - \mathcal{B}u$ in $\mathcal{D}'(\mathbb{R}^{1+n})$, with $-\Delta_\omega u \in L^2(\mathbb{R};D_{S,-1})$ and $\mathcal{B}u \in \dot{\Sigma}^{r,q}(\mathbb{R})^\star = L^2(\mathbb{R};D_{S,-1}) + L^{r'}(\mathbb{R},L^{q'}_{\omega^{q'/2}}(\mathbb{R}^n)) \hookrightarrow L^2(\mathbb{R};D_{S,-1}) + L^{r'}(\mathbb{R};D_{S,-\frac{2}{r}})$ by Lemma \ref{lem:xSobolev}. Using Proposition \ref{prop : passage au concret}, we also have 
\begin{equation*}
    \partial_t u - \Delta_\omega u = -\Delta_\omega u - \mathcal{B}u \quad \text{in} \ \mathcal{D}'(\mathbb{R};E_\infty).
\end{equation*}
As $-\Delta_\omega u - \mathcal{B}u \in L^2(\mathbb{R};D_{S,-1}) + L^{r'}(\mathbb{R};D_{S,-\frac{2}{r}}) \hookrightarrow \dot{V}_{\frac{2}{r}}^{\star}=\dot{W}_{-1}+\dot{W}_{-\frac{2}{r}}$ by Corollary \ref{cor:embeddings}, we use \cite[Theorem 4.16]{auscherbaadi2024fundamental} to conclude that $u \in \dot{V}_{-1}+\dot{V}_{-\frac{2}{r}}=\dot{V}_{-\frac{2}{r}} \subset \dot{V}_0$ by Proposition \ref{prop:embeddings}, Point (1). Thus, we have $u \in \dot{V}_0^{\mathrm{loc}}$ and $\mathcal{H}u = 0$, hence $u = 0$ by the invertibility assumption \textbf{(I)}.
\end{proof}

\subsubsection{Source term in suitable subspaces of $\dot{V}_0^\star$} 

We recall that we have 
$$\left( \partial_t + \Tilde{\mathcal{B}} \right)_{\scriptscriptstyle{\vert \dot{V}_0}} = \widetilde{\mathcal{H}},$$
and $\Tilde{\mathcal{B}}$ defined in Remark \ref{rem: TildeB}. We also recall that, for any open interval $I\subset \mathbb{R}$,
\begin{equation*}
    L^{\rho'}(I;D_{S,-\beta})= \left\{ (-\Delta_\omega)^{\frac{\beta}{2}}h, \ h\in L^{\rho'}(I;L^2_\omega(\mathbb{R}^n) \right\}.
\end{equation*}
We now state the following result.
\begin{prop}[Source term in $L^{\rho'}(\mathbb{R};D_{S,-\beta})$, $2\le\rho<\infty$]\label{prop: existence OK}
Let $\rho \in [2,\infty)$ and set $\beta = \frac{2}{\rho}$. Let $h \in L^{\rho'}(\mathbb{R};L^2_\omega(\mathbb{R}^n))$. Then, there exists a unique $u \in \dot{\Sigma}^{r,q}(\mathbb{R})$ solution of
    \begin{equation*}
    \partial_t u + \mathcal{B}u =  (-\Delta_\omega)^{\frac{\beta}{2}} h \quad \text{in} \ \mathcal{D}'(\mathbb{R}^{1+n}).
    \end{equation*}
    Moreover, $u \in C_0(\mathbb{R};L^2_\omega(\mathbb{R}^n)) \cap \dot{V}_0$, and there exists a constant $C = C(M,\mathfrak{I}, P_{r,q}, [\omega]_{A_2}, [\omega]_{RH_{\frac{q}{2}}}, n, q,\rho)$ such that
    \begin{equation}\label{eq: estimates energy}
    \sup_{t \in \mathbb{R}} \| u(t) \|_{L^2_\omega(\mathbb{R}^n)} + \| \nabla_x u \|_{L^2(\mathbb{R};L^2_\omega(\mathbb{R}^n))} 
    \leq C\| h \|_{L^{\rho'}(\mathbb{R};L^2_\omega(\mathbb{R}^n))}.
    \end{equation}
\end{prop}

\begin{proof}
Uniqueness follows from Proposition \ref{thm: Uniqueness}. For existence, we observe that the source term $ (-\Delta_\omega)^{\frac{\beta}{2}} h$ lies in $\dot{V}_0^\star$, so defining $u := \widetilde{\mathcal{H}}^{-1}( (-\Delta_\omega)^{\frac{\beta}{2}} h)$ provides a function $u \in \dot{V}_0$ that solves $\partial_t u + \mathcal{B}u =  (-\Delta_\omega)^{\frac{\beta}{2}} h$ in $\mathcal{D}'(\mathbb{R}, E_\infty)$.

Moreover, using the equation, we have $\partial_t u = -\Tilde{\mathcal{B}}u + (-\Delta_\omega)^{\frac{\beta}{2}} h$ in $\mathcal{D}'(\mathbb{R}, E_\infty)$, which implies that $u \in C_0(\mathbb{R}; L^2_\omega(\mathbb{R}^n))$ by Proposition \ref{prop: Lions non borné}. Even better, applying \cite[Theorem~4.16]{auscherbaadi2024fundamental}, we obtain $u \in \dot{V}_{-\min(\frac{2}{r},\beta)}$, and by (2) of Proposition \ref{prop:embeddings}, we have the embedding $\dot{V}_{-\min(\frac{2}{r},\beta)} \subset C_0(\mathbb{R}; L^2_\omega(\mathbb{R}^n)) \cap \dot{V}_0$.

Now, using the first embedding in (4) of Proposition \ref{prop:embeddings}, we have $u \in L^{\tilde{r}}(\mathbb{R}; D_{S, \tilde{\alpha}})^{{\mathrm{loc}}}$ for all $\tilde{r} \in [2,\infty)$ with $\tilde{\alpha} = \frac{2}{\tilde{r}}$. By taking $\tilde{r} = r$ and applying Lemma \ref{lem:xSobolev}, we conclude that $u \in \dot{\Sigma}^{r,q}(\mathbb{R})$. Furthermore, since $u \in C_0(\mathbb{R}; L^2_\omega(\mathbb{R}^n)) \subset L^1_{\mathrm{loc}}(\mathbb{R}; L^2_\omega(\mathbb{R}^n))$, it then follows from Proposition \ref{prop : passage au concret} that $\partial_t u + \mathcal{B}u =  (-\Delta_\omega)^{\frac{\beta}{2}} h$ in $\mathcal{D}'(\mathbb{R}^{1+n})$.

Finally, to prove \eqref{eq: estimates energy}, since by definition $u=\widetilde{\mathcal{H}}^{-1}( (-\Delta_\omega)^{\frac{\beta}{2}} h)$, we obtain
\begin{equation}\label{eq: I}
\|u \|_{\dot{V_0}}=\|\widetilde{\mathcal{H}}^{-1}( (-\Delta_\omega)^{\frac{\beta}{2}} h) \|_{\dot{V_0}}\le \mathfrak{I} \|(-\Delta_\omega)^{\frac{\beta}{2}} h \|_{\dot{V}_0^\star} \le \mathfrak{I}\times C(\rho) \, \| h \|_{L^{\rho'}(\mathbb{R};L^2_\omega(\mathbb{R}^n))}.
\end{equation}
In this last inequality, we used that $(-\Delta_\omega)^{\frac{\beta}{2}} h \in L^{\rho'}(\mathbb{R};D_{S,-\beta}) \hookrightarrow \dot{W}_{-\beta} \hookrightarrow \dot{V}_{-\beta}^\star \hookrightarrow \dot{V}_{0}^\star$, which follows from Points (3) and (1) of Proposition \ref{prop:embeddings}. This yields the desired estimate on $\nabla_x u$ as
\begin{equation*}
\max(\|D_t^{1/2} u \|_{L^2(\mathbb{R};L^2_\omega(\mathbb{R}^n))},\| \nabla_x u \|_{L^2(\mathbb{R};L^2_\omega(\mathbb{R}^n))}) \le \|u \|_{\dot{V_0}}.
\end{equation*}
To prove the $L^\infty$-norm estimate, we use the energy identity \eqref{eq:integralidentity} in Proposition \ref{prop: Lions non borné} to deduce
\begin{align*}
\| u \|^2_{C_0(\mathbb{R};L^2_\omega(\mathbb{R}^n))} \le &2M \| \nabla_x u \|^2_{L^2(\mathbb{R};L^2_\omega(\mathbb{R}^n))}+2CP_{r,q} \|u \|^2_{\dot{V_0}}\\&+2 \| (-\Delta_\omega)^{\beta/2} u \|_{L^{\rho}(\mathbb{R};L^2_\omega(\mathbb{R}^n))}  \| h \|_{L^{\rho'}(\mathbb{R};L^2_\omega(\mathbb{R}^n))},
\end{align*}
where $C = C([\omega]_{A_2}, [\omega]_{RH_{\frac{q}{2}}}, n, q)$ is the constant from Corollary \ref{cor: borne Beta}. Combining \eqref{eq: I} with the second embedding in (4) of Proposition~\ref{prop:embeddings} for the norm $\| (-\Delta_\omega)^{\beta/2} u \|_{L^{\rho}(\mathbb{R};L^2_\omega(\mathbb{R}^n))}$, one then obtains the desired estimate.
\end{proof}

By the assumption \textbf{(I)}, $\widetilde{\mathcal{H}}: \dot{V}_0 \rightarrow \dot{V}_0^\star$ is injective, and hence so is $\mathcal{H}: \dot{V}_0^\mathrm{loc} \rightarrow \dot{V}_0^\star$. Moreover, by the above result, we have
\begin{equation*}
    \widetilde{\mathcal{H}}^{-1} : \bigcup_{2 \le \rho<\infty} L^{\rho'}(\mathbb{R};D_{S,-\beta}) \rightarrow C_0(\mathbb{R};L^2_\omega(\mathbb{R}^n)) \cap \dot{V}_0 \subset \dot{V}_0 ^\mathrm{loc}.
\end{equation*}
As $\widetilde{\mathcal{H}}$ is the extension of $\mathcal{H}$ to $\dot{V}_0$, then $\mathcal{H}^{-1} : \bigcup_{2 \le \rho<\infty} L^{\rho'}(\mathbb{R};D_{S,-\beta}) \rightarrow C_0(\mathbb{R};L^2_\omega(\mathbb{R}^n)) \cap \dot{V}_0 \subset \dot{V}_0 ^\mathrm{loc}$. We stress that whenever we write $u = \mathcal{H}^{-1} f$, for $f \in \bigcup_{2 \le \rho < \infty} L^{\rho'}(\mathbb{R}; D_{S,-\beta})$, this means that $u$ is the unique element of $\dot{V}_0^{\mathrm{loc}}$ satisfying $\mathcal{H}u = f$. We then know \textit{a posteriori} that 
$u \in C_0(\mathbb{R}; L^2_\omega(\mathbb{R}^n))$, etc. Providing such a candidate $u$ was the sole purpose of introducing $\widetilde{\mathcal{H}}$. From this point onward, we work directly with $\mathcal{H}$. We could do the same with sources in $\dot{W}_\alpha$ for $-1 \le \alpha < 0$, but we refrain from doing so: these spaces are defined on the whole real line via the Fourier transform, and we will later need to restrict to intervals.

\begin{cor}[Boundedness properties of $\mathcal{H}^{-1}  $]\label{cor: boudedness H}
Fix $\rho \in [2,\infty)$ and set $\beta={2}/{\rho}\in (0,1]$. Then, 
$$\mathcal{H}^{-1} : L^{\rho'}(\mathbb{R};D_{S,-\beta}) \rightarrow \dot{V}_{-\min(\frac{2}{r},\beta)} \cap C_0(\mathbb{R};L^2_\omega(\mathbb{R}^n)) \ (\subset \dot{V}_0^{\mathrm{loc}}) $$
is  bounded. The same holds for $(\mathcal{H}^\star)^{-1}$.
\end{cor}
\begin{proof}
Refer to Proposition~\ref{prop: existence OK} above and its proof. We stress that this does not imply that $\mathcal{H}$ maps $\dot{V}_{-\min(\frac{2}{r},\beta)}$ into $L^{\rho'}(\mathbb{R}; D_{S,-\beta})$, since $\mathcal{H}^{-1}$, as mapped above, is not necessarily surjective.
\end{proof}

\begin{cor}
    Let $(\Tilde{r}, \Tilde{q})$ be a pair with $r\le \Tilde{r}<\infty$ and $\frac{1}{\Tilde{r}} + \frac{n}{2\Tilde{q}} = \frac{n}{4}$. Let $g \in L^{\Tilde{r}'}(\mathbb{R};L^{\Tilde{q}'}_{\omega^{\Tilde{q}'/2}}(\mathbb{R}^n))$. Then, there exists a unique $u \in \dot{\Sigma}^{r,q}(\mathbb{R})$ solution of
    \begin{equation*}
    \partial_t u + \mathcal{B}u = g \quad \text{in} \ \mathcal{D}'(\mathbb{R}^{1+n}).
    \end{equation*}
    Moreover, $u \in C_0(\mathbb{R};L^2_\omega(\mathbb{R}^n)) \cap \dot{V}_0$, and there exists a constant $C = C(M,\mathfrak{I}, P_{r,q}, [\omega]_{A_2}, [\omega]_{RH_{\frac{q}{2}}}, n, q,\Tilde{q})$ such that
    \begin{equation*}
    \sup_{t \in \mathbb{R}} \| u(t) \|_{L^2_\omega(\mathbb{R}^n)} + \| \nabla_x u \|_{L^2(\mathbb{R};L^2_\omega(\mathbb{R}^n))} 
    \leq C \| g \|_{L^{\Tilde{r}'}(\mathbb{R};L^{\Tilde{q}'}_{\omega^{\Tilde{q}'/2}}(\mathbb{R}^n))}.
    \end{equation*}
\end{cor}
\begin{proof}
The pair $(\Tilde{r}, \Tilde{q})$ is admissible; see Point (2) of Remarks \ref{rems: RH}. The result then follows directly by combining Lemma \ref{lem:xSobolev} (the duality result case $I = \mathbb{R}$) with Proposition \ref{prop: existence OK} applied with $\rho = \Tilde{r}$.
\end{proof}

\begin{rem}
Since we have $\omega \in \bigcup_{q_0>q} RH_{\frac{q_0}{2}}(\mathbb{R}^n)$ by Point (2) of Proposition \ref{prop:weights}, when $r>2$ we may take $g \in L^{\tilde{r}'}(\mathbb{R}; L^{\tilde{q}'}_{\omega^{\tilde{q}'/2}}(\mathbb{R}^n))$ in the above corollary for any admissible pair $(\Tilde{r}, \Tilde{q})$ with $\tilde{r} \in [r_\omega, r)$, for some $r_\omega \in [2,r)$.
\end{rem}

\subsubsection{Source term in $L^1(\mathbb{R};L^2_\omega(\mathbb{R}^n))$} The assumption \textbf{(I)} cannot be straightforwardly applied to prove existence for a source term in $L^1(\mathbb{R}; L^2_\omega(\mathbb{R}^n))$, since $L^1(\mathbb{R}; L^2_\omega(\mathbb{R}^n)) \nsubseteq \dot{V}_0^\star$. However, it is still possible to handle such source terms by employing a duality scheme.

\begin{prop}[Source term in $L^1(\mathbb{R};L^2_\omega(\mathbb{R}^n))$]\label{thm: existence L1}
    Let $f\in L^1(\mathbb{R};L^2_\omega(\mathbb{R}^n))$. Then, there exists a unique $u \in \dot{\Sigma}^{r,q}(\mathbb{R})$ solution of 
    \begin{equation*}
    \partial_t u + \mathcal{B}u = f \quad \text{in} \ \mathcal{D}'(\mathbb{R}^{1+n}).
    \end{equation*}
    Moreover, $u \in C_0(\mathbb{R};L^2_\omega(\mathbb{R}^n))$, and there exists a constant $C = C(M,\mathfrak{I}, P_{r,q}, [\omega]_{A_2}, [\omega]_{RH_{\frac{q}{2}}}, n, q) $ such that
    \begin{equation*}
    \sup_{t \in \mathbb{R}} \| u(t) \|_{L^2_\omega(\mathbb{R}^n)} + \| \nabla_x u \|_{L^2(\mathbb{R};L^2_\omega(\mathbb{R}^n))} 
    \leq C \| f \|_{L^1(\mathbb{R};L^2_\omega(\mathbb{R}^n))}.
    \end{equation*}
\end{prop}
\begin{proof}
    Uniqueness is provided by Proposition \ref{thm: Uniqueness}. To prove the existence, we remark that  Corollary \ref{cor: boudedness H} for the backward operator $\mathcal{H}^\star$ in the case  $\rho=2$ implies that $(\mathcal{H}^\star)^{-1}$ is bounded from $L^2(\mathbb{R};D_{S,-1})$ into $C_0(\mathbb{R}; L^2_\omega(\mathbb{R}^n))$. We define 
\begin{align*}
    \mathcal{T}: L^1(\mathbb{R};L^2_\omega(\mathbb{R}^n)) \rightarrow \mathcal{D'}(\mathbb{R};E_\infty), \   \llangle \mathcal{T}f, \varphi  \rrangle:= \llangle f, (\mathcal{H}^\star)^{-1}\varphi  \rrangle_{L^1(\mathbb{R};L^2_\omega(\mathbb{R}^n)),L^\infty(\mathbb{R};L^2_\omega(\mathbb{R}^n))}.
\end{align*}
We have, for $C=C(M,\mathfrak{I}, P_{r,q}, [\omega]_{A_2}, [\omega]_{RH_{\frac{q}{2}}}, n, q)$ a constant,
\begin{equation}\label{1}
   \left \| \mathcal{T}f \right \|_{L^2(\mathbb{R};D_{S,1})}
   \leq C \left \| f \right \|_{L^1(\mathbb{R};L^2_\omega(\mathbb{R}^n))}.
\end{equation}
 Next, let $f$ in $\mathcal{D}(\mathbb{R};E_{-\infty})$. We write for all $\varphi \in \mathcal{D}(\mathbb{R};E_{-\infty})$, observing that $\mathcal{D}(\mathbb{R};E_{-\infty})\subset \dot{V}_0^\star$,
\begin{align*}
   \llangle \mathcal{T}f, \varphi  \rrangle = \llangle \mathcal{H}\mathcal{H}^{-1}f, (\mathcal{H}^\star)^{-1}\varphi  \rrangle_{\dot{V}_0^\star, \dot{V}_0}= \llangle \mathcal{H}^{-1}f, \varphi  \rrangle.
\end{align*}
Hence, $\mathcal{T}f=\mathcal{H}^{-1}f \in C_0(\mathbb{R}; L^2_\omega(\mathbb{R}^n))$ and
\begin{equation}\label{Eq}
    \forall \varphi \in \mathcal{D}(\mathbb{R};E_{-\infty}), \quad -\llangle \mathcal{T}f, \partial_t \varphi \rrangle + \llangle \mathcal{B}(\mathcal{T}f) , \varphi\rrangle = \int_{\mathbb{R}} \langle f(t), \varphi(t) \rangle_{2,\omega} \, \mathrm d t.
\end{equation} 
Since $\partial_t \mathcal{T}f = -\mathcal{B}(\mathcal{T}f) + f$ in $\mathcal{D}'(\mathbb{R};E_\infty)$, the energy equality \eqref{eq:integralidentity} in Proposition \ref{prop: Lions non borné} implies that
\begin{align*}
    \| \mathcal{T}f \|^2_{L^\infty(\mathbb{R};L^2_\omega(\mathbb{R}^n))} \leq 2M \| \mathcal{T}f \|^2_{L^2(\mathbb{R};D_{S,1})}+2P_{r,q} \| \mathcal{T}f \|^2_{\dot{\Sigma}^{r,q}(\mathbb{R})}+2  \| \mathcal{T}f \|_{L^\infty(\mathbb{R};L^2_\omega(\mathbb{R}^n))}  \| f \|_{L^1(\mathbb{R};L^2_\omega(\mathbb{R}^n))}.
\end{align*}
Next, using the inequality $ \| \mathcal{T}f \|^2_{\dot{\Sigma}^{r,q}(\mathbb{R})} \le 2 \big( \| \mathcal{T}f \|^2_{L^2(\mathbb{R};D_{S,1})} + \| \mathcal{T}f \|^2_{L^r(\mathbb{R};L^q_{\omega^{q/2}}(\mathbb{R}^n))} \big) $, and combining Lemma \ref{lem:xSobolev} with the first embedding in (4) of Proposition \ref{prop:embeddings} for this second term, together with basic inequalities of the type $ab \le \varepsilon \frac{a^2}{2} + \frac{1}{\varepsilon} \frac{b^2}{2}$, we obtain, via this absorption argument, the following estimate holds, with $C=C(M,\mathfrak{I}, P_{r,q}, [\omega]_{A_2}, [\omega]_{RH_{\frac{q}{2}}}, n, q)$ a constant:
\begin{equation}\label{2}
    \| \mathcal{T}f \|_{C_0(\mathbb{R};L^2_\omega(\mathbb{R}^n))} \le C \| f \|_{L^1(\mathbb{R};L^2_\omega(\mathbb{R}^n))}.
\end{equation}

Now, let us pick $f \in L^1(\mathbb{R};L^2_\omega(\mathbb{R}^n)).$ Let $(f_k)_{k \in \mathbb{N}} \in \mathcal{D}(\mathbb{R};E_{-\infty})^{\mathbb{N}}$ such that $f_k \to f$ in $L^1(\mathbb{R};L^2_\omega(\mathbb{R}^n))$. By \eqref{1} and \eqref{2}, we have $\mathcal{T} f_k \to \mathcal{T} f$ in $L^2(\mathbb{R}; D_{S,1}) \cap C_0(\mathbb{R}; L^2_\omega(\mathbb{R}^n))$, and consequently also in $L^r(\mathbb{R}; L^q_{\omega^{q/2}}(\mathbb{R}^n))$ and in $\dot{\Sigma}^{r,q}(\mathbb{R})$ by combining Lemma \ref{lem:xSobolev} with the first embedding in (4) of Proposition \ref{prop:embeddings}. Using \eqref{Eq} with $\mathcal{T}f_k$ for a fixed $\varphi \in \mathcal{D}(\mathbb{R};E_{-\infty})$ and letting $k \to \infty$ imply that $\partial_t (\mathcal{T}f) + \mathcal{B}(\mathcal{T}f) = f$ in $\mathcal{D}'(\mathbb{R};E_{\infty})$. Furthermore, since $\mathcal{T}f \in C_0(\mathbb{R}; L^2_\omega(\mathbb{R}^n)) \subset L^1_{\mathrm{loc}}(\mathbb{R}; L^2_\omega(\mathbb{R}^n))$, it then follows from Proposition \ref{prop : passage au concret} that $\mathcal{T}f$ is asolution of $\partial_t (\mathcal{T}f) + \mathcal{B}(\mathcal{T}f) = f$ in $\mathcal{D}'(\mathbb{R}^{1+n})$. 
    
\end{proof}

\subsubsection{Source term is a Dirac measure on $L^2_\omega(\mathbb{R}^n)$}

For any $s \in \mathbb{R}$ and $\psi \in L^2_\omega(\mathbb{R}^n)$, we denote by $\delta_s \otimes \psi$ the Dirac measure on $s$ carried by $\psi$ which is defined by 
\begin{equation*}
    \llangle \delta_s \otimes \psi,  \phi \rrangle =  \langle \psi,  \phi(s) \rangle_{2,\omega}, \quad \text{for all} \ \phi \in C_0(\mathbb{R};L^2_\omega(\mathbb{R}^n)).
\end{equation*}

\begin{prop}[Source term is a Dirac measure]\label{cor: CorRadon}
Assume that \textbf{(C)} holds. Let $s \in \mathbb{R}$ and $\psi \in L^2_\omega(\mathbb{R}^n)$. Then there exists a unique $u \in \dot{\Sigma}^{r,q}(\mathbb{R})$ solution of
\begin{align*}
\partial_t u + \mathcal{B}u = \delta_s \otimes \psi \quad \text{in} \ \mathcal{D}'(\mathbb{R}^{1+n}).
\end{align*}
Moreover, $u \in C_0(\mathbb{R} \setminus \left \{ s \right \}; L^2_\omega(\mathbb{R}^n))$, equals  $0$ on $(-\infty,s)$ and $\lim_{t \to s^+} u(t)=\psi$ in $L^2_\omega(\mathbb{R}^n))$ and the map $t \mapsto \|u(t)\|_{2,\omega}^2$ admits an absolutely continuous extension to $[s, +\infty[$. Furthermore, there exists a constant $C = C(M,\mathfrak{I}, P_{r,q}, [\omega]_{A_2}, [\omega]_{RH_{\frac{q}{2}}}, n, q) $ such that 
\begin{equation*}
\sup_{s\le t}\left \|u(t)  \right \|_{L^2_\omega(\mathbb{R}^n)} + \left \| \nabla_x u \right \|_{L^2((s,\infty);L^2_\omega(\mathbb{R}^n))} \leq C \left \| \psi \right \|_{L^2_\omega(\mathbb{R}^n)}.
\end{equation*}
\end{prop}
\begin{proof}
This follows directly from \cite[Corollary 6.19]{auscherbaadi2024fundamental}, now that we can solve for source terms in $L^1(\mathbb{R};L^2_\omega(\mathbb{R}^n))$. 
\end{proof}

\begin{rems}
\begin{enumerate}
\item When formulating Proposition \ref{cor: CorRadon} with the backward equation
\begin{equation*}
    -\partial_s \Tilde{u} + \mathcal{B}^\star \Tilde{u} = \delta_t \otimes \psi, \quad t \in \mathbb{R},
\end{equation*}
the same result holds, with $\Tilde{u} \in C_0(\mathbb{R} \setminus \{t\}; L^2_\omega(\mathbb{R}^n))$, vanishing on $(t,+\infty)$ and satisfying $\lim_{s \to t^-} \Tilde{u}(s) = \psi$.
\item There is a statement for a source term that is any bounded $L^2_\omega(\mathbb{R}^n)$-valued measure on $\mathbb{R}$, defined as an element of the topological anti-dual of $C_0(\mathbb{R};L^2_\omega(\mathbb{R}^n))$ with respect to the sup norm. In this case, $u$ is only in $L^\infty(\mathbb{R};L^2_\omega(\mathbb{R}^n))$. See \cite[Proposition~6.16]{auscherbaadi2024fundamental}.
\item If we dont not assume that \textbf{(C)} holds, the result above can still be proved but $u$ is not necessarily equal to $0$ on $(-\infty,s)$ and has limits at $s^\pm$. See \cite{auscheregert2023universal}.

\end{enumerate}
\end{rems}

\begin{rem}[Dependency of constants]
When causality and invertibility are of the perturbative case (Theorems \ref{thm: Invertibility} and \ref{thm: Causality}), all the constants $C = C(\mathfrak{I}, P_{r,q}, \dots)$, as in Propositions \ref{prop: existence OK}, \ref{thm: existence L1}, and \ref{cor: CorRadon}, depend only on $(M,\nu, [\omega]_{A_2}, [\omega]_{RH_{\frac{q}{2}}}, n, q,\dots)$. When a lower bound $c$ is available, as in Theorem \ref{thm : Causality and invertibility}, all the constants $C = C(\mathfrak{I}, \dots)$ depend on $(c,M, [\omega]_{A_2}, [\omega]_{RH_{\frac{q}{2}}}, n, q,P_{r,q},\dots)$.
\end{rem}

\subsection{The Fundamental solution}\label{sub: FS}

The \emph{fundamental solution} is defined as the family of operators that represent the inverse of the degenerate parabolic operator $\partial_t + \mathcal{B}$ on $\mathbb{R} \times \mathbb{R}^n$.
\begin{defn}[Fundamental solution for $\partial_t +\mathcal{B}$]\label{def: FS}
A fundamental solution for $\partial_t +\mathcal{B}$ is a family $\Gamma=(\Gamma(t,s))_{t,s \in \mathbb{R}}$ of bounded operators on $L^2_\omega(\mathbb{R}^n)$ such that :
\begin{enumerate}
        \item (Boundedness) $\sup_{t,s \in \mathbb{R}} \left\|\Gamma(t,s) \right\|_{\mathcal{L}(L^2_\omega(\mathbb{R}^n))} < +\infty.$
        \item (Causality) $\Gamma(t,s)=0$ if $t<s$.
        \item (Weak measurability) For all $\psi,\Tilde{\psi}\in \mathcal{D}(\mathbb{R}^n)$, the function $(t,s) \mapsto \langle \Gamma(t,s)\psi, \Tilde{\psi} \rangle_{2,\omega}$ is Borel measurable on $\mathbb{R}^2$.
        \item (Representation) For all $\varphi \in \mathcal{D}(\mathbb{R})$ and $\psi \in \mathcal{D}(\mathbb{R}^n)$, any $u \in \dot{\Sigma}^{r,q}(\mathbb{R})$ solution of the equation $\partial_t u +\mathcal{B}u = \varphi \otimes \psi $ in $\mathcal{D}'(\mathbb{R}\times\mathbb{R}^{n})$ satisfies $\langle u(t), \Tilde{\psi} \rangle_{2,\omega} = \int_{\mathbb{R}} \varphi(s) \langle \Gamma (t,s)\psi, \Tilde{\psi} \rangle_{2,\omega} \, \mathrm{d}s$, for all $\Tilde{\psi} \in \mathcal{D}(\mathbb{R}^n)$ and for almost every $t\in \mathbb{R}$. We have set $(\varphi \otimes \psi)(t,x)=\varphi(t) \psi(x)$.
\end{enumerate}
In the same manner, one defines a fundamental solution $(\Tilde{\Gamma}(s,t))_{s,t \in \mathbb{R}}$ to the backward operator $-\partial_t +\mathcal{B}^\star$.
\end{defn}

\begin{thm}[Existence, uniqueness, and properties of the fundamental solution]\label{thm: representation} Assume that \textbf{(C)} holds. There exists a unique fundamental solution $\Gamma$ (up to equality almost everywhere) to $\partial_t + \mathcal{B}$ in the sense of the above definition, and it satisfies the following properties.
\begin{enumerate}
    \item (Estimates for the fundamental solution of $\partial_t + \mathcal{B}$)
    For all $s \in \mathbb{R}$ and $\psi \in L^2_\omega(\mathbb{R}^n)$, we have 
    $\Gamma(\cdot,s)\psi \in C_0([s,\infty);L^2_\omega(\mathbb{R}^n))$ with $ \Gamma(s,s)\psi =\psi $, and there exists a constant $C = C(M,\mathfrak{I}, P_{r,q}, [\omega]_{A_2}, [\omega]_{RH_{\frac{q}{2}}}, n, q) $ such that
    \begin{equation*}
        \sup_{s\le t } \left\| \Gamma(t,s) \right\|_{\mathcal{L}(L^2_\omega(\mathbb{R}^n))} \leq C 
        \quad \text{and} \quad 
        \left\| \nabla_x \Gamma(\cdot,s)\psi \right\|_{L^2((s,\infty); L^2_\omega(\mathbb{R}^n))} 
        \leq C \left\| \psi \right\|_{2,\omega}.
    \end{equation*}
    \item (Estimates for the fundamental solution of $-\partial_t + \mathcal{B}^\star$) 
    For all $t \in \mathbb{R}$ and $\Tilde{\psi} \in L^2_\omega(\mathbb{R}^n)$, we have 
    $\Tilde{\Gamma}(\cdot,t)\Tilde{\psi} \in C_0((-\infty,t];L^2_\omega(\mathbb{R}^n))$ with $\Tilde{\Gamma}(t,t)\psi =\psi$, and there exists a constant $C = C(M,\mathfrak{I}, P_{r,q}, [\omega]_{A_2}, [\omega]_{RH_{\frac{q}{2}}}, n, q)$ such that
    \begin{equation*}
        \sup_{s \le t} \| \Tilde{\Gamma}(s,t) \|_{\mathcal{L}(L^2_\omega(\mathbb{R}^n))} \leq C 
        \quad \text{and} \quad 
        \| \nabla_x \Tilde{\Gamma}(\cdot,t)\Tilde{\psi} \|_{L^2((-\infty,t); L^2_\omega(\mathbb{R}^n))} 
        \leq C\| \Tilde{\psi} \|_{2,\omega}.
    \end{equation*}
    \item (Full representation) Let $\rho \in [2,\infty]$ and set $\beta = \frac{2}{\rho}$. Let $s\in \mathbb{R}$ and $\psi \in L^2_\omega(\mathbb{R}^n)$. Let $h \in L^{\rho'}(\mathbb{R};L^2_\omega(\mathbb{R}^n))$. Then the unique solution $u \in \dot{\Sigma}^{r,q}(\mathbb{R})$ of the equation 
    \begin{equation*}
    \partial_t u +\mathcal{B}u= \delta_s \otimes \psi + (-\Delta_\omega)^{\frac{\beta}{2}} h  \ \ \mathrm{in} \ \mathcal{D}'(\mathbb{R}^{1+n})
    \end{equation*}
    obtained by  combining Proposition \ref{prop: existence OK}, Proposition \ref{thm: existence L1} and Proposition \ref{cor: CorRadon} can be represented pointwisely by the equation 
    \begin{equation*}
    u(t) = \Gamma(t,s)\psi +\int_{-\infty}^{t} \Gamma(t,\tau) ((-\Delta_\omega)^{\frac{\beta}{2}} h(\tau))\, \mathrm d \tau,
    \end{equation*}
    where the integral is weakly defined in $L^2_\omega(\mathbb{R}^n)$ and also strongly defined in the Bochner sense when $\rho = \infty$. More precisely, for all $t\in \mathbb{R}$ and $\Tilde{\psi}\in L^2_\omega(\mathbb{R}^n)$, we have the equality with absolutely converging integral
    \begin{align*}
    \langle u(t), \Tilde{\psi} \rangle_{2,\omega} = \langle \Gamma(t,s)\psi, \Tilde{\psi} \rangle_{2,\omega}&+ \int_{-\infty}^{t} \langle h(\tau), (-\Delta_\omega)^{\frac{\beta}{2}}\Tilde{\Gamma}(\tau,t)\Tilde{\psi} \rangle_{2,\omega} \, \mathrm d \tau .
    \end{align*}
    \item (Adjointness property) For all $s < t$, $\Gamma(t,s)$ and $\Tilde{\Gamma}(s,t)$ are adjoint operators.
    \item (Chapman–Kolmogorov identities) For all $s <r< t$, we have $\Gamma(t,s) = \Gamma(t,r) \Gamma(r,s)$.
\end{enumerate}
\end{thm}
\begin{proof}
The proof of uniqueness proceeds exactly as in \cite[Lemma~6.24]{auscherbaadi2024fundamental}, with $E_{-\infty}$ replaced by $\mathcal{D}(\mathbb{R}^n)$. For existence, one defines the Green operator using Proposition \ref{cor: CorRadon}, as done by J.-L. Lions \cite{lions2013equations}. One then proves that it is the fundamental solution and that it satisfies all the listed properties. See \cite{auscheregert2023universal,auscherbaadi2024fundamental} for details. 
\end{proof}
\begin{rem}
Fundamental solutions can be defined without the causality assumption (2) in Definition \ref{def: FS}, and the result above remain valid with appropriate modifications (e.g., replacing $(-\infty,t)$ by $\mathbb{R}$, etc.). We do not pursue this here, as the results are identical to those in the unweighted case \cite{auscheregert2023universal}, and causality is essential for Cauchy problems.
\end{rem}

\subsection{The Cauchy problem on \texorpdfstring{$(0,\infty)$}{tempsinfini}}
In this section, we consider the Cauchy problem on the half-line, namely $(0, \infty)$. The main result of this section is the following.

\begin{thm}[The Cauchy problem on $(0, \infty)$]\label{thm: Pb Cauchy homogène}
    Assume that, in addition to the invertibility assumption \textbf{(I)}, the causality assumption \textbf{(C)} also holds. Let $\rho \in [2, \infty]$, and define $\beta = \frac{2}{\rho} \in [0, 1]$. Let $h \in L^{\rho'}((0, \infty);L^2_\omega(\mathbb{R}^n))$ and $\psi \in L^2_\omega(\mathbb{R}^n)$. Then there exists a unique $u \in \dot{\Sigma}^{r,q}((0,\infty))$ solution to the Cauchy problem 
    \begin{align}\label{eq: Pb Cauchy homogène}
    \left\{
    \begin{array}{ll}
        \partial_t u + \mathcal{B}u  =  (-\Delta_\omega)^{\beta/2}h \quad \mathrm{in} \  \mathcal{D}'((0,\infty)\times \mathbb{R}^n), \\
        u(t) \rightarrow \psi \  \mathrm{ in } \ \mathcal{D'}(\mathbb{R}^n) \ \mathrm{as} \ t \rightarrow 0^+.
    \end{array}\right.
    \end{align} 
    Moreover, $u\in C_0([0,\infty);L^2_\omega(\mathbb{R}^n))$, with $u(0)=\psi$, $t \mapsto \| u(t)  \|^2_{2,\omega}$ is absolutely continuous on $[0,\infty)$ and energy equalities hold. Furthermore, there exists a constant $C = C(M,\mathfrak{I}, P_{r,q}, [\omega]_{A_2}, [\omega]_{RH_{\frac{q}{2}}}, n, q)$ such that
    \begin{align*}
            \sup_{t \ge 0} \| u(t) \|_{2,\omega}+ \| \nabla_x u\|_{L^2((0,\infty);L^2_\omega(\mathbb{R}^n))} 
            \leq C  ( \left \| h \right \|_{L^{\rho'}((0,\infty);L^2_\omega(\mathbb{R}^n))}+  \| \psi  \|_{2,\omega}  ).
    \end{align*} 
    Lastly,  for all $t \ge 0$, we have the following representation of $u$ : 
    \begin{align*}
    u(t) = \Gamma(t,0)\psi  + \int_{0}^{t} \Gamma(t,\tau) (-\Delta_\omega)^{\beta/2}h(\tau)\ \mathrm d \tau,
    \end{align*}
    where the integral is weakly defined in $L^2_\omega(\mathbb{R}^n)$ and strongly defined in the Bochner sense when $\rho = \infty$, as described in Theorem \ref{thm: representation}, Point (3).
\end{thm}
\begin{proof}
We start with the existence of such a solution. We first extend $h$ by $0$ on $(-\infty,0]$ and keep the same notation for the extension. By  combining Proposition \ref{prop: existence OK}, Proposition \ref{thm: existence L1} and Proposition \ref{cor: CorRadon}, there exists a unique $\Tilde{u}\in \dot{\Sigma}^{r,q}(\mathbb{R})$ solution of the equation 
\begin{equation*}
\partial_t \Tilde{u} +\mathcal{B} \Tilde{u}= \delta_0 \otimes \psi + (-\Delta_\omega)^{\beta/2} h  \ \ \mathrm{in} \ \mathcal{D}'(\mathbb{R}^{1+n}).
\end{equation*}
Moreover, by the causality assumption \textbf{(C)}, $\Tilde{u} \in C_0([0,\infty);L^2_\omega(\mathbb{R}^n))$, $\Tilde{u}=0$ on $(-\infty,0)$ and $\Tilde{u}(0)=\psi$. The candidate $u := \Tilde{u}_{\scriptscriptstyle{\vert (0,\infty)}}$ satisfies all the required properties of the theorem, thereby proving existence. As for the representation, it follows from that of $\Tilde{u}(t)$ for all $t \in \mathbb{R}$, using the fundamental solution defined on $\mathbb{R}$. See Theorem \ref{thm: representation}, Point (3).

Next, we check uniqueness in the space $\dot{\Sigma}^{r,q}((0,\infty))$. If $u \in \dot{\Sigma}^{r,q}((0,\infty))$ is a solution to \eqref{eq: Pb Cauchy homogène} with $h=0$ and $\psi=0$, then, by Proposition \ref{prop: Lions non borné}, we have $u \in C_0([0, \infty); L^2_\omega(\mathbb{R}^n))$ with $u(0) = 0$. We extend $u$ by zero on $(-\infty, 0) \times \mathbb{R}^n$ to obtain $u \in \dot{\Sigma}^{r,q}(\mathbb{R})$, and use its $L^2_\omega(\mathbb{R}^n)$-valued continuity on $[0, \infty)$ to deduce easily that $u$ is a solution to $\partial_t u + \mathcal{B}u = 0$ in $\mathcal{D}'(\mathbb{R}^{1+n})$. Therefore, by Proposition \ref{thm: Uniqueness}, we have $u = 0$.

\end{proof}

\begin{rem}
Theorem \ref{thm: intro} is a particular case of Theorem \ref{thm: Pb Cauchy homogène}, when the invertibility and causality assumptions are those of the perturbative case, as given by Theorems \ref{thm: Invertibility} and \ref{thm: Causality}, respectively.
\end{rem}
\begin{rem}
In this section on Cauchy problems, one may define the coefficients of the elliptic part $\mathcal{B}$, namely $A$, $a$, $b$, and $c$, on $(0, \infty) \times \mathbb{R}^n$. We then extend $A$ by $I_n$ and $a$, $b$, and $c$ by $0$ on $(-\infty, 0] \times \mathbb{R}^n$; that is, $B_t(u,v) = \langle \nabla_x u, \nabla_x v \rangle_{2, \omega}$ for all $t \le 0$. In the same manner, the definition, existence, and uniqueness of the fundamental solution on $(0, \infty)$ follow \textit{verbatim} as in Section \ref{sub: FS}.
\end{rem}
\begin{rem}[The Cauchy problem on segments]
By restricting from the case $(0, \infty)$, one obtains the well-posedness of Cauchy problems on bounded intervals, namely $(0, \mathfrak{T})$, for any final time $\mathfrak{T} > 0$. The only difference in this case is that the uniqueness class becomes $\dot{\Sigma}^{r,q}((0, \mathfrak{T})) \cap L^1((0, \mathfrak{T}); L^2_\omega(\mathbb{R}^n))$, which is not clear to obtain in the presence of bounded first-order terms, as we will see.
\end{rem}
\begin{rem}[Case of purely second-order degenerate elliptic part]
In the absence of lower-order terms, this variational framework is used in \cite{baadi2025degenerate} to prove well-posedness of the Cauchy problem. Assuming moreover that $\omega \in RH_{\frac{q}{2}}(\mathbb{R}^n)$ for some $q>2$ (which holds for any $q \in [2,q_\omega]$ for some $q_\omega>2$ by Point (2) of Proposition \ref{prop:weights}), the notion of $(r,q)$-admissibility allows to treat source terms in mixed Lebesgue spaces and to derive corresponding mixed-norm estimates.
\end{rem}

\section{Lower-order coefficients in mixed Lorentz spaces}\label{section 4}

Let $\mu$ be a measure on $\mathbb{R}^n$. The Lorentz space $L^{p,r}_{\mu}(\mathbb{R}^n)$ is defined as the set of all measurable functions $f$ on $\mathbb{R}^n$ such that $\left\| f\right\|_{L^{p,r}_{\mu}(\mathbb{R}^n)} < \infty$, with
\begin{align*}
\left\| f\right\|_{L^{p,r}_{\mu}(\mathbb{R}^n)}  =\left\{\begin{matrix}
\left ( \frac{r}{p} \int_{0}^{\infty} \left ( t^{1/p} f^\star (t) \right )^r \, \frac{\mathrm{d}t}{t} \right )^{1/r}, \quad \ 1\leq p,r <\infty , \\ 
\sup_{t>0} t^{1/p} f^\star(t), \quad \ 1\le p \le \infty, \ r=\infty,
\\ 
\end{matrix}\right. 
\end{align*}
or equivalently (see \cite[Proposition 1.4.9]{GrafakosBook2014})
\begin{align*}
\left\| f\right\|_{L^{p,r}_{\mu}(\mathbb{R}^n)} =\left\{\begin{matrix}
\left (r \int_{0}^{\infty} s^{r-1} \lambda_f(s)^{r/p} \, \mathrm{d}s \right )^{1/r}, \quad \ 1\leq p,r <\infty , \\ 
 \sup_{s>0} s \, \lambda_f(s)^{1/p}, \quad \ 1\le p < \infty, \ r=\infty,
 \\ 
 \end{matrix}\right. 
 \end{align*}
 where 
$$ \lambda_f(s)=\mu\left ( \left\{ x\in \mathbb{R}^n : |f(x)|>s\right\} \right ) \ \ \text{and} \ \ f^\star(t)=\inf \left\{s>0 : \lambda _f(s)\le t\right\} 
. 
$$
It is known that $L^{p,p}_{\mu}(\mathbb{R}^n)=L^{p}_{\mu}(\mathbb{R}^n)$ with $\left\| f\right\|_{L^{p,p}_{\mu}(\mathbb{R}^n)}=\left\| f\right\|_{L^{p}_{\mu}(\mathbb{R}^n)}$ and $\left\| f\right\|_{L^{p,r}_{\mu}(\mathbb{R}^n)}$ is non-increasing as a function of $r$, so that $L^{p,r_1}_{\mu}(\mathbb{R}^n) \subset L^{p,r_2}_{\mu}(\mathbb{R}^n)$ if $r_1\leq r_2$. We recall that H\"older's inequality holds in Lorentz spaces: the product of two functions in $L_{\mu}^{p_i,r_i}$ belongs to $L^1_{\mu}$ if $\frac{1}{p_1} + \frac{1}{p_2} = 1$ and $\frac{1}{r_1} + \frac{1}{r_2} = 1$. We refer to \cite[Ch. V, \S3 ]{stein1971introduction} for more details.

\subsection{Embeddings}

We begin by stating the following result, which refines certain abstract embeddings from Proposition \ref{prop:embeddings}.

\begin{prop}\label{prop: embeddings Lorentz} We have the following embeddings.
    \begin{enumerate}
        \item \textbf{(Hardy-Littlewood-Sobolev embedding)} Let $\alpha \in (0,1]$ and let $r=\tfrac{2}{\alpha} \in [2,\infty)$. Then,  we have $\dot{V}_{\alpha} \hookrightarrow L^{r,2}(\mathbb{R};D_{S,\alpha}) $ and there is a constant $C=C(r)$ such that for all $u \in \dot{V}_{\alpha}$,  
        \begin{equation*} 
        \left\| u \right\|_{L^{r,2}(\mathbb{R};D_{S,\alpha})} \leq C \|D_t^{\frac{1-\alpha}{2}}u \|_{L^2(\mathbb{R};D_{S,\alpha})}.
        \end{equation*}
        Consequently, we have $ L^{r',2}(\mathbb{R};D_{S,-\alpha}) \hookrightarrow \dot{W}_{-\alpha}$, where  $r'$ is the H\"older conjugate of $r$.
    \item \textbf{(Mixed norm embedding)} For $r\in (2,\infty)$ and $\alpha= 2/ r \in (0,1)$, we have the embedding $\dot{V}_{0} \hookrightarrow L^{r,2}(\mathbb{R};D_{S,\alpha}) $,  
    with 
    \begin{align*}
        \left\| u \right\|_{L^{r,2}(\mathbb{R};D_{S,\alpha})} &\leq C(r)  \|u \|_{L^2(\mathbb{R};D_{S,1})}^\alpha \|D_t^{{1}/{2}}u \|_{L^2(\mathbb{R};H)}^{1-\alpha}.
    \end{align*}
    Consequently, $L^{r',2}(\mathbb{R};D_{S,-\alpha}) \hookrightarrow L^2(\mathbb{R}; D_{S,-1}) + \dot{W}_0=\dot{V}^\star_0$.
    \end{enumerate}
\end{prop}
\begin{proof}
The first result is obtained, as in the proof of \cite[Lemma~5.3]{auscherbaadi2024fundamental}, by using the fact that the $D_{S,\alpha}$-valued Riesz potential with exponent $\frac{1-\alpha}{2}$ is $L^2$–$L^{r,2}$ bounded, a result established by O'Neil \cite{O'Neil1963convolution} (see also Tartar \cite{MR1662313}). The second result follows directly from the first, proceeding as in the proof of \cite[Proposition~5.4]{auscherbaadi2024fundamental}. Consequences by duality are immediate, so we omit the details.
\end{proof}

We now present the following embedding in the space variable, which extends Theorem \ref{thm:fractionalHLS} to boundedness into Lorentz spaces. For the proof, see the case $p = 2$ in \cite[Theorem~3.17]{auscherbaadi2025hardy}.

\begin{thm}\label{thm: HLS Lorentz}
Let $q \in [2,\infty)$. Assume that $\omega \in RH_{\frac{q}{2}}(\mathbb{R}^n)$ and set $\alpha=n\left ( \frac{1}{2}-\frac{1}{q} \right )$. Then, there exists a constant $C = C([ \omega  ]_{A_2},[ \omega  ]_{RH_{\frac{q}{2}}},n,q)$ such that
\begin{equation*}
\|\sqrt{\omega} \,  f \|_{L^{q,2}(\mathbb{R}^n)} \leq C \| (-\Delta_\omega)^{\frac{\alpha}{2}} f \|_{L^2_{\omega}(\mathbb{R}^n)}, \quad \forall f \in D((-\Delta_\omega)^{\frac{\alpha}{2}}).
\end{equation*}
\end{thm}

Motivated by the above results, we introduce the following space for any open interval $I \subset \mathbb{R}$:
\begin{equation*}
     \dot{\mathcal{L}}^{r,q}(I):= \left\{ u \in L^1_{\mathrm{loc}}(\mathbb{R};L^2_\omega(\mathbb{R}^n)) \ : \ \sqrt{\omega} \, u \in L^{r,2}(I;L^{q,2}(\mathbb{R}^n)) \ \ \text{and} \ \ \nabla_x u \in L^2(I;L^2_\omega(\mathbb{R}^n)) \right\},
\end{equation*}
where the gradient is taken in the sense of distributions, with norm
\begin{equation*}
    \|u \|_{\dot{\mathcal{L}}^{r,q}(I)}:= \| \sqrt{\omega} \, u \|_{L^{r,2}(I;L^{q,2}(\mathbb{R}^n))}+ \| \nabla_x u \|_{L^2(I;L^2_\omega(\mathbb{R}^n))}.
\end{equation*}
Proceeding as in the previous section for $\dot{\Sigma}^{r,q}(I)$, we have 
\begin{equation*}
     \dot{\mathcal{L}}^{r,q}(I)^\star = L^2(I;D_{S,-1})+\sqrt{\omega} \,L^{r',2}(I;L^{q',2}(\mathbb{R}^n)).
\end{equation*}
More precisely, for $\varphi \in \dot{\mathcal{L}}^{r,q}(I)^\star$, there exists $(F,g) \in L^2(I;L^2_\omega(\mathbb{R}^n)^n) \times L^{r',2}(I;L^{q',2}(\mathbb{R}^n)) $ such that 
\begin{align*}
    \varphi=-\omega^{-1} \mathrm{div}_x(\omega \, F)+\sqrt{\omega} \,g,
\end{align*}
that is
\begin{align*}
\ \varphi(u)=\int_I \left ( \langle F(t), \nabla_x u(t) \rangle_{2,\omega}+ \langle g(t), u(t)\rangle_{2,\omega} \right ) \, \mathrm{dt}, \quad \forall u \in \dot{\mathcal{L}}^{r,q}(I),
\end{align*}
with 
\begin{align*}
    \|  \varphi \|_{\dot{\mathcal{L}}^{r,q}(I)^\star} \simeq \inf_{\varphi =-\omega^{-1} \mathrm{div}_x(\omega F)+g} \| F \|_{L^2(I;L^2_\omega(\mathbb{R}^n))} +\| \sqrt{\omega} \,g \|_{L^{r',2}(I;L^{q',2}(\mathbb{R}^n))} .
\end{align*}

We record the following lemma, which states the embedding results obtained by combining both time and space embeddings in this new setting.

\begin{lem}\label{lem: coeff Lorentz}
    If $(r,q)$ is an admissible pair and $I \subset \mathbb{R}$ is an interval, then we have a continuous inclusion
    \begin{equation*}
        L^{r,2}(I;D_{S,\frac{2}{r}})^{\mathrm{loc}} \hookrightarrow \sqrt{\omega} \,L^{r,2}(I;L^{q,2}(\mathbb{R}^n))^{\mathrm{loc}}.
    \end{equation*}
    By duality, this yields the continuous embedding
    \begin{equation*}
        \sqrt{\omega} \, L^{r',2}(I;L^{q',2}(\mathbb{R}^n)) \hookrightarrow L^{r',2}(I;D_{S,-\frac{2}{r}})\footnote{This is not a genuine set-theoretic inclusion.}.
    \end{equation*}
    Consequently, we have
    \begin{equation*}
       \dot{V}_0^{\mathrm{loc}} \hookrightarrow \dot{V}_{\frac{2}{r}}^{\mathrm{loc}}  \hookrightarrow L^{r,2}(\mathbb{R};D_{S,\frac{2}{r}})\cap L^2(\mathbb{R};D_{S,1})^{\mathrm{loc}} \hookrightarrow \dot{\mathcal{L}}^{r,q}(\mathbb{R}).
    \end{equation*}
    By duality, this yields the continuous embeddings
    \begin{equation*}
    \dot{\mathcal{L}}^{r,q}(\mathbb{R})^{\star}  \hookrightarrow L^{r',2}(\mathbb{R};D_{S,-\frac{2}{r}})+ L^2(\mathbb{R};D_{S,-1}) \hookrightarrow \dot{V}_{\frac{2}{r}}^{\star} \hookrightarrow \dot{V}_0 ^{\star}.
    \end{equation*}
\end{lem}
\begin{proof}
    This follows by proceeding exactly as in Lemma \ref{lem:xSobolev} and Corollary \ref{cor:embeddings}, using Theorem \ref{thm: HLS Lorentz} and Proposition \ref{prop: embeddings Lorentz} in place of Theorem \ref{thm:fractionalHLS} and Proposition \ref{prop:embeddings}, respectively.
\end{proof}

\subsection{The variational approach and the resulting theory} For any $(r,q)\in [2,\infty]$, we set 
\begin{equation*}
L_{r,q}:=  \| a \|_{L^{{\frac{2r}{r-2}},\infty}(\mathbb{R};{L^{\frac{2q}{q-2},\infty}(\mathbb{R}^n)})} + \| b \|_{L^{{\frac{2r}{r-2}},\infty}(\mathbb{R};{L^{\frac{2q}{q-2},\infty}(\mathbb{R}^n)})}+\| c \|_{L^{{\frac{r}{r-2}},\infty}(\mathbb{R};{L^{\frac{q}{q-2},\infty}(\mathbb{R}^n)})}. 
\end{equation*}
    Once again, this definition is motivated by the following lemma.
\begin{lem}
        Let $(r,q)$ be an admissible pair and assume that $L_{r,q}<\infty$. Then, for all $u,v\in \dot{\mathcal{L}}^{r,q}(\mathbb{R})$, $t \mapsto \beta(u,v)(t) \in L^1(\mathbb{R})$ and we have 
        \begin{equation*}
            |\beta(u,v)| \le L_{r,q} \|u \|_{\dot{\mathcal{L}}^{r,q}(\mathbb{R})} \|v \|_{\dot{\mathcal{L}}^{r,q}(\mathbb{R})}.
        \end{equation*}
        Moreover, there exists a constant $C = C([\omega]_{A_2}, [\omega]_{RH_{\frac{q}{2}}}, n, q) $ such that
    \begin{equation*}
        |\beta(u, v)| \le C L_{r,q} \|u\|_{\dot{V}_0} \|v\|_{\dot{V}_0},
    \end{equation*}
    for all $u, v \in \dot{V}_0^{\mathrm{loc}}$. In other words, $\beta :  \dot{V}_0^{\mathrm{loc}}  \times \dot{V}_0^{\mathrm{loc}} \rightarrow \mathbb{C}$ is a bounded sesquilinear form.
    \end{lem}
    \begin{proof}
        For all $u,v \in \dot{\mathcal{L}}^{r,q}(\mathbb{R})$ and almost every $t\in \mathbb{R}$, we have 
        \begin{align*}
            |\beta(u,v)(t)| \le &\| a(t) \|_{L^{\frac{2q}{q-2},\infty}(\mathbb{R}^n)} \|\sqrt{\omega} \, u(t) \|_{L^{q,2}(\mathbb{R}^n)} \| \nabla_x v(t) \|_{L^2_{\omega}(\mathbb{R}^n)}\\&+ \| b(t) \|_{L^{\frac{2q}{q-2},\infty}(\mathbb{R}^n)} \|\sqrt{\omega} \, v(t) \|_{L^{q,2}(\mathbb{R}^n)} \| \nabla_x u(t) \|_{L^2_{\omega}(\mathbb{R}^n)}\\&+ \| c(t) \|_{L^{\frac{q}{q-2},\infty}(\mathbb{R}^n)} \|\sqrt{\omega} \, u(t) \|_{L^{q,2}(\mathbb{R}^n)} \|\sqrt{\omega} \, v(t) \|_{L^{q,2}(\mathbb{R}^n)} .
        \end{align*}
    This is true as we have H\"older's inequality in Lorentz spaces: the product of three functions in $L^{q_i,r_i}$ belongs to $L^1$ if $\frac{1}{q_1} + \frac{1}{q_2}+\frac{1}{q_3}= 1$ and $\frac{1}{r_1} + \frac{1}{r_2}+\frac{1}{r_2} = 1$. The remainder of the proof follows as in the proof of Lemma \ref{lem:conditions coeff}, and we make use of Lemma \ref{lem: coeff Lorentz} when it comes to involving $\dot{V}_0$.
    \end{proof}

Now, we can proceed as in the previous section, and all the results hold with $\dot{\mathcal{L}}^{r,q}$ replacing $\dot{\Sigma}^{r,q}$. The only exception concerns the causality assumption \textbf{(C)} as stated in Theorem \ref{thm: Causality}. It is based on the first embedding of (4) in Proposition \ref{prop:embeddings}, which does not extend to time-boundedness in Lorentz spaces. In this case, we consider the lower-order terms in the Lorentz space only with respect to the space variable. More precisely, we assume the stronger condition
\begin{equation*}
\ell_{r,q} := \| a \|_{L^{\frac{2r}{r-2}}(\mathbb{R}; L^{\frac{2q}{q-2},\infty}(\mathbb{R}^n))} 
+ \| b \|_{L^{\frac{2r}{r-2}}(\mathbb{R}; L^{\frac{2q}{q-2},\infty}(\mathbb{R}^n))} + \| c \|_{L^{\frac{r}{r-2}}(\mathbb{R}; L^{\frac{q}{q-2},\infty}(\mathbb{R}^n))} < \infty,
\end{equation*}
which is still weaker than the case of lower-order terms in mixed Lebesgue spaces.

\subsubsection*{\textbf{An example}} Let us consider the case $r = 2$ (hence $n \ge 3$ and $q = 2^\star$), $A = I_n$, $a = b = 0$, and $c(t,x) = \dfrac{c_\infty(t,x)}{|x|^2}$, where $c_\infty$ is a bounded, measurable, complex-valued function on $\mathbb{R}^{1+n}$. In this case, the degenerate parabolic operator reduces to the degenerate parabolic Schr\"odinger operator
\begin{equation*}
    \partial_t - \Delta_\omega + \frac{c_\infty(t,x)}{|x|^2}.
\end{equation*}
As the norm $\|c\|_{L^{\infty}(\mathbb{R}; L^{\frac{n}{2},\infty}(\mathbb{R}^n))}$ 
is controlled by $\| c_\infty \|_{L^\infty(\mathbb{R}^{1+n})}$, we state the following result, which is the version of Theorem \ref{thm: Pb Cauchy homogène} in this setting and when causality and invertibility are of perturbative nature (see Theorems \ref{thm: Invertibility} and \ref{thm: Causality}).
\begin{cor}\label{cor: example}
There exists a constant $\varepsilon_0 = \varepsilon_0\bigl(n, [\omega]_{A_2}, [\omega]_{RH_{\frac{n}{n-2}}}\bigr)$ such that if $\| c_\infty \|_{L^\infty(\mathbb{R}^{1+n})} \leq \varepsilon_0$, then the following holds: Let $\rho \in [2,\infty]$ and define $\beta = \dfrac{2}{\rho} \in [0,1]$. Let $h \in L^{\rho'}\bigl((0,\infty); L^2_\omega(\mathbb{R}^n)\bigr)$ 
and $\psi \in L^2_\omega(\mathbb{R}^n)$. Then there exists a unique solution $u \in \dot{\mathcal{L}}^{r,q}((0,\infty))$ to the Cauchy problem
\begin{align*}
\left\{
\begin{array}{ll}
\partial_tu-\Delta_\omega u+\frac{c_\infty(t,x)}{|x|^2}u  =  (-\Delta_\omega)^{\beta/2}h \quad \mathrm{in} \  \mathcal{D}'((0,\infty)\times \mathbb{R}^n), \\
u(t) \rightarrow \psi \  \mathrm{ in } \ \mathcal{D'}(\mathbb{R}^n) \ \mathrm{as} \ t \rightarrow 0^+.
\end{array}\right.
\end{align*} 
Moreover, $u \in C_0([0,\infty); L^2_\omega(\mathbb{R}^n))$, and the map $t \mapsto \|u(t)\|^2_{2,\omega}$ is absolutely continuous. The energy equalities, estimates, and the representation (by the fundamental solution of the degenerate parabolic Schr\"odinger operator) as in Theorem \ref{thm: Pb Cauchy homogène} also hold.
\end{cor}
Here we anticipate a bit, but note that, since $a = b = 0$ and in the setting of the above result, the fundamental solution of the degenerate parabolic Schr\"odinger operator satisfies 
$L^2$ off-diagonal estimates (see Theorem \ref{thm :L2 off diagonal}). In the unweighted setting ($\omega = 1$), with $c_\infty = c$ a negative constant, and for non-negative $L^1$ initial data, the connected yet subtly different question of the existence of a distributional non-negative solution to the above Cauchy problem is also answered positively in \cite{MR742415} for $c \in [-c_n, 0]$, where $c_n = c(n)$ is a constant. Always for $\omega = 1$, the Hardy inequality provides a lower bound, allowing the recovery of the constant $c_n$ via the (inhomogeneous) variational method (see \cite{auscheregert2023universal}). However, the weighted Hardy inequality requires different conditions on $\omega$ than those we imposed (see \cite[Theorem 4]{MR311856} and \cite[Theorem 4.1]{MR1007530}). This is why we have restricted ourselves to the perturbative case in Corollary \ref{cor: example}, which is, of course, still a new result.

\section{The inhomogeneous version}\label{section 5}

So far, the results have a homogeneous flavor. We now develop an inhomogeneous version of the theory, covering non-autonomous elliptic parts with bounded and unbounded lower-order terms.

\subsection{Setup} The key ingredient to obtain the inhomogeneous theory from the homogeneous one is to formulate the homogeneous theory for $\tilde{S} := (1 + S^2)^{1/2} = (1 - \Delta_\omega)^{1/2}$. This formulation naturally leads to the inhomogeneous theory for $S$. In fact, the "homogeneous" fractional domains of $\tilde{S}$ are precisely the "inhomogeneous" fractional domains of $S$. More precisely, for all $\alpha \ge 0$, we have $D_{\tilde{S}, \alpha} = D(S^\alpha) = D((-\Delta_\omega)^{\alpha/2})$ and $\| \cdot \|_{\tilde{S}, \alpha} \simeq \| \cdot \|_{S, \alpha} + \| \cdot \|_{2, \omega} = \| S^\alpha \cdot \|_{2, \omega} + \| \cdot \|_{2, \omega}$. For $\alpha < 0$, we have $D_{\tilde{S}, \alpha} = D_{S, \alpha} + L^2_\omega(\mathbb{R}^n)$, with a norm equivalent to the corresponding quotient norm.

We denote by $(V_\alpha)_{-1 \le \alpha \le 1}$ and $(W_\alpha)_{-1 \le \alpha \le 1}$ the universal solution and source spaces associated with $\tilde{S}$. Note that for all $\alpha \in [0,1]$, we have the continuous embedding $V_\alpha \hookrightarrow \dot{V}_\alpha$, since the operator $D_t^{\frac{1-\alpha}{2}}$ is the same in both spaces. By duality, this implies the dual embedding $\dot{V}_\alpha^\star \hookrightarrow V_\alpha^\star$, where the duality is understood with respect to the $L^2(\mathbb{R}; L^2_\omega(\mathbb{R}^n))$ pairing, as usual. As a consequence, the Hardy–Littlewood–Sobolev-type embeddings stated in Proposition~\ref{prop:embeddings}, both in the homogeneous and inhomogeneous settings, together with their duals, hold. Moreover, for the variational spaces $\dot{V}_0$ and $V_0$, we have
\begin{equation*}
    V_0 = \dot{V}_0 \cap L^2(\mathbb{R}; L^2_\omega(\mathbb{R}^n)) \ (\subset  \dot{V}_0^\mathrm{loc})
    \quad \text{and} \quad 
    \| \cdot \|^2_{V_0} \simeq \| \cdot \|^2_{\dot{V}_0} + \| \cdot \|^2_{L^2(\mathbb{R}; L^2_\omega(\mathbb{R}^n))}.
\end{equation*}

Formulating Proposition \ref{prop: Lions non borné} for $\tilde{S}$ yields the following proposition.

\begin{prop}\label{prop: Lions inh}
    Let $I \subset \mathbb{R}$ be an open interval. Let $u \in L^2(I; H^1_\omega(\mathbb{R}^n))$. Let $\rho \in [2, \infty]$ and set $\beta = \frac{2}{\rho} \in [0,1]$. Assume that $\partial_t u = (-\Delta_\omega)^{\beta/2} h_1 + h_2 $ in $\mathcal{D}'(I \times \mathbb{R}^n)$,
    with $ h_1, h_2 \in L^{\rho'}(I; L^2_\omega(\mathbb{R}^n))$. Then $u \in C_0(\overline{I}; L^2_\omega(\mathbb{R}^n))$, and the map $t \mapsto \|u(t)\|^2_{2,\omega}$ is absolutely continuous on $\overline{I}$. Moreover, the integral identities, as in \eqref{eq:integralidentity}, hold for all $\sigma, \tau \in \overline{I}$ with $\sigma < \tau$.
\end{prop}

\begin{rems}\label{rem: Lions inh}
    \begin{enumerate}
        \item In the case where $I$ is a bounded open interval, the conclusion of the above proposition remains valid by viewing $h_2$ as an element of $L^1(I; L^2_\omega(\mathbb{R}^n))$, and by applying Proposition \ref{prop: Lions non borné} (\textit{i.e.}, the homogeneous theory for $S$, rather than for $\widetilde{S}$). Moreover, in this setting, it suffices to assume that $u \in L^1(I; H^1_\omega(\mathbb{R}^n))$ with $\nabla_x u \in L^2(I; L^2_\omega(\mathbb{R}^n))$, rather than requiring $u\in L^2(I; H^1_\omega(\mathbb{R}^n))$.
        \item Note that, since the embedding $H^1_\omega(\mathbb{R}^n) \hookrightarrow L^2_\omega(\mathbb{R}^n)$ is continuous and dense, the Lions embedding theorem \cite{lions1957problemes} applies only in the case where $I$ is a bounded interval and $\rho=2$, and one must \textit{a priori} assume that $u \in L^2(I; H^1_\omega(\mathbb{R}^n))$. Point (1) above is a better result.
    \end{enumerate}
\end{rems}

\subsection{The degenerate parabolic operator and the variational approach} We consider degenerate parabolic operators of type $\partial_t  + \mathcal{B}_{\mathrm{inh}}$ with a time-dependent degenerate elliptic part $\mathcal{B}_{\mathrm{inh}}$ in divergence form perturbed with both bounded and unbounded lower-order terms. We shall only consider the case of unbounded lower-order terms in mixed Lebesgue spaces, since the case of mixed Lorentz spaces can be treated in exactly the same way, with the modifications already mentioned in Section \ref{section 4} when proving causality results. Thus, we consider
\begin{equation}\label{eq: Binh}
    \mathcal{B}_{\mathrm{inh}}u=-\omega^{-1} \mathrm{div}_x(A\nabla_x u) -\omega^{-1} \mathrm{div}_x(\omega \, a u)+b \cdot \nabla_x u + c u.
\end{equation}
Let $(r,q)$ be an admissible pair. We assume that $a$, $b$, $c$ can be decomposed as $a=a_{\mathrm{h}}+a_\infty$, $b=b_{\mathrm{h}}+b_\infty$, $c=c_{\mathrm{h}}+c_\infty$ with
\begin{equation}\label{eq: quantité Prq}
    P_{r,q}:=  \| a_{\mathrm{h}} \|_{L^{{\frac{2r}{r-2}}}(\mathbb{R};{L^{\frac{2q}{q-2}}(\mathbb{R}^n)})} + \| b_{\mathrm{h}} \|_{L^{{\frac{2r}{r-2}}}(\mathbb{R};{L^{\frac{2q}{q-2}}(\mathbb{R}^n)})}+\| c_{\mathrm{h}} \|_{L^{{\frac{r}{r-2}}}(\mathbb{R};{L^{\frac{q}{q-2}}(\mathbb{R}^n)})} <\infty
\end{equation}
and
\begin{equation}\label{eq: quantité Pinfini}
    P_{\infty,\infty}:= \| a_\infty \|_{L^\infty( \mathbb{R}^{1+n})} + \| b_\infty \|_{L^\infty( \mathbb{R}^{1+n})}+\| c_\infty \|^{\frac{1}{2}}_{L^\infty( \mathbb{R}^{1+n})} <\infty.
\end{equation}
The power $1/2$ on $\| c_\infty \|^{1/2}_{L^\infty(\mathbb{R}^{1+n})}$ is relevant for scaling and for a cleaner presentation, as we will see.
\begin{rem}[Example: subcritical exponents]\label{rem: Subcritical exponents}
Let $(\Tilde{r},\Tilde{q})$ be a pair with $r\le \Tilde{r}$, $2<\Tilde{q}$ and $\frac{1}{\Tilde{r}} + \frac{n}{2 \Tilde{q}} > \frac{n}{4}$, with the additional condition $\frac{1}{\Tilde{r}} - \frac{1}{2\Tilde{q}} < \frac{1}{4}$ when $n=1$. Any coefficients
\begin{equation*}
    a, b \in L^{\frac{2\Tilde{r}}{\Tilde{r}-2}}(\mathbb{R}; L^{\frac{2 \Tilde{q}}{\Tilde{q}-2}}(\mathbb{R}^n)^n)\quad \text{and}  \quad c \in L^{\frac{\Tilde{r}}{\Tilde{r}-2}}(\mathbb{R}; L^{\frac{\Tilde{q}}{\Tilde{q}-2}}(\mathbb{R}^n)),
\end{equation*}
can be decomposed as above, that is, $a = a_{\mathrm{h}} + a_\infty$, $b = b_{\mathrm{h}} + b_\infty$, and $c = c_{\mathrm{h}} + c_\infty$, with 
\begin{equation*}
    a_{\mathrm{h}}, b_{\mathrm{h}} \in L^{\frac{2r}{r-2}}(\mathbb{R}; L^{\frac{2q}{q-2}}(\mathbb{R}^n)^n), \  
c_{\mathrm{h}} \in L^{\frac{r}{r-2}}(\mathbb{R}; L^{\frac{q}{q-2}}(\mathbb{R}^n)), \ \ a_\infty, b_\infty, \in L^\infty(\mathbb{R}^{1+n})^n \ \ \text{and} \ c_\infty \in L^\infty(\mathbb{R}^{1+n}),
\end{equation*}
by visualizing exponents in a ($\frac{\Tilde{r}-2}{\Tilde{r}}$,$\frac{\Tilde{q}-2}{\Tilde{q}}$)-plane and truncating each coefficient at a fixed level $\ell > 0$ (namely, $y_{\mathrm{h}} = y\,\mathbb{1}_{\{|y| > \ell\}}$ and $y_\infty = y\,\mathbb{1}_{\{|y| \le \ell\}}$). In fact, setting $r^\circ = \frac{\Tilde{r}}{\Tilde{r}-2} \in (1,\infty]$ and $q^\circ = \frac{\Tilde{q}}{\Tilde{q}-2} \in (1,\infty]$, then the condition $\frac{1}{\Tilde{r}} + \frac{n}{2 \Tilde{q}} > \frac{n}{4}$ is equivalent to $$\frac{1}{r^\circ} + \frac{n}{2 q^\circ} < 1,$$
which are the subcritical exponents appearing in \cite{aronson1968non, ladyzhenskaia1968linear, AMP2019, auscheregert2023universal}. For the additional assumption when $n=1$, it is equivalent to $\frac{1}{r^\circ} > \frac{1}{2q^\circ}$.
\end{rem}

Next, we introduce the sesquilinear forms associated with the lower-order terms. For all $t\in \mathbb{R}$ and $u,v\in H^1_\omega(\mathbb{R}^n)$, we set 
    \begin{align*}
        \beta_{\mathrm{h}}(t)(u,v)&:= \langle a_{\mathrm{h}}(t) u , \nabla_x v \rangle_{2,\omega}+ \langle b_{\mathrm{h}}(t) \cdot \nabla_x u , v \rangle_{2,\omega}+\langle c_{\mathrm{h}}(t) u, v \rangle_{2,\omega},\\
        \beta_\infty(t)(u,v)&:= \langle a_\infty(t) u , \nabla_x v \rangle_{2,\omega}+ \langle b_\infty(t) \cdot \nabla_x u , v \rangle_{2,\omega}+\langle c_\infty(t) u , v \rangle_{2,\omega},\\
        \beta_{\mathrm{inh}}(t)(u,v)&:= \beta_{\mathrm{h}}(t)(u,v)+ \beta_\infty(t)(u,v).
    \end{align*}
    Now, for all $u,v \in \dot{\Sigma}^{r,q}(\mathbb{R}) \cap L^2(\mathbb{R};L^2_\omega(\mathbb{R}^n))$, we set
    \begin{equation*}
        \beta_{\mathrm{inh}}(u,v):= \beta_{\mathrm{h}}(u,v)+\beta_\infty(u,v):= \int_{\mathbb{R}} \beta_{\mathrm{h}}(t)(u(t),v(t))\, \mathrm{d}t+ \int_{\mathbb{R}} \beta_\infty(t)(u(t),v(t)) \, \mathrm{d}t.
    \end{equation*}
    We record the following lemma.
    \begin{lem}\label{lem: conditions coeff inh}
        For all $u,v \in \dot{\Sigma}^{r,q}(\mathbb{R}) \cap L^2(\mathbb{R};L^2_\omega(\mathbb{R}^n))$, we have 
        \begin{align*}
            &\star \ |\beta_\infty(u,v)| \le P_{\infty,\infty}(\|u \|_{L^2(\mathbb{R};L^2_\omega(\mathbb{R}^n))} \| \nabla_x v \|_{L^2(\mathbb{R};L^2_\omega(\mathbb{R}^n))}+\|\nabla_x u \|_{L^2(\mathbb{R};L^2_\omega(\mathbb{R}^n))} \| v \|_{L^2(\mathbb{R};L^2_\omega(\mathbb{R}^n))})\\& \hspace{4cm}+ P_{\infty,\infty}^{2}\|u \|_{L^2(\mathbb{R};L^2_\omega(\mathbb{R}^n))} \| v \|_{L^2(\mathbb{R};L^2_\omega(\mathbb{R}^n))},
            \\& \star \ |\beta_{\mathrm{h}}(u,v)| \le P_{r,q} \|u \|_{\dot{\Sigma}^{r,q}(\mathbb{R})} \|v \|_{\dot{\Sigma}^{r,q}(\mathbb{R})}.
        \end{align*}
        In particular, $\beta_{\mathrm{inh}} :  V_0 \times V_0 \rightarrow \mathbb{C}$ is a bounded sesquilinear form.
    \end{lem}
    \begin{proof}
    The bound on $\beta_\infty$ follows directly from the Cauchy–Schwarz inequalities. For $\beta_{\mathrm{h}}$, this is exactly Lemma \ref{lem:conditions coeff}. The statement on $\beta_{\mathrm{inh}}$ follows directly by combining Corollary \ref{cor:embeddings} with the embedding $V_0 \hookrightarrow \dot{V}_0$.
    \end{proof}
    \begin{rem}
        We may define the bounded operator, with bound $P_{\infty,\infty}+P_{\infty,\infty}^2$,
        $$\beta_\infty: L^2(\mathbb{R}; H^1_\omega(\mathbb{R}^n)) \rightarrow L^2(\mathbb{R}; H^1_\omega(\mathbb{R}^n))^\star=L^2(\mathbb{R};L^2_\omega(\mathbb{R}^n))+L^2(\mathbb{R};D_{S,-1}).$$
        We also define the operators $\beta_{\mathrm{h}}: \dot{\Sigma}^{r,q}(\mathbb{R}) \rightarrow \dot{\Sigma}^{r,q}(\mathbb{R})^\star= L^2(\mathbb{R};D_{S,-1})+L^{r'}(\mathbb{R},L^{q'}_{\omega^{q'/2}}(\mathbb{R}^n))$ and $\mathcal{A}: L^2(\mathbb{R};D_{S,1}) \rightarrow L^2(\mathbb{R};D_{S,-1})$ as in Remark \ref{rem: beta et A}.
    \end{rem}
    
    In order to employ the variational approach, we introduce the following bounded sesquilinear form $B_{V_0} : V_0 \times V_0 \to \mathbb{C}$, defined for all $u, v \in V_0$ by
\begin{align*}
      B_{V_0}(u,v):&=\int_\mathbb{R}\langle H_t D_t^{1/2}u(t), D_t^{1/2}v(t) \rangle_{2,\omega} +  \langle A(t) \nabla_x u(t), \nabla_x \varphi(t) \rangle_{2,\omega} 
        + \beta_{\mathrm{inh}}(t)(u(t), v(t))\, \mathrm{d}t
        \\&= \int_\mathbb{R}\langle H_t D_t^{1/2}u(t), D_t^{1/2}v(t) \rangle_{2,\omega} \, \mathrm{d}t + \mathcal{A}(u, v) + \beta_{\mathrm{h}}(u, v)+\beta_\infty(u,v)
        \\&=: B_{\dot{V}_0}(u,v)+ \beta_\infty(u,v).
\end{align*}
    Next, we define $\mathcal{H}_{\mathrm{inh}} \colon V_0 \to V_0^\star$ by $\llangle \mathcal{H}_{\mathrm{inh}} u, v \rrangle_{V_0^\star, V_0} := B_{V_0}(u, v)$ for all $u, v \in V_0$. Note that, as before, we have $\left ( \partial_t+\mathcal{B}_{\mathrm{inh}} \right )_{\scriptscriptstyle{\vert V_0}}  = \mathcal{H}_{\mathrm{inh}}$ and $\left ( -\partial_t+\mathcal{B}_{\mathrm{inh}}^\star \right )_{\scriptscriptstyle{\vert V_0 }} = \mathcal{H}_{\mathrm{inh}}^\star$. Note also that we have the decomposition
    \begin{equation*}
    \mathcal{H}_{\mathrm{inh}} = \mathcal{H}_{\mathrm{h}} + \beta_\infty,
    \end{equation*}
    where $\mathcal{H}_{\mathrm{h}}$ denotes the operator constructed in the absence of bounded lower-order terms, as considered in Section \ref{section 3}.

    Now, as in the homogeneous theory, we provide sufficient conditions for invertibility and causality results. We begin with perturbation-type results.

\begin{thm}[Invertibility, inhomogeneous version]\label{thm: Invertibility inh}
    Assume that \eqref{eq: ellipticité} holds. There are two constants $\varepsilon_0 = \varepsilon_0(M, \nu, [\omega]_{A_2}, [\omega]_{RH_{\frac{q}{2}}}, n, q)$ and $\kappa_0=\kappa_0(M, \nu, [\omega]_{A_2}, [\omega]_{RH_{\frac{q}{2}}}, n, q,P_{\infty,\infty})$ such that if $P_{r, q} \leq \varepsilon_0$ and $\kappa \ge \kappa_0$, then the operator $\mathcal{H}_{\text{inh}}+\kappa: V_0 \rightarrow V_0^\star$ is invertible, with a bound of its inverse depending only on $M$ and $\nu$.
\end{thm}
\begin{proof}
    Since $H_t$ is skew-adjoint on $L^2(\mathbb{R};L^2_\omega(\mathbb{R}^n))$, for all $u \in V_0$ and $\kappa>0$, we have 
    \begin{align*}
        \mathrm{Re} \, \llangle (\mathcal{H}_{\text{inh}}+\kappa) u, (1+\delta H_t)u \rrangle_{V_0^\star, V_0} =  \mathrm{Re} \, \llangle \mathcal{H}_{\text{h}} u, (1+\delta H_t)u \rrangle_{\dot{V}_0^\star, \dot{V}_0} &+  \mathrm{Re} \, \beta_\infty(u,(1+\delta H_tu) \\&+ \kappa \| u \|^2_{L^2(\mathbb{R};L^2_\omega(\mathbb{R}^n))}.
    \end{align*}
    Now, proceeding as in the proof of Theorem \ref{thm: Invertibility} for the homogeneous part $\mathcal{H}_{\text{h}}$, and choosing $\delta > 0$ and $\varepsilon_0$ as done there, if $P_{r,q} \le \varepsilon_0$, then by Lemma \ref{lem: conditions coeff inh}, we have
    \begin{align*}
        \mathrm{Re} \, \llangle (\mathcal{H}_{\text{inh}}+\kappa) u, (1+\delta &H_t)u \rrangle_{V_0^\star, V_0} \ge \frac{\delta}{2} \|u \|^2_{\dot{V}_0}\\&- 2P_{\infty,\infty}\sqrt{1+\delta^2} \|u \|_{L^2(\mathbb{R};L^2_\omega(\mathbb{R}^n))} \| \nabla_x u \|_{L^2(\mathbb{R};L^2_\omega(\mathbb{R}^n))} \\& - P_{\infty,\infty}^2\sqrt{1+\delta^2} \|u \|_{L^2(\mathbb{R};L^2_\omega(\mathbb{R}^n))}^2 + \kappa \| u \|^2_{L^2(\mathbb{R};L^2_\omega(\mathbb{R}^n))}.
    \end{align*}
    If $P_{\infty,\infty} = 0$, then we are done. Otherwise, using the inequality $2ab \le \frac{1}{\lambda} a^2 + \lambda b^2$ with $\lambda > 0$ such that $P_{\infty,\infty} \sqrt{1 + \delta^2} \lambda = \frac{\delta}{4}$, and using the fact that $\| \nabla_x \cdot \|_{L^2(\mathbb{R}; L^2_\omega(\mathbb{R}^n))} \le \| \cdot \|_{\dot{V}_0}$, we see that
    \begin{align*}
        \mathrm{Re} \, \llangle (\mathcal{H}_{\text{inh}}+\kappa) u, \, &(1+\delta H_t)u \rrangle_{V_0^\star, V_0} \ge \frac{\delta}{4} \|u \|^2_{\dot{V}_0}\\&+\left(\kappa-P^2_{\infty,\infty}\sqrt{1+\delta^2}-\frac{P_{\infty,\infty}\sqrt{1+\delta^2}}{\lambda} \right) \| u \|^2_{L^2(\mathbb{R};L^2_\omega(\mathbb{R}^n))}.
    \end{align*}
    Now, if 
    \begin{equation*}
        \kappa \ge \kappa_0 := \frac{\delta}{4} + P^2_{\infty,\infty} \sqrt{1 + \delta^2} + \frac{P_{\infty,\infty}\sqrt{1+\delta^2}}{\lambda} = \frac{\delta}{4}+P^2_{\infty,\infty}\left(\sqrt{1 + \delta^2}+\frac{4}{\delta}(1+\delta^2) \right),
    \end{equation*}
    then we have
    \begin{align*}
        \mathrm{Re} \, \llangle (\mathcal{H}_{\text{inh}}+\kappa) &u, (1+\delta H_t)u \rrangle_{V_0^\star, V_0} \ge \frac{\delta}{4} \|u \|^2_{\dot{V}_0}+\frac{\delta}{4} \| u \|^2_{L^2(\mathbb{R};L^2_\omega(\mathbb{R}^n))} \simeq \frac{\delta}{4} \|u \|^2_{V_0},
    \end{align*}
    and the invertibility follows.
\end{proof}
\begin{rem}
As already noted in the proof, if $P_{\infty,\infty} = 0$ and $P_{r,q} \leq \varepsilon_0$, then the operator $\mathcal{H}_\mathrm{h}+\kappa : V_0 \to V_0^\star$ is invertible for all $\kappa > 0$.
\end{rem}

\begin{thm}[Causality, inhomogeneous version]\label{thm: Causality inh}
    Assume that \eqref{eq: ellipticité} holds. There are two constants $\varepsilon_0 = \varepsilon_0(M, \nu, [\omega]_{A_2}, [\omega]_{RH_{\frac{q}{2}}}, n, q)$ and $\kappa_0=\kappa_0(M, \nu, [\omega]_{A_2}, [\omega]_{RH_{\frac{q}{2}}}, n, q,P_{\infty,\infty})$ such that if $P_{r, q} \leq \varepsilon_0$ and $\kappa \ge \kappa_0$, then $\partial_t+\mathcal{B}_{\mathrm{inh}}+\kappa$ is causal in the following sense: if $I$ is an open interval that contains a neighborhood of $-\infty$ and $u \in \dot{\Sigma}^{r,q}(I) \cap L^2(I; L^2_\omega(\mathbb{R}^n))$ satisfies $\partial_t u + \mathcal{B}_{\mathrm{inh}} u+\kappa u = 0$ in $\mathcal{D}'(I \times \mathbb{R}^n)$, then $u = 0$.
\end{thm}
\begin{proof}
    If $u \in \dot{\Sigma}^{r,q}(I) \cap L^2(I; L^2_\omega(\mathbb{R}^n))$ satisfies $\partial_t u + \mathcal{B}_{\mathrm{inh}} u+\kappa u = 0$ in $\mathcal{D}'(I \times \mathbb{R}^n)$, then, by Proposition \ref{prop: Lions inh}, we have $u \in C_0(\overline{I}; L^2_\omega(\mathbb{R}^n))$. We choose $\tau \in \overline{I}$ such that $\|u(\tau)\|_{2,\omega}^2 = \sup_{t \in \overline{I}} \|u(t)\|_{2,\omega}^2 =: \underline{S}$, and set $\underline{I} = \int_{-\infty}^{\tau} \|\nabla_x u(t)\|_{2,\omega}^2 \, \mathrm{d}t$. Writing the energy equality between $\sigma$ and $\tau$, and letting $\sigma \to -\infty$, we deduce, using \eqref{eq: ellipticité}, that
    \begin{align*}
        \underline{S} \le -2\nu \, \underline{I} + 2 \int_{-\infty}^{\tau} |\beta_{\mathrm{h}}(t)(u(t),u(t))| \, \mathrm d t+2 \int_{-\infty}^{\tau} |\beta_\infty(t)(u(t),u(t))| \, \mathrm d t - 2\kappa \|u \|^2_{L^2((-\infty,\tau); L^2_\omega(\mathbb{R}^n))}.
    \end{align*}
    Proceeding as in the proof of Theorem \ref{thm: Causality} for $\beta_{\mathrm{h}}$, and by using Lemma \ref{lem: conditions coeff inh} for $\beta_\infty$, we have 
    \begin{align*}
         \underline{S} &\le -2\nu \, \underline{I} + 2 C P_{r,q} \left ( (\frac{3}{2}-\frac{3}{r})\underline{S} + (\frac{1}{2}+\frac{3}{r})\underline{I} \right )+4P_{\infty,\infty}\underline{I}^{\frac{1}{2}} \|u \|_{L^2((-\infty,\tau); L^2_\omega(\mathbb{R}^n))}\\& \hspace{4cm}+2P^2_{\infty,\infty}\|u \|^2_{L^2((-\infty,\tau); L^2_\omega(\mathbb{R}^n))} - 2\kappa \|u \|^2_{L^2((-\infty,\tau); L^2_\omega(\mathbb{R}^n))}
         \\& \le -2\nu \, \underline{I} + 4 C P_{r,q} \left ( \underline{S} +\underline{I} \right )+ 4P_{\infty,\infty} \underline{I}^{\frac{1}{2}} \|u \|_{L^2((-\infty,\tau); L^2_\omega(\mathbb{R}^n))}+ (2P^2_{\infty,\infty}-2\kappa) \|u \|^2_{L^2((-\infty,\tau); L^2_\omega(\mathbb{R}^n))}
         \\& \le (-2\nu+ 4C P_{r,q}+2P_{\infty,\infty} \lambda )\underline{I} + 4C P_{r,q} \underline{S} + \left(\frac{2P_{\infty,\infty}}{\lambda}+2P^2_{\infty,\infty}-2\kappa\right) \|u \|^2_{L^2((-\infty,\tau); L^2_\omega(\mathbb{R}^n))},
    \end{align*}
    for all $\lambda > 0$, using the inequality $2ab \le \frac{1}{\lambda} a^2 + \lambda b^2$. We choose $\lambda > 0$ such that $-2\nu + 2P_{\infty,\infty} \lambda = -\nu$, {from which the conclusion easily follows by taking $P_{r,q}$ sufficiently small and $\kappa$ sufficiently large, with the desired dependencies.}
\end{proof}

We conclude this section with the following non-perturbative result.

\begin{thm}[Invertibility and causality through lower bounds, inhomogeneous version]\label{thm : Causality and invertibility inh}
The following assertions are true.
\begin{enumerate}
    \item Assume that there exist constants $c,c'>0$ such that 
    \begin{equation*}
        \mathrm{Re} \, \llangle \mathcal{B}_{\mathrm{inh}}u, u \rrangle \ge c \| \nabla_x u \|^2_{L^2(\mathbb{R};L^2_\omega(\mathbb{R}^n))} -c' \|u \|^2_{L^2(\mathbb{R};L^2_\omega(\mathbb{R}^n))} \quad \text{for all} \ u \in V_0.
    \end{equation*}
    Then $\mathcal{H}_{\text{inh}}+\kappa$ is invertible, for all $\kappa \ge  \kappa_0 = \kappa_0(c,c')$.
    \item Assume that that there exists a constant $c'>0$ such that
    \begin{equation*}
        \mathrm{Re} \,( \mathcal{A}(t)(u,v)+ \beta_{\mathrm{inh}}(t)(u, u)) \ge -c' \|u \|^2_{2,\omega} \quad \text{for all} \ t\in \mathbb{R} \ \ \text{and} \ \ u \in H^1_{\omega}(\mathbb{R}^n).
    \end{equation*}
    Then $\partial_t+\mathcal{B}_{\mathrm{inh}}+\kappa$ is causal in the sense of Theorem \ref{thm: Causality inh}, for any $\kappa\ge c'$.
\end{enumerate}
\end{thm}
\begin{proof}
For the first point on invertibility, taking $\kappa_0 := c + c'$, we have, for all $\kappa \ge \kappa_0$,
\begin{equation*}
\mathrm{Re}\,\llangle (\mathcal{B}_{\mathrm{inh}} + \kappa)u, u \rrangle 
    \ge c \left( 
        \| \nabla_x u \|^2_{L^2(\mathbb{R};L^2_\omega(\mathbb{R}^n))} 
        + \|u \|^2_{L^2(\mathbb{R};L^2_\omega(\mathbb{R}^n))} 
    \right)
    = c \|u \|^2_{L^2(\mathbb{R};H^1_\omega(\mathbb{R}^n))}
    \simeq c \|u \|^2_{L^2(\mathbb{R};D_{\tilde{S},1})},
\end{equation*}
for all $u \in V_0$. The argument then follows exactly the same lines as in the proof of Theorem \ref{thm : Causality and invertibility}, with $\tilde{S}$ in place of $S$, that is, the inhomogeneous version of Theorem \ref{thm : Causality and invertibility}. In this case, the corresponding constant $\delta$ obtained there depends on $P_{r,q}$ as well as on $P_{\infty,\infty}$, in addition to the structural constants ($c$, $c'$, $n$, $[\omega]_{A_2}$, $[\omega]_{RH_{\frac{q}{2}}}$, and $M$). We emphasize, also, that by separating the homogeneous and inhomogeneous parts, applying the assumptions, and taking suitable bounds for the bounded and unbounded coefficients, and proceeding as in the proof of Theorem~\ref{thm : Causality and invertibility}, one can choose $\kappa_0$ sufficiently large, of the form $\kappa_0 = C(\delta, M, n, c, c') (1 + P_{\infty,\infty}^2)$, where $\delta$ depends on $c$, $c'$, $n$, $[\omega]_{A_2}$, $[\omega]_{RH_{\frac{n}{n-2}}}$, $M$, and $P_{r,q}$, but not on $P_{\infty,\infty}$.

The second point on causality is straightforward and we omit the details.

\end{proof}

\begin{rem}\label{rem: decoupage}
We observe that if $r > 2$, then for any $a, b \in L^{\frac{2r}{r - 2}}(\mathbb{R}; L^{\frac{2q}{q - 2}}(\mathbb{R}^n)^n) + L^\infty(\mathbb{R}^{1+n})^n$ and $c \in L^{\frac{r}{r - 2}}(\mathbb{R}; L^{\frac{q}{q - 2}}(\mathbb{R}^n)) + L^\infty(\mathbb{R}^{1+n})$, and for every $\varepsilon > 0$, we can write $y = y_{\mathrm{h}} + y_\infty$ for each $y \in \{a, b, c\}$ such that $P_{r,q}(a_{\mathrm{h}}, b_{\mathrm{h}}, c_{\mathrm{h}}) < \varepsilon$ and $P_{\infty,\infty}(a_\infty, b_\infty, c_\infty) < \infty$, by truncating the unbounded parts at large heights $\ell > 0$ (namely, $y_{\mathrm{h}} = y\,\mathbb{1}_{\{|y| > \ell\}}$ and $y_\infty = y\,\mathbb{1}_{\{|y| \le \ell\}}$) together with the dominated convergence theorem for 
$L^{\frac{2q}{q - 2}}(\mathbb{R}^n)$- (resp.\ $L^{\frac{q}{q - 2}}(\mathbb{R}^n)$-) valued functions. In other words, we can always decompose each coefficient into an unbounded part and a bounded part, in such a way that the conditions of Theorems \ref{thm: Invertibility inh} and \ref{thm: Causality inh} are satisfied. However, when $r = 2$ ($n \ge 3$ and $q = \frac{2n}{n - 2}$), we cannot perform the same truncation as in the case $r > 2$, since $\frac{2r}{r - 2} = \frac{r}{r - 2} = \infty$. In this case, we assume directly that $P_{r,q}$ is small or that there exist constants $c, c' > 0$ such that
\begin{equation*}
\mathrm{Re} \,( \mathcal{A}(t)(u,v)+ \beta_{\mathrm{inh}}(t)(u, u)) \ge c \| \nabla_x u \|^2_{2,\omega} - c' \| u \|^2_{2,\omega} \quad \text{for all} \ t\in \mathbb{R} \ \ \text{and} \ \ u \in H^1_{\omega}(\mathbb{R}^n),
\end{equation*}
which ensures that the setting falls within the scope of Theorem \ref{thm : Causality and invertibility inh}, covering both causality and invertibility.
\end{rem}

\subsection{The Cauchy problem on \texorpdfstring{$(0,\mathfrak{T})$}{tempsfini}}
We now study the Cauchy problem in the inhomogeneous setting.

\begin{thm}[The Cauchy problem on $(0, \mathfrak{T})$]\label{thm: Pb de Cauchy inh}
Assume that, for some $\kappa \ge 0$, $\mathcal{H}_{\mathrm{inh}} + \kappa: V_0 \to V_0^\star$ is invertible and $\partial_t + \mathcal{B}_{\mathrm{inh}} + \kappa$ is causal. Let $\mathfrak{T}>0$. Let $\rho \in [2, \infty]$, and define $\beta = \frac{2}{\rho} \in [0, 1]$. Let $h \in L^{\rho'}((0,\mathfrak{T});L^2_\omega(\mathbb{R}^n))$, and $\psi \in L^2_\omega(\mathbb{R}^n)$.
\begin{enumerate}
\item There exists a unique $u \in \dot{\Sigma}^{r,q}((0,\mathfrak{T})) \cap L^2((0,\mathfrak{T});L^2_\omega(\mathbb{R}^n))$ solution to the Cauchy problem 
    \begin{align}\label{eq: Pb Cauchy inhomogène}
    \left\{
    \begin{array}{ll}
        \partial_t u + \mathcal{B}_{\mathrm{inh}}u  =  (-\Delta_\omega)^{\beta/2}h \quad \mathrm{in} \  \mathcal{D}'((0,\mathfrak{T})\times \mathbb{R}^n), \\
        u(t) \rightarrow \psi \  \mathrm{ in } \ \mathcal{D'}(\mathbb{R}^n) \ \mathrm{as} \ t \rightarrow 0^+.
    \end{array}\right.
    \end{align} 
    Moreover, $u\in C([0,\mathfrak{T}];L^2_\omega(\mathbb{R}^n))$, with $u(0)=\psi$, $t \mapsto \| u(t)  \|^2_{2,\omega}$ is absolutely continuous on $[0,\mathfrak{T}]$ and we can write the energy equalities. Furthermore, there exists a constant $C$ independent of $\psi$ and $h$ such that
    \begin{align*}
            \sup_{0\le t \le \mathfrak{T}} \| u(t) \|_{2,\omega}+ \| \nabla_x u\|_{L^2((0,\mathfrak{T});L^2_\omega(\mathbb{R}^n))} 
            \leq C  (  \left \| h \right \|_{L^{\rho'}((0,\mathfrak{T});L^2_\omega(\mathbb{R}^n))}+  \| \psi  \|_{2,\omega}  ).
    \end{align*} 
    \item There exists a unique fundamental solution $\Gamma=(\Gamma(t,s))_{0\leq s \leq t \leq \mathfrak{T}}$  for $\partial_t+\mathcal{B}_{\mathrm{inh}}$ in the sense of Definition \ref{def: FS} in $(0,\mathfrak{T})$. In particular, for all $t \in [0,\mathfrak{T}]$, we have the following representation of $u$ : 
    \begin{align*}
    u(t) = \Gamma(t,0)\psi  +\int_{0}^{t} \Gamma(t,\tau) (-\Delta_\omega)^{\beta/2}h(\tau)\ \mathrm d \tau,
    \end{align*}
    where the integral is weakly defined in $L^2_\omega(\mathbb{R}^n)$, and also strongly defined in the Bochner sense when $\rho = \infty$, as described in Theorem \ref{thm: representation}.
    \end{enumerate}
\end{thm}
\begin{proof}
    We begin with existence. We extend $h$ by $0$ on $[\mathfrak{T},\infty) \times \mathbb{R}^n$. Apply Theorem \ref{thm: Pb Cauchy homogène} with $\Tilde{S}$ replacing $S$ to the auxiliary Cauchy problem 
    \begin{align*}
    \left\{
    \begin{array}{ll}
        \partial_t v + (\mathcal{B}_{\mathrm{inh}}+\kappa)v  =  (-\Delta_\omega)^{\beta/2} e^{-\kappa t}h \quad \mathrm{in} \  \mathcal{D}'((0,\infty)\times \mathbb{R}^n), \\
        v(t) \rightarrow \psi \  \mathrm{ in } \ \mathcal{D'}(\mathbb{R}^n) \ \mathrm{as} \ t \rightarrow 0^+,
    \end{array}\right.
    \end{align*} 
    and obtain a weak solution $v \in \dot{\Sigma}^{r,q}((0,\infty)) \cap L^2((0,\infty);L^2_\omega(\mathbb{R}^n))$. The function $u:=e^{\kappa t} v$ restricted to $[0,\mathfrak{T}]$ gives us a weak solution with the desired properties. 

    Next, we check uniqueness. Assume that $u \in \dot{\Sigma}^{r,q}((0,\mathfrak{T})) \cap L^2((0,\mathfrak{T});L^2_\omega(\mathbb{R}^n))$ is a weak solution to \eqref{eq: Pb Cauchy inhomogène} with $h=0$ and $\psi=0$. Set $v = e^{-\kappa t}u$ on $(0,\mathfrak{T})$, so that $v \in \dot{\Sigma}^{r,q}((0,\mathfrak{T})) \cap L^2((0,\mathfrak{T});L^2_\omega(\mathbb{R}^n))$ and $v$ is a solution to the Cauchy problem
    \begin{align*} \left\{ \begin{array}{ll} \partial_t v + (\mathcal{B}_{\mathrm{inh}}+\kappa)v = 0 \quad \mathrm{in} \ \mathcal{D}'((0,\mathfrak{T})\times \mathbb{R}^n), \\ v(t) \rightarrow 0 \ \mathrm{ in } \ \mathcal{D'}(\mathbb{R}^n) \ \mathrm{as} \ t \rightarrow 0^+, \end{array}\right.
    \end{align*}
    In particular, by Proposition \ref{prop: Lions inh}, we have $v \in C([0,\mathfrak{T}];L^2_\omega(\mathbb{R}^n))$ with $v(0) = 0$. By restriction from $\mathbb{R}$, as in the existence part, we can construct $v_\mathfrak{T} \in \dot{\Sigma}^{r,q}((\mathfrak{T},\infty)) \cap L^2((\mathfrak{T},\infty);L^2_\omega(\mathbb{R}^n)) \cap C_0([\mathfrak{T},\infty);L^2_\omega(\mathbb{R}^n))$ with initial data $v_\mathfrak{T}(\mathfrak{T}) = v(\mathfrak{T})$, which satisfies $\partial_t v + (\mathcal{B}_{\mathrm{inh}}+\kappa)v = 0$ in $\mathcal{D}'((\mathfrak{T},\infty)\times \mathbb{R}^n)$. For $t \in (-\infty, 0] \times \mathbb{R}^n$, we set $v_0(t) = 0$. Using the $L^2_\omega(\mathbb{R}^n)$-valued continuity, we can glue $v_0$, $v$, and $v_\mathfrak{T}$ together to obtain $\tilde{v} \in \dot{\Sigma}^{r,q}(\mathbb{R}) \cap L^2(\mathbb{R};L^2_\omega(\mathbb{R}^n))$ solving $\partial_t \tilde{v} + (\mathcal{B}_{\mathrm{inh}}+\kappa)\tilde{v} = 0$ in $\mathcal{D}'(\mathbb{R}^{1+n})$. Using the inhomogeneous version of Proposition \ref{thm: Uniqueness}, we obtain $\tilde{v} = 0$, hence $v = 0$, and therefore $u = 0$.

    Finally, definition, existence and uniqueness of the fundamental solution $\Gamma$ can be obtained by proceeding as in Section \ref{sub: FS}.

\end{proof}

\begin{rem}
If $\mathcal{H}_{\mathrm{inh}}: V_0 \to V_0^\star$ is invertible and $\partial_t+\mathcal{B}_{\mathrm{inh}}$ is causal (the case $\kappa=0$), then, as in the homogeneous case, we may take $\mathfrak{T}=\infty$ with a source term of the form $(-\Delta_\omega)^{\beta/2} h_1 + h_2$ with $h_1,h_2 \in L^{\rho'}((0,\infty);L^2_\omega(\mathbb{R}^n))$ where $\rho \in [2,\infty]$, and $\beta=\tfrac{2}{\rho} \in [0,1]$. The existence and uniqueness class is $\dot{\Sigma}^{r,q}((0,\infty)) \cap L^2((0,\infty);L^2_\omega(\mathbb{R}^n)),$ with regularity in $C_0([0,\infty);L^2_\omega(\mathbb{R}^n))$.
\end{rem}

\begin{rem}
If $a=0$, then $\beta_\infty u=b_\infty\cdot \nabla_x u+ c_\infty u \in L^1((0,\mathfrak{T});L^2_\omega(\mathbb{R}^n))$ for any $u \in L^1((0,\mathfrak{T});L^2_\omega(\mathbb{R}^n))$ with $\nabla_x u \in L^2((0,\mathfrak{T});L^2_\omega(\mathbb{R}^n))$. In this case, by Point (1) of Remark \ref{rem: Lions inh}, we may extend the uniqueness class in Theorem \ref{thm: Pb de Cauchy inh} to $\dot{\Sigma}^{r,q}((0,\mathfrak{T})) \cap L^1((0,\mathfrak{T});L^2_\omega(\mathbb{R}^n))$.
\end{rem}

\begin{rem}
As mentioned for the homogeneous Cauchy problem, the coefficients of the elliptic part $\mathcal{B}$—namely $A$, $a_{\mathrm{h}}$, $a_\infty$, $b_{\mathrm{h}}$, $b_\infty$, $c_{\mathrm{h}}$, and $c_\infty$—may be defined on $(0,\mathfrak{T}) \times \mathbb{R}^n$. We then extend $A$ by $I_n$, and extend $a_{\mathrm{h}}$, $a_\infty$, $b_{\mathrm{h}}$, $b_\infty$, $c_{\mathrm{h}}$, and $c_\infty$ by $0$, on $\mathbb{R} \setminus (0,\mathfrak{T}) \times \mathbb{R}^n$.
\end{rem}

\begin{rem}\label{rem: FS}
The proof of Theorem \ref{thm: Pb de Cauchy inh} shows that for any $\kappa_0 \ge 0$ such that $\mathcal{H}_{\mathrm{inh}} + \kappa_0 : V_0 \to V_0^\star$ is invertible and $\partial_t + \mathcal{B}_{\mathrm{inh}} + \kappa_0$ is causal, if $\Gamma_{\kappa_0}$ is the fundamental solution for $\partial_t + \mathcal{B}_{\mathrm{inh}} + \kappa_0$, then
\begin{equation*}
    \Gamma(t,s) = e^{\kappa_0 (t-s)} \Gamma_{\kappa_0}(t,s), \quad \text{for all } 0\le s < t \le \mathfrak{T}, \quad \forall \mathfrak{T}>0.
\end{equation*}
Moreover, since the Cauchy problem \eqref{eq: Pb Cauchy inhomogène} can be solved uniquely on any bounded interval $(a,b) \subset \mathbb{R}$, instead of just $(0,\mathfrak{T})$, the fundamental solution $\Gamma$ is defined on $\mathbb{R}$, and the above equality holds throughout $\mathbb{R}$.
\end{rem}

\section{Bounds for the fundamental solution}\label{section 6}

In this section, we establish classical estimates for the fundamental solution in our degenerate setting with both bounded and unbounded lower-order terms. These are the $L^2$ off-diagonal estimates (also known as Gaffney estimates) and pointwise Gaussian upper bounds, under the assumption of Moser’s $L^2$-$L^\infty$ estimates on local weak solutions.

\subsection{\texorpdfstring{$L^2$}{L2}-decay for the fundamental solution}

Let us first recall some relevant literature concerning $L^2$ off-diagonal estimates. In the unweighted case without lower-order terms and with time-independent coefficient matrix $A$, these estimates follow from Davies’ exponential trick \cite{Davies87}. Fabes and Stroock \cite{Fabes1986} extended this approach to time-dependent $A$; see also \cite{hofmann2004gaussian} for related results. In the degenerate case without lower-order terms, Davies’ exponential trick still applies for time-independent $A$; for a more general framework, see \cite{davies1992heat}. In the context of $A_2$-weights, such estimates appear in \cite{cruz2014corrigendum} under the additional assumptions that $A$ is real-valued and symmetric, and have been extended to real-valued, time-dependent $A$ in \cite{ataei2024fundamental}. For the fully general case of complex, time-dependent $A$, see \cite{baadi2025degenerate}, which employs the Grönwall lemma as a key tool. While similar ideas apply when lower-order coefficients are bounded, handling unbounded lower-order terms is more subtle, even in the unweighted case. A novel approach addressing this in the unweighted setting is presented in \cite{auscheregert2023universal}, and we will see that these ideas extend to the weighted setting, providing complete details due to the presence of the weight $\omega$, while refining and adding more precision to certain arguments. 

Taking into account the invertibility and causality results from the previous section, 
we consider an elliptic operator $\mathcal{B}_{\mathrm{inh}}$ as in \eqref{eq: Binh}, 
with lower-order coefficients decomposed as in that section and satisfying 
\eqref{eq: quantité Prq} and \eqref{eq: quantité Pinfini}, and we adopt the notation 
$P_{r,q}$ and $P_{\infty,\infty}$ as used there. We consider two cases: 
\begin{itemize}
\item Case $r > 2$: we assume that \eqref{eq: ellipticité} holds, and by Remark \ref{rem: decoupage}, we also assume that $P_{r,q} \le \varepsilon_0$ as in Theorems \ref{thm: Invertibility inh} and \ref{thm: Causality inh}. 
\item Case $r = 2$: we assume the existence of constants $c, c' > 0$ such that
\begin{equation}\label{eq: cas2}
    \mathrm{Re} \left( \mathcal{A}(t)(u, u) + \beta_{\mathrm{inh}}(t)(u, u) \right) 
\ge c \| \nabla_x u \|^2_{2,\omega} - c' \| u \|^2_{2,\omega}, 
\quad \text{for all } t \in \mathbb{R} \text{ and } u \in H^1_\omega(\mathbb{R}^n).
\end{equation}
In this case, Theorem \ref{thm : Causality and invertibility inh} applies.
\end{itemize}
In both situations, Theorem \ref{thm: Pb de Cauchy inh} applies. We now state the following $L^2$ off-diagonal estimates for the fundamental solution $\Gamma$ constructed therein.

\begin{thm}\label{thm :L2 off diagonal} Fix $\mathfrak{T}>0$. Then, there exist constants $c,C>0$ and such that
\begin{equation}
\| \Gamma(t,s) \psi \|_{L^2_\omega(F)} \le C e^{-\frac{\mathrm{d}(E,F)^2}{4c(t-s)}+cP_{\infty,\infty}^2(t-s)} \|\psi \|_{L^2_\omega(E)},
\end{equation}
for all $t,s\in [0,\mathfrak{T}]$ with $s<t$, all measurable sets $E,F \subset \mathbb{R}^n$ and all $\psi \in L^2_\omega(\mathbb{R}^n)$, with support in $E$. Here, $\mathrm{d}(E,F)$ denotes the Euclidean distance between $E$ and $F$.
\begin{itemize}
\item In the case $r>2$, the constants $c$ and $C$ depend only on $M$, $\nu$, $[\omega]_{A_2}$, $[\omega]_{RH_{\frac{q}{2}}}$, $n$, and $q$.
\item In the case $r=2$, they depend on $c$, $c'$, $n$, $[\omega]_{A_2}$, $[\omega]_{RH_{\frac{n}{n-2}}}$, $M$, and also on $P_{r,q}$.
\end{itemize}
\end{thm}
\begin{proof}
We will prove the result for all $-\infty < s < t < \infty$, since the fundamental solution is defined on $\mathbb{R}$, see Remark \ref{rem: FS}. For a bounded Lipschitz function $\chi : \mathbb{R}^n \to \mathbb{R}^+$, we define the operator on $\mathbb{R} \times \mathbb{R}^n$ obtained by conjugating $\partial_t +
\mathcal{B}_{\mathrm{inh}}$ with the multiplication by $e^{\chi}$. More precisely, in the weak sense $\llangle \cdot, \cdot \rrangle$ defined in Definition \ref{def: distributions}, we have
\begin{equation}\label{eq: Conjugaison}
e^{\chi}(\partial_t+\mathcal{B}_{\mathrm{inh}})e^{-\chi}= \partial_t+\mathcal{B}_{\mathrm{inh}} + \beta_\chi  \quad \text{in} \ \mathcal{D}'(\mathbb{R}\times \mathbb{R}^n),
\end{equation}
with
\begin{align*}
\llangle \beta_\chi u,\varphi \rrangle:= \beta_\chi( u,\varphi):&=\llangle -\omega^{-1}(A \nabla_x \chi)u ,\nabla_x \varphi \rrangle+ \llangle \omega^{-1}A \nabla_x u \cdot \nabla_x \chi ,\varphi \rrangle\\& \hspace{0.5cm}+\llangle -\omega^{-1}A \nabla_x \chi \cdot \nabla_x \chi ,\varphi \rrangle+\llangle (a-b)\cdot(\nabla_x \chi) \, u ,\varphi \rrangle
\\&= \llangle a_\chi u,\nabla_x\varphi \rrangle+ \llangle b_\chi \cdot \nabla_x u,\varphi \rrangle+\llangle c_\chi u,\varphi \rrangle,
\end{align*}
with
\begin{align*}
    a_\chi &:= -\omega^{-1} A \nabla_x \chi,\\
    b_\chi &:= \omega^{-1} A^T \nabla_x \chi,\\
    c_\chi &:= -\omega^{-1} A \nabla_x \chi \cdot \nabla_x \chi + (a - b) \cdot \nabla_x \chi,
\end{align*}
where $A^T$ denotes the real transpose of $A$. 
The coefficients $a_\chi$ and $b_\chi$ are bounded by the term $\| \omega^{-1} A \|_{L^\infty(\mathbb{R}^{n+1})} \| \nabla_x \chi \|_{L^\infty(\mathbb{R}^n)} = M \| \nabla_x \chi \|_{L^\infty(\mathbb{R}^n)}$. The first term in $c_\chi$ is bounded by $M \| \nabla_x \chi \|_{L^\infty(\mathbb{R}^n)}^2$. However, the second term in $c_\chi$ is more delicate to handle directly, which makes it necessary to distinguish between the cases $r>2$ and $r=2$.\newline
\textit{\textbf{Proof in the case $r>2$.}} 
We proceed in several steps.\\[0.3em]
\textit{Step 1: Decomposition and lower bound for the inverse of $\partial_t + \mathcal{B}_{\mathrm{inh}} + \beta_\chi$.} We may write $a - b = a_\infty - b_\infty + a_{\mathrm{h}} - b_{\mathrm{h}}$, where $P_{r,q}(a_{\mathrm{h}},b_{\mathrm{h}}) \le \varepsilon_0$ and $P_{\infty,\infty}(a_\infty,b_\infty) < \infty$. Thus, the second term of $c_\chi$ can be written as $(a_\infty - b_\infty) \cdot \nabla_x \chi + (a_{\mathrm{h}} - b_{\mathrm{h}}) \cdot \nabla_x \chi$. The first part is controlled by $2 P_{\infty,\infty} \| \nabla_x \chi \|_{L^\infty(\mathbb{R}^n)}$. Let $\partial_t + \mathcal{B}_{\mathrm{h}}$ denote the degenerate parabolic operator associated with $A$, $a_{\mathrm{h}}$, $b_{\mathrm{h}}$, and $c_{\mathrm{h}}$ and let $\mathcal{H}_\mathrm{h}$ be the variational operator associated to $\partial_t + \mathcal{B}_{\mathrm{h}}$. Remark that, since $P_{r,q} \le \varepsilon_0$ by hypothesis, $\mathcal{H}_{\mathrm{h}}$ satisfies \eqref{eq: LaxM H} on $V_0$ with $\delta/2$, where $\delta = \delta(M, \nu)$. For $V \in L^{\frac{r}{r - 2}}(\mathbb{R}; L^{\frac{q}{q - 2}}(\mathbb{R}^n))$, the multiplication by $V$ defines a bounded operator 
$\dot{V}_0^{\mathrm{loc}} \to \dot{V}^\star_0$, 
given by $\llangle V u, \varphi \rrangle_{\dot{V}^\star_0,\dot{V}_0} := \llangle V u, \varphi \rrangle$, 
with bound $C \|V\|_{L^{\frac{r}{r - 2}}(\mathbb{R}; L^{\frac{q}{q - 2}}(\mathbb{R}^n))}$, where 
$C = C([\omega]_{A_2}, [\omega]_{RH_{\frac{q}{2}}}, n, q)$ is a constant. In particular, there exists a constant $\lambda = \lambda ([ \omega  ]_{A_2},[ \omega  ]_{RH_{\frac{q}{2}}},n,q)$ such that if $\|V\|_{L^{\frac{r}{r - 2}}(\mathbb{R}; L^{\frac{q}{q - 2}}(\mathbb{R}^n))} \leq \lambda $, then 
\begin{align*}
\mathrm{Re} \llangle (\mathcal{H}_\mathrm{h}+V)u, (1+\delta H_t)u \rrangle_{\dot{V}^\star_0,\dot{V}_0} \geq \frac{\delta}{4} \|u \|_{\dot{V}_0}^2, \quad \text{for all} \ u \in V_0.
\end{align*}
For all $\ell > 0$, using Markov's inequality, the truncation $V_{\mathrm{h}} := \mathbb{1}_{|a_{\mathrm{h}} - b_{\mathrm{h}}| > \ell}(a_{\mathrm{h}} - b_{\mathrm{h}}) \cdot \nabla_x \chi$ satisfies $$\|V_{\mathrm{h}}\|_{L^{\frac{r}{r - 2}}(\mathbb{R}; L^{\frac{q}{q - 2}}(\mathbb{R}^n))} \leq 4\varepsilon_0^2 \ell^{-1} \| \nabla_x \chi \|_{L^\infty(\mathbb{R}^n)}.$$
We then choose $\ell > 0$ with $4\varepsilon_0^2 \ell^{-1} \| \nabla_x\chi \|_{L^\infty(\mathbb{R}^n)} = \lambda$. On the other hand, the complementary truncation $V_\infty := \mathbb{1}_{|a_{\mathrm{h}} - b_{\mathrm{h}}| \le \ell}(a_{\mathrm{h}} - b_{\mathrm{h}}) \cdot \nabla_x \chi$ has the bound $$\|V_\infty\|_{L^\infty(\mathbb{R}^{1+n})} \leq \ell \|\nabla_x \chi \|_{L^\infty(\mathbb{R}^n)} = \frac{4\varepsilon_0^2}{\lambda} \| \nabla_x\chi \|^2_{L^\infty(\mathbb{R}^n)}.$$
We set $\beta_\infty:=\mathcal{B}_{\text{inh}}-\mathcal{B}_{\text{h}}$ and $\Tilde{\beta}_\chi:=\beta_\chi-V_{\mathrm{h}}$, so that we have the following decomposition:
\begin{equation*}
\partial_t+ \mathcal{B}_{\text{inh}}+\beta_\chi+\kappa=(\partial_t+ \mathcal{B}_{\text{h}}+V_{\mathrm{h}})+(\Tilde{\beta}_\chi+\beta_\infty)+\kappa.
\end{equation*}
The variational operator associated to $\partial_t + \mathcal{B}_{\text{h}} + V_{\mathrm{h}}$ is $\mathcal{H}_\mathrm{h} + V_{\mathrm{h}}$, and the term $\Tilde{\beta}_\chi + \beta_\infty$ has first order coefficients bounded by $P_{\infty,\infty} + \| \nabla_x\chi \|_{L^\infty(\mathbb{R}^n)}$, and zero order coefficients bounded by $ P_{\infty,\infty}^2 + \| \nabla_x\chi \|^2_{L^\infty(\mathbb{R}^n)}$, up to multiplicative constants depending only on $\nu$, $M$, $[\omega]_{A_2}$, $[\omega]_{RH_{\frac{q}{2}}}$, $n$, and $q$.

Proceeding as in the proof of Theorem \ref{thm: Invertibility inh}, we obtain the lower bound
\begin{equation*}
\mathrm{Re} \llangle (\mathcal{H}_{\mathrm{inh}} + \beta_\chi + \kappa) u, (1 + \delta H_t) u \rrangle_{V_0^\star, V_0} \ge \frac{\delta}{8} \| u \|_{V_0}^2,
\end{equation*}
for all $\kappa \ge \kappa_0 := 1 + C(\delta) \big(P_{\infty,\infty}^2 + \| \nabla_x \chi \|_{L^\infty(\mathbb{R}^n)}^2 \big)$ and all $u \in V_0$.
\newline
\textit{Step 2:  Norm bounds and scaling.} We denote by $\Gamma^{\chi}_{k_0}$ the fundamental solution for $\partial_t+ \mathcal{B}_{\text{inh}}+\beta_\chi+\kappa_0$ and $\Gamma_{k_0}$ the fundamental solution for $\partial_t+ \mathcal{B}_{\text{inh}}+\kappa_0$. Using \eqref{eq: Conjugaison}, we have $\Gamma_{k_0}=e^{-\chi}\Gamma^{\chi}_{k_0}e^{\chi}$ and by Remark \ref{rem: FS}, we have $\Gamma(t,s)=e^{k_0(t-s)} \Gamma_{k_0}$, for all $s<t$. Thus, for all $\psi \in L^2_\omega(\mathbb{R}^n)$ and all $s<t$, 
\begin{equation*}
\|e^\chi \Gamma(t,s) \psi \|_{L^2_\omega(\mathbb{R}^n)} \le e^{\kappa_0(t-s)} \| \Gamma^{\chi}_{k_0}(t,s) e^\chi \psi\|_{L^2_\omega(\mathbb{R}^n)}.
\end{equation*}
We know that $\Gamma^{\chi}_{k_0}(t,s))$, ${s<t}$, are uniformly bounded operators on $L^2_\omega(\mathbb{R}^n)$ with a bound of type $$C(\delta)(1+M+P_{\infty,\infty}^2 + \|\nabla_x \chi \|^2_{L^\infty(\mathbb{R}^n)}).$$
We first assume that $t-s=1$. Using this, we have
\begin{equation*}
\|e^\chi \Gamma(t,s) \psi \|_{L^2_\omega(\mathbb{R}^n)} \le (C(\delta) e)(1+M+P_{\infty,\infty}^2 + \| \nabla_x\chi \|^2_{L^\infty(\mathbb{R}^n)}) e^{c_\delta(\| \nabla_x\chi \|^2_{L^\infty(\mathbb{R}^n)}+P_{\infty,\infty}^2 )}\| e^\chi \psi\|_{L^2_\omega(\mathbb{R}^n)}.
\end{equation*}
Using the elementary inequality $(1+M+x)e^{c_\delta x} \le (2+M)e^{(c_\delta+1)x}$ for all $x\ge 0$, we deduce that 
\begin{equation}\label{eq: scaling1}
\|e^\chi \Gamma(t,s) \psi \|_{L^2_\omega(\mathbb{R}^n)} \le (C(\delta) e)(2+M) e^{(c_\delta+1)(\| \nabla_x\chi \|^2_{L^\infty(\mathbb{R}^n)}+P_{\infty,\infty}^2 )}\| e^\chi \psi\|_{L^2_\omega(\mathbb{R}^n)}.
\end{equation}
A scaling argument will allow to remove the assumption $t-s=1$. By definition, $u:=\Gamma(\cdot,s)\psi$ is the unique solution in $\dot{\Sigma}^{r,q}(\mathbb{R}) \cap L^2(\mathbb{R};L^2_\omega(\mathbb{R}^n))$ to the equation $$\partial_t u+\mathcal{B}_{\mathrm{inh}}u=\delta_s \otimes \psi \quad \text{in} \ \mathcal{D}'(\mathbb{R}\times \mathbb{R}^n).$$
       Fix $R>0$. Taking $\varphi \in \mathcal{D}(\mathbb{R}\times \mathbb{R}^n)$ and testing the above equation against $\varphi_R(t,s):=\varphi(s+\frac{t-s}{R^2},\frac{x}{R})$, we deduce that $u_R(t,x):=u(s+R^2(t-s),Rx)$ solves $$\partial_t u_R+\mathcal{B}^R_{\mathrm{inh}}u_R=\delta_s \otimes \psi_R \quad \text{in} \ \mathcal{D}_R'(\mathbb{R}\times \mathbb{R}^n),$$
       where this distributional anti-duality bracket is exactly the one in \ref{def: distributions}, but with $\omega_R(x):=\omega(Rx)$ instead of $\omega$, $\psi_R(x):=\psi(Rx)$, and the elliptic part $\mathcal{B}^R_{\mathrm{inh}}$ given by
       $$\mathcal{B}^R_{\mathrm{inh}}u_R=-\omega_R^{-1}\mathrm{div}_x(A_R\nabla_x u_R) -\omega_R^{-1} \mathrm{div}_x(\omega_R \, a_R u_R)+b_R \cdot \nabla_x u_R + c_R u_R,$$
       with
       \begin{align*}
           &A_R(t,x)=A(s+R^2(t-s),Rx),\\
           &a_R(t,x):=R \, a(s+R^2(t-s),Rx), \ b_R(t,x):=R\,b(s+R^2(t-s),Rx),\\
           &c_R(t,x):=R^2 \, c(s+R^2(t-s),Rx).
       \end{align*}
        Note that $\omega_R \in A_2(\mathbb{R}^n)\cap RH_{\frac{q}{2}}(\mathbb{R}^n)$ with $[\omega_R]_{A_2} = [\omega]_{A_2}$ and $[\omega_R]_{RH_{\frac{q}{2}}} = [\omega]_{RH_{\frac{q}{2}}}$. Moreover, we have $P^R_{\infty,\infty} = R\,P_{\infty,\infty}$, while $P_{r,q}$ is scale-invariant, \textit{i.e.} $P^R_{r,q} = P_{r,q}$, since $\frac{2}{r} = n\left( \frac{1}{2} - \frac{1}{q} \right)$ by the definition of admissibility. Now, applying \eqref{eq: scaling1} to the fundamental solution of $\partial_t+\mathcal{B}^R_{\mathrm{inh}}$ at $t$ such that $t-s=1$ with $\chi_R(x):=\chi(Rx)$, and changing variables in space, yields
       \begin{equation*}
           \|e^\chi \Gamma(s+R^2,s) \psi \|_{L^2_\omega(\mathbb{R}^n)} \le (C(\delta) e)(2+M) e^{(c_\delta+1)(\| \nabla_x\chi \|^2_{L^\infty(\mathbb{R}^n)}+P_{\infty,\infty}^2 )R^2}\| e^\chi \psi\|_{L^2_\omega(\mathbb{R}^n)}.
       \end{equation*}
       In particular, for all $t>s$ and $\psi\in L^2_\omega(\mathbb{R}^n)$, one has
       \begin{equation}\label{eq: scaling2}
           \|e^\chi \Gamma(t,s) \psi \|_{L^2_\omega(\mathbb{R}^n)} \le (C(\delta) e)(2+M) e^{(c_\delta+1)(\| \nabla_x\chi \|^2_{L^\infty(\mathbb{R}^n)}+P_{\infty,\infty}^2 )(t-s)}\| e^\chi \psi\|_{L^2_\omega(\mathbb{R}^n)}.
       \end{equation}
       \newline
       \textit{Step 3: choice of $\chi$.} Let $E,F \subset \mathbb{R}^n$ be two measurable sets and let $s<t$. We can assume that $t-s< \mathrm{d}(E,F)^2$, otherwise we conclude easily using \eqref{eq: scaling2} with $\chi=0$. Fix $N> \frac{\mathrm{d}(E,F)^2}{(c_\delta+1)(t-s)}$,
       and set
       $$\chi(x):=\inf(\frac{\mathrm{d}(E,F)\mathrm{d}(x,E)}{2(c_\delta+1)(t-s)},N), \quad \text{for all} \ x\in \mathbb{R}^n.$$
       Then, we have $\chi_{\scriptscriptstyle{\vert F}} \ge \frac{\mathrm{d}(E,F)^2}{2(c_\delta+1)(t-s)}$, $\chi_{\scriptscriptstyle{\vert E}}=0$ and $\| \nabla_x\chi \|_{L^\infty(\mathbb{R}^n)} \le \frac{\mathrm{d}(E,F)}{2(c_\delta+1)(t-s)}$. Therefore, for $\psi \in L^2_\omega(\mathbb{R}^n)$ with support in $E$, applying \eqref{eq: scaling2} yields 
       \begin{align*}
           e^{\frac{\mathrm{d}(E,F)^2}{2(c_\delta+1)(t-s)}} \|\Gamma(t,s) \psi \|_{L^2_\omega(F)} &\leq  \|e^\chi \Gamma(t,s) \psi \|_{L^2_\omega(\mathbb{R}^n)} \\&\le C(\delta,M) e^{(c_\delta+1)(\frac{\mathrm{d}(E,F)^2}{4(c_\delta+1)^2(t-s)^2}+P_{\infty,\infty}^2 )(t-s)}\| \psi\|_{L^2_\omega(\mathbb{R}^n)},
       \end{align*}
       with $C(\delta,M)=(C(\delta) e)(2+M)$. Therefore, we obtain
       \begin{equation*}
           \|\Gamma(t,s) \psi \|_{L^2_\omega(F)} \le C(\delta,M) e^{-\frac{\mathrm{d}(E,F)^2}{4(c_\delta+1)(t-s)}+(c_\delta+1)P_{\infty,\infty}^2(t-s)}\| \psi\|_{L^2_\omega(\mathbb{R}^n)}.
       \end{equation*}
\textit{\textbf{Proof in the case $r=2$.}} Integrating \eqref{eq: cas2} against $t$, we obtain
\begin{equation}\label{eq: Cas2}
\mathrm{Re} \, \llangle \mathcal{B}_{\mathrm{inh}}u, u \rrangle \ge c \| \nabla_x u \|^2_{L^2(\mathbb{R};L^2_\omega(\mathbb{R}^n))} -c' \|u \|^2_{L^2(\mathbb{R};L^2_\omega(\mathbb{R}^n))} \quad \text{for all} \ u \in V_0.
\end{equation}
As $r = 2$, we have $n \ge 3$, $q = \frac{2n}{n - 2} = 2^\star$, and therefore $\frac{q}{q - 2} = \frac{n}{2}$. For $V \in L^{\infty}(\mathbb{R}; L^{\frac{n}{2}}(\mathbb{R}^n))$, we have
\begin{align*}
           |\llangle V u, v \rrangle| &\le \|V\|_{L^{\infty}(\mathbb{R}; L^{\frac{n}{2}}(\mathbb{R}^n))} \| u \|_{L^2(\mathbb{R}; L^{2^\star}_{\omega^{2^\star/2}}(\mathbb{R}^n))} \| v \|_{L^2(\mathbb{R}; L^{2^\star}_{\omega^{2^\star/2}}(\mathbb{R}^n))} 
           \\ &\le C \|V\|_{L^{\infty}(\mathbb{R}; L^{\frac{n}{2}}(\mathbb{R}^n))} \| \nabla_x u \|_{L^2(\mathbb{R}; L^{2}_{\omega}(\mathbb{R}^n))} \| \nabla_x v \|_{L^2(\mathbb{R}; L^{2}_{\omega}(\mathbb{R}^n))},
\end{align*}
with $C = C([\omega]_{A_2}, [\omega]_{RH_{\frac{n}{n-2}}}, n)$ a constant. In particular, if $\|V\|_{L^{\infty}(\mathbb{R}; L^{\frac{n}{2}}(\mathbb{R}^n))} \leq \frac{c}{2C}$, then \eqref{eq: Cas2} still holds upon adding $V$, with $\frac{c}{2}$ in place of $c$. By definition, we decompose $a - b = a_\infty - b_\infty + a_{\mathrm{h}} - b_{\mathrm{h}}$, where $a_{\mathrm{h}}, b_{\mathrm{h}} \in L^{\infty}(\mathbb{R}; L^{n}(\mathbb{R}^n))$, so that $P_{r,q}(a_{\mathrm{h}},b_{\mathrm{h}}) < \infty$ and $P_{\infty,\infty}(a_\infty,b_\infty) < \infty$, but we do not control these quantities. As before, we decompose $(a_{\mathrm{h}} - b_{\mathrm{h}})\cdot \nabla_x \chi$ into $V_{\mathrm{h}} + V_\infty$, with $V_{\mathrm{h}} \in L^{\infty}(\mathbb{R}; L^{\frac{n}{2}}(\mathbb{R}^n))$ and $V_\infty \in L^{\infty}(\mathbb{R}^{1+n})$, by setting
\begin{align*}
    &V_{\mathrm{h}} := \mathbb{1}_{|a_{\mathrm{h}} - b_{\mathrm{h}}| > \ell}(a_{\mathrm{h}} - b_{\mathrm{h}})\cdot \nabla_x \chi,\\
    &V_\infty := \mathbb{1}_{|a_{\mathrm{h}} - b_{\mathrm{h}}| \le \ell}(a_{\mathrm{h}} - b_{\mathrm{h}})\cdot \nabla_x \chi,
\end{align*}
and we choose $\ell > 0$ such that $$\frac{\|a_{\mathrm{h}} - b_{\mathrm{h}}\|^2_{L^{\infty}(\mathbb{R}; L^{n}(\mathbb{R}^n))}\|\nabla_x \chi\|_{L^\infty(\mathbb{R}^n)}}{\ell} = \frac{c}{2C},$$ so that

\begin{align*}
    &\|V_{\mathrm{h}} \|_{L^{\infty}(\mathbb{R}; L^{\frac{n}{2}}(\mathbb{R}^n))} \leq \frac{\|a_{\mathrm{h}}-b_{\mathrm{h}} \|^2_{L^{\infty}(\mathbb{R}; L^{n}(\mathbb{R}^n))}\| \nabla_x\chi \|_{L^\infty(\mathbb{R}^n)}}{\ell}=\frac{c}{2C}, \\&\|V_\infty \|_{L^{\infty}(\mathbb{R}^{1+n})} \le \ell \| \nabla_x\chi \|_{L^\infty(\mathbb{R}^n)}= \frac{2C \|a_{\mathrm{h}}-b_{\mathrm{h}} \|^2_{L^{\infty}(\mathbb{R}; L^{n}(\mathbb{R}^n))}}{c} \| \nabla_x\chi \|_{L^\infty(\mathbb{R}^n)}^2.
\end{align*}
The adapted decomposition this time is simply
\begin{equation*}
\partial_t+ \mathcal{B}_{\text{inh}}+\beta_\chi+\kappa= (\partial_t+ \mathcal{B}_{\text{inh}}+V_{\mathrm{h}})+\beta_\chi-V_{\mathrm{h}}+\kappa.
\end{equation*}
The operator $\partial_t + \mathcal{B}_{\mathrm{inh}} + V_{\mathrm{h}}$ still satisfies \eqref{eq: Cas2} with $\frac{c}{2}$ instead of $c$, and $\beta_\chi - V_{\mathrm{h}}$ has first order coefficients bounded by $2M \| \nabla_x \chi \|_{L^\infty(\mathbb{R}^n)}$, and zeroth order coefficients bounded by the term $$M \| \nabla_x \chi \|_{L^\infty(\mathbb{R}^n)}^2 + 2P_{\infty,\infty}(a_\infty, b_\infty) \| \nabla_x \chi \|_{L^\infty(\mathbb{R}^n)} + \tilde{C} \| \nabla_x \chi \|_{L^\infty(\mathbb{R}^n)}^2,$$
where $\tilde{C} := \frac{2C \| a_{\mathrm{h}} - b_{\mathrm{h}} \|^2_{L^{\infty}(\mathbb{R}; L^{n}(\mathbb{R}^n))}}{c}$. We write this bound of zeroth order coefficients as $$C_1 (1 +P_{\infty,\infty}^2+ \| \nabla_x \chi \|_{L^\infty(\mathbb{R}^n)}^2),$$
where $C_1$ depends also on $P_{r,q}$ but not on $P_{\infty,\infty}$.

Proceeding as in the proof of Theorem \ref{thm : Causality and invertibility inh}, we prove that we have the lower bound
\begin{equation*}
           \mathrm{Re} \llangle (\mathcal{H}_{\text{inh}}+\beta_\chi+\kappa)u, (1+\delta H_t)u \rrangle_{V^\star_0,V_0} \geq \frac{\delta}{4} \|u \|_{V_0}^2 \quad \text{for all} \ u\in V_0,
\end{equation*}
and for all 
$\kappa \ge k_0:=1+C_2(1 +P_{\infty,\infty}^2+ \|\nabla_x \chi \|^2_{L^\infty(\mathbb{R}^n)}),$with $C_2$ depending on $c$, $c'$, $n$, $[ \omega  ]_{A_2}$, $[ \omega  ]_{RH_{\frac{n}{n-2}}}$, $M$ and also on $P_{r,q}$, but not on $P_{\infty,\infty}$ and $\delta$ is independent of $\chi$ and $P_{\infty,\infty}$. Now, it is easy to conclude by proceeding exactly as in the case of $r>2$: the fundamental solution $\Gamma^{\chi}_{k_0}(t,s)$, ${s<t}$, are uniformly bounded operators on $L^2_\omega(\mathbb{R}^n)$ with a bound $$C_3(1+M+P_{r,q}^2+P_{\infty,\infty}^2 + \|\nabla_x \chi \|^2_{L^\infty(\mathbb{R}^n)}),$$
with $C_3$ depending on $c$, $c'$, $n$, $[ \omega  ]_{A_2}$, $[ \omega  ]_{RH_{\frac{n}{n-2}}}$, $M$, and we write this bound as $$C_4(1 +P_{\infty,\infty}^2 + \| \nabla_x \chi \|^2_{L^\infty(\mathbb{R}^n)}),$$
with $C_4$ depending only on $c$, $c'$, $n$, $[\omega]_{A_2}$, $[ \omega  ]_{RH_{\frac{n}{n-2}}}$, $M$ and $P_{r,q}$. Using an elementary inequality as before in the case $r>2$, we obtain that, for all $s<t$ with $t-s=1$,
\begin{equation}\label{eq: scaling3}
    \|e^\chi \Gamma(t,s) \psi \|_{L^2_\omega(\mathbb{R}^n)} \le 2C_4 e^{(C_2+1)(1+P_{\infty,\infty}^2 +\|\nabla_x\chi \|^2_{L^\infty(\mathbb{R}^n)})}\| e^\chi \psi\|_{L^2_\omega(\mathbb{R}^n)}.
\end{equation}
A scaling argument, together with choosing $\chi$ exactly as before, yields that
\begin{equation*}
        \| \Gamma(t,s) \psi \|_{L^2_\omega(F)} \le 2C_4 e^{C_2+1} e^{-\frac{\mathrm{d}(E,F)^2}{4(C_2+1)(t-s)}+(C_2+1)P_{\infty,\infty}^2(t-s)} \|\psi \|_{L^2_\omega(E)},
\end{equation*}
for all $s<t$, all measurable sets $E,F \subset \mathbb{R}^n$ and all $\psi \in L^2_\omega(\mathbb{R}^n)$ with support in $E$.
\end{proof}

\subsection{Pointwise Gaussian upper bounds and Moser’s \texorpdfstring{$L^2$}{L2}-\texorpdfstring{$L^\infty$}{Linfini} estimates}

As for pointwise Gaussian bounds, Hofmann and Kim \cite{hofmann2004gaussian} proved, in the case $\omega=1$, the equivalence between Moser’s $L^2$–$L^\infty$ estimates and Gaussian upper bounds, extending Aronson’s results to small complex perturbations of real coefficients. In the weighted case ($\omega \neq 1$), \cite{baadi2025degenerate} recently extended this result by establishing the same equivalence, so that Gaussian upper bounds follow immediately for real $A$, with Gaussian lower bounds then deduced from the Harnack inequality and H\"older continuity.

In the presence of lower-order terms, we will also derive Gaussian upper bounds using the assumption of Moser’s $L^2$-$L^\infty$ estimates for local weak solutions. We plan to study this under our assumption on $\omega$ and for real coefficients in a subsequent work.

For any $(t,x) \in \mathbb{R} \times \mathbb{R}^n$ and $r>0$, we set $Q_r(t,x):=(t-r^2,t]\times B(x,r)$ and $Q_r^\star(t,x):=[t,t+r^2)\times B(x,r)$ where $B(x,r)$ is the Euclidean ball of radius $r$ and center $x$. Thus, $Q_r(t,x)$ and $Q_r^\star(t,x)$ denote the usual forward and backward in time parabolic cylinders. For any $x \in \mathbb{R}^n$ and $r>0$, we set 
\begin{equation*}
    \omega_r(x):=\omega(B(x,\sqrt{r}))=\int_{B(x,\sqrt{r})}\omega(y) \, \mathrm{d}y.
\end{equation*}
Finally, we define the measure $\mu$ on $\mathbb{R}^{n+1}$ by $\mathrm{d} \mu(t,x):= \omega(x)\mathrm{d}x\mathrm{d}t.$

\begin{defn}[Local weak solutions]
    Fix an admissible pair $(r, q)$ and assume that $P_{r,q} < \infty$ and $P_{\infty,\infty} < \infty$, as in the previous section on Gaffney estimates. We say that $u$ is a local weak solution to the equation $\partial_t u + \mathcal{B}_{\mathrm{inh}}u=0$ in an open set $\Omega=I \times \mathcal{O} \subset \mathbb{R}^{1+n}$, with $I \subset \mathbb{R}$ and $\mathcal{O} \subset \mathbb{R}^n$ open sets, if $u \in L^\infty(I; L^2_\omega(\mathcal{O}))$ with $\nabla_x u \in L^2(I; L^2_\omega(\mathcal{O})^n)$ and satisfies the equation $\partial_t u + \mathcal{B}_{\mathrm{inh}}u=0$ in $\mathcal{D'}(\Omega)$ as in Definition \ref{def: distributions} with $\mathcal{O}$ replacing $\mathbb{R}^n$. Local weak solutions to the equation $-\partial_t v + \mathcal{B}_{\mathrm{inh}}^\star v=0$ are defined similarly.
\end{defn}

\begin{rem}
For a local weak solution $u$ as above, localizing in space and applying the first embedding of (4) in Proposition \ref{prop:embeddings} together with Lemma \ref{lem:xSobolev} yield $u \in L^r(I; L^q_{\omega^{q/2}, \mathrm{loc}}(\mathcal{O}))$. Moreover, if $\mathcal{O'}$ is an open set with $\overline{\mathcal{O'}} \subset \mathcal{O}$, then such local weak solutions are continuous in time with values in $L^2_\omega(\mathcal{O'})$. This follows directly from Point (2) of Remark \ref{rem: Lions inh}.
\end{rem}

\begin{defn}[Moser's $L^2$-$L^\infty$ estimates]\label{def:Moser}
    We say that $\partial_t  + \mathcal{B}_{\mathrm{inh}}$, respectively $-\partial_t  + \mathcal{B}_{\mathrm{inh}}^\star$, satisfies Moser’s $L^2$-$L^\infty$ estimates if there exist constants $B>0$ and $0<R_0\le\infty$ such that that for all $R\in (0,R_0)$, $(t_0,x_0) \in \mathbb{R}^{1+n}$, and all local weak solutions of $\partial_t u  + \mathcal{B}_{\mathrm{inh}}u=0$ on a neighborhood of $Q_{2R}(t_0,x_0)$, respectively $-\partial_t v  + \mathcal{B}_{\mathrm{inh}}^\star v=0$ and $Q_{2R}^\star(t_0,x_0)$, respectively, have local bounds of the form, respectively,
    \begin{equation}\label{Moser H}
        \supess_{Q_{R}(t_0,x_0)} \left | u \right |=\sup_{t \in (t_0-R^2, t_0]}  \supess_{B(x_0,R)} \left | u(t,\cdot) \right |  \leq B \left ( \frac{1}{\mu(Q_{2R}(t_0,x_0))} \int_{Q_{2R}(t_0,x_0)} \left | u \right |^2 \ \mathrm d\mu\right  )^{1/2},
    \end{equation}
    \begin{equation}\label{Moser Hstar}
        \supess_{Q_{R}^\star(t_0,x_0)} \left | v \right |=\sup_{t \in [t_0,t_0+R^2)}  \supess_{B(x_0,R)} \left | v(t,\cdot) \right |  \leq B \left ( \frac{1}{\mu(Q_{2R}^\star(t_0,x_0))} \int_{Q_{2R}^\star(t_0,x_0)} \left | v \right |^2 \ \mathrm d\mu\right  )^{1/2}.
    \end{equation}
    The above equalities follow from the local $L^2_\omega$-valued continuity of local weak solutions. See \cite[Appendix A]{baadi2025degenerate}.
\end{defn}

The main result of this section is the following.

\begin{thm}[Pointwise Gaussian upper bounds]\label{Thm: pointwise bounds}
Assume that we are in the setting of Theorem \ref{thm :L2 off diagonal}. If $\partial_t  + \mathcal{B}_{\mathrm{inh}}$ and $-\partial_t  + \mathcal{B}_{\mathrm{inh}}^\star$ satisfy Moser's $L^2$-$L^\infty$ estimates, then $\partial_t  + \mathcal{B}_{\mathrm{inh}}$ has Gaussian upper bounds, that is, for all $t, s \in \mathbb{R}$ with $0<t - s < 4R_0^2$, $\Gamma(t,s)$ is an integral operator with a kernel $\Gamma(t,x;s,y)$ satisfying a pointwise Gaussian upper bound, that is, for almost every $(x,y) \in \mathbb{R}^{2n}$,
\begin{equation}\label{pgb}
    \left | \Gamma(t,x;s,y) \right | \leq \frac{\Tilde{C}}{\sqrt{\omega_{t-s}(x)}\sqrt{\omega_{t-s}(y)}} e^{-\frac{16}{c} \frac{|x-y|^2}{t-s}+cP_{\infty,\infty}^2(t-s)},
\end{equation}
where $\Tilde{C} = \Tilde{C}(D,B,C)$, $C$ is the constant from Theorem \ref{thm :L2 off diagonal}, and $c$ is the constant from the same theorem. Moreover, in the case $R_0 < \infty$, there exist constants $K_0 = K_0([\omega]_{A_2}, n,c) \ge 1$ and $k_0 = k_0([\omega]_{A_2}, n) \ge 1$ such that, for all $\ell \ge 1$ and all $t,s$ with $\ell(4R_0^2) \le t-s < (\ell+1)(4R_0^2)$, the operator $\Gamma(t,s)$ is an integral operator with kernel $\Gamma(t,x;s,y)$, which satisfies, for almost every $(x,y) \in \mathbb{R}^{2n}$,
\begin{equation}\label{eq: iteration}
    \left | \Gamma(t,x;s,y) \right | \leq \frac{\Tilde{C}^{\ell+1} K_0^{\ell+2}}{\sqrt{\omega_{t-s}(x)}\sqrt{\omega_{t-s}(y)}} e^{-\tfrac{16|x-y|^2}{k_0^{2}c (t-s)}+cP_{\infty,\infty}^2 (t-s)}.
\end{equation}
The function $\Gamma(t,x;s,y)$ is referred to as the generalized fundamental solution of $\partial_t  + \mathcal{B}_{\mathrm{inh}}$.

\end{thm}
\begin{rem}
The factor $\frac{1}{\sqrt{\omega_{t-s}(x)}\sqrt{\omega_{t-s}(y)}}$ appearing in \eqref{pgb} may be replaced by one of 
\begin{equation*}
\frac{1}{\omega_{t-s}(x) }, \ \ \frac{1}{\omega_{t-s}(y) }, \ \ \frac{1}{\max(\omega_{t-s}(x) ,\omega_{t-s}(y))},
\end{equation*}
and $\Tilde{C}$ and $c$ are replaced respectively by $C^\circ=C^\circ(\Tilde{C},c,D)$ and $2c$. See \cite[Rem. 3]{cruz2014corrigendum}.
\end{rem}
\begin{proof}[Proof of Theorem \ref{Thm: pointwise bounds}]
Let $\gamma \geq0$ and $\chi \in \mathrm{Lip}(\mathbb{R}^n)$ a bounded Lipschitz function such that $\left \| \nabla_x \chi \right \|_{L^\infty(\mathbb{R}^n)} \leq \gamma.$ We fix $\psi \in \mathcal{D}(\mathbb{R}^n)$ and also fix $s \in \mathbb{R}$. Under the hypothesis of Theorem \ref{thm :L2 off diagonal}, we rewrite \eqref{eq: scaling2} and \eqref{eq: scaling3} (after scaling) as
\begin{equation}\label{UUUU}
\|e^\chi \Gamma(t,s) e^{-\chi} \psi \|_{L^2_\omega(\mathbb{R}^n)} \le C e^{cP_{\infty,\infty}^2 (t-s)} e^{c\| \nabla_x\chi \|^2_{L^\infty(\mathbb{R}^n)}(t-s)} \|  \psi\|_{L^2_\omega(\mathbb{R}^n)} \quad \text{for all} \ t>s,
\end{equation}
with $C$, $c$ the constants obtained there that are independent of $\chi$. By duality, for all $t>s$, we have
\begin{equation*}
\|e^{-\chi} \Tilde{\Gamma}(s,t) e^{\chi} \psi \|_{L^2_\omega(\mathbb{R}^n)} \le C e^{cP_{\infty,\infty}^2 (t-s)} e^{c\| \nabla_x\chi \|^2_{L^\infty(\mathbb{R}^n)}(t-s)} \|  \psi\|_{L^2_\omega(\mathbb{R}^n)}.
\end{equation*}
For all $t>s$, we set 
\begin{equation*}
    U(t):=\Gamma(t,s)e^{-\chi} \psi.
\end{equation*}
For all $t, s \in \mathbb{R}$ with $0 < t - s < 4R_0^2$, Moser's $L^2$–$L^\infty$ estimate \eqref{Moser H} with $R = \frac{\sqrt{t-s}}{2} < R_0$ implies that for all $x \in \mathbb{R}^n$ and for almost every $z \in B(x,R)$,
\begin{align*}
    \left | U(t,z) \right |^2\leq \frac{B^2}{\mu(Q_{\sqrt{t-s}}(t,x))}\int_{s}^{t}
\int_{B(x,\sqrt{t-s})} \left | U(\tau,y) \right |^2\ \mathrm d\mu,
\end{align*}
hence, 
\begin{align*}
     |e^{\chi(z)} U(t,z) |^2 &\leq \frac{B^2}{(t-s)\omega_{t-s}(x)}\int_{s}^{t}
\int_{B(x,\sqrt{t-s})}e^{2(\chi(z)-\chi(y)}  | e^{\chi(y)}U(\tau,y)  |^2\ \mathrm d\omega(y)  \mathrm d\tau
\\ & \leq \frac{B^2 e^{4\gamma \sqrt{t-s}}}{(t-s)\omega_{t-s}(x)} \int_{s}^{t}
\int_{B(x,\sqrt{t-s})} | e^{\chi(y)}U(\tau,y) |^2\ \mathrm d\omega(y) \mathrm d \tau
\\ & \leq \frac{B^2 e^{4\gamma \sqrt{t-s}}}{(t-s)\omega_{t-s}(x)} \int_{s}^{t} \|e^{\chi}U(\tau)   \|^2_{2,\omega} \ \mathrm d\tau.
\end{align*}
Using \eqref{UUUU}, we have for almost every $z \in B(x,R)$,
\begin{equation}\label{JJJJ}
     |e^{\chi(z)} U(t,z)  |^2  \leq \frac{C^2B^2 e^{4\gamma \sqrt{t-s}}}{(t-s)\omega_{t-s}(x)}  \left ( \int_{s}^{t} e^{2cP_{\infty,\infty}^2 (\tau-s)} e^{2c \gamma^2(\tau-s)} \, \mathrm d\tau \right ) \left \|\psi  \right \|^2_{2,\omega}.
\end{equation}
Remark that for all $z \in B(x,R)$, we have $B(z,\frac{\sqrt{t-s}}{2})\subset B(x,\sqrt{t-s})$ and using \eqref{DoublingMuck} we obtain $$\omega_{t-s}(z)=\omega( B(z,\sqrt{t-s})) \le D \, \omega(B(z,\frac{\sqrt{t-s}}{2})) \le D \omega_{t-s}(x).$$
As $\tau-s \le t-s \ (< 4R_0^2)$, then in \eqref{JJJJ} we can majorate, for almost every $z \in B(x,R)$, with
\begin{equation*}
     |e^{\chi(z)} U(t,z)  |^2  \leq \frac{D C^2B^2 e^{4\gamma \sqrt{t-s}}}{\omega_{t-s}(z)}  e^{2cP_{\infty,\infty}^2 (t-s)} e^{2c \gamma^2(t-s)} \left \|\psi  \right \|^2_{2,\omega}.
\end{equation*}
Therefore,
\begin{equation}\label{AAAAAA}
     \| \sqrt{\omega_{t-s}}\ e^{\chi} \Gamma(t,s)e^{-\chi} \psi  \|_{L^\infty(\mathbb{R}^n)}\leq \sqrt{D} CB e^{cP_{\infty,\infty}^2 (t-s)}  e^{2\gamma \sqrt{t-s}+c \gamma^2(t-s)}  \left \|\psi  \right \|_{2,\omega}.
\end{equation}
Using \eqref{AAAAAA} and a duality argument, we obtain
\begin{equation*}
     \| e^{-\chi} \Tilde{\Gamma}(s,t) (\sqrt{\omega_{t-s}}\ e^{\chi}\psi ) \|_{2,\omega}\leq \sqrt{D} CB e^{cP_{\infty,\infty}^2 (t-s)}  e^{2\gamma \sqrt{t-s}+c \gamma^2(t-s)}  \left \|\psi  \right \|_{L_\omega^1(\mathbb{R}^n)}.
\end{equation*}
Using the same computations and arguments with $\Tilde{\Gamma}(s,t)$ (by using \eqref{Moser Hstar}), a duality argument yields
\begin{equation}\label{AAAAAAA}
     \|e^{\chi}\Gamma(t,s)(\sqrt{\omega_{t-s}}\ e^{-\chi}\psi )  \|_{2,\omega}\leq \sqrt{D} CB e^{cP_{\infty,\infty}^2 (t-s)}  e^{2\gamma \sqrt{t-s}+c \gamma^2(t-s)}  \left \|\psi  \right \|_{L_\omega^1(\mathbb{R}^n)}.
\end{equation}
Using the Chapman-Kolmogorov identities, we write $\Gamma(t,s)=\Gamma(t,\frac{t+s}{2})\Gamma(\frac{t+s}{2},s).$ Hence,
\begin{align*}
    \sqrt{\omega_{t-s}}\ &e^\chi \Gamma(t,s)(\sqrt{\omega_{t-s}}\ e^{-\chi}\psi)=\sqrt{\omega_{t-s} }\ e^{\chi}\Gamma(t,\frac{t+s}{2})e^{-\chi} ( e^{\chi}\Gamma(\frac{t+s}{2},t) \sqrt{\omega_{t-s}}\ e^{-\chi} \psi )
    \\&=\sqrt{\frac{\omega_{t-s}}{\omega_{\frac{t-s}{2}}}}\sqrt{\omega_{\frac{t-s}{2}}}\ e^{\chi}\Gamma(t,\frac{t+s}{2})e^{-\chi} \left ( e^{\chi}\Gamma(\frac{t+s}{2},t) \sqrt{\omega_{\frac{t-s}{2}}}\ e^{-\chi}\sqrt{\frac{\omega_{t-s}}{\omega_{\frac{t-s}{2}}}} \, \psi \right ).
\end{align*}
Notice that by the doubling property \eqref{DoublingMuck}, we have 
\begin{equation}\label{Doubling}
    \left \|  \sqrt{\frac{\omega_{t-s}}{\omega_{\frac{t-s}{2}}}} \right \|_{L^\infty(\mathbb{R}^n)} \leq D^{1/2}.
\end{equation}
Combining \eqref{AAAAAA}, \eqref{AAAAAAA}, \eqref{Doubling} and the above Chapman-Kolmogorov identity, we deduce that
\begin{equation}\label{gamma=0}
    \| \sqrt{\omega_{t-s}}\ e^{\chi} \Gamma(t,s)(\sqrt{\omega_{t-s}}\ e^{-\chi}\psi) \|_{L^\infty(\mathbb{R}^n)}\leq D^{3/2} C^2B^2 e^{cP_{\infty,\infty}^2 (t-s)}  e^{2\sqrt{2}\gamma \sqrt{t-s}+c \gamma^2(t-s)}  \left \|\psi  \right \|_{L_\omega^1(\mathbb{R}^n)}
\end{equation}
Since $L^\infty(\mathbb{R}^n)=L^\infty_\omega(\mathbb{R}^n)$, the estimate \eqref{gamma=0} and the Dunford-Pettis theorem \cite{dunford1940linear} ensures that, for all $0<t-s< 4R_0^2$, $\Gamma(t,s)$ is an integral operator with a unique kernel $\Gamma(t,x;s,y)$ satisfying, for almost all $x,y \in \mathbb{R}^n$, the estimate
\begin{align}\label{CCCCC}
    \left | \Gamma(t,x;s,y) \right | \leq \frac{D^{3/2} C^2B^2}{\sqrt{\omega_{t-s}(x)}\sqrt{\omega_{t-s}(y)}} \times e^{cP_{\infty,\infty}^2 (t-s)} e^{2\sqrt{2}\gamma \sqrt{t-s}+c \gamma^2(t-s)} e^{\chi(y)-\chi(x)}.
\end{align}
When $\gamma=0$, that is $\chi$ is constant, we obtain the following bound
\begin{align}\label{FF}
    \left | \Gamma(t,x;s,y) \right | \leq \frac{D^{3/2} C^2B^2}{\sqrt{\omega_{t-s}(x)}\sqrt{\omega_{t-s}(y)}} \times e^{cP_{\infty,\infty}^2 (t-s)} .
\end{align}
To prove \eqref{pgb}, we fix $s,t$ with $0<t-s< 4R_0^2$ and $x\neq y\in \mathbb{R}^n$ for which \eqref{CCCCC} is valid. If $\frac{|x-y|}{4\sqrt{2} \sqrt{t-s}} < 1$, we have $1\le e^{\frac{2}{c}} e^{-\frac{|x-y|^2}{16c(t-s)}}$ and using simply \eqref{FF}, we obtain
\begin{equation*}
    \left | \Gamma(t,x;s,y) \right | \leq \frac{D^{3/2} C^2B^2}{\sqrt{\omega_{t-s}(x)}\sqrt{\omega_{t-s}(y)}} e^{\frac{2}{c}} \times e^{-\frac{|x-y|^2}{16c(t-s)}+cP_{\infty,\infty}^2 (t-s)}.
\end{equation*}
Now, assume that $\frac{|x-y|}{4\sqrt{2} \sqrt{t-s}} \ge 1$. Let $\gamma:=\frac{|x-y|}{4c(t-s)}$, take $N>\gamma |x-y|$ and set, for all $z\in \mathbb{R}^n$, $$\chi(z):=\inf(\gamma|z-y|,N).$$
It is a bounded Lipschitz function satisfying $\| \nabla_x \chi \|_{L^\infty(\mathbb{R}^n)}=\gamma$, $\chi(x)=\gamma|x-y|$ and $\chi(y)=0$. Remark that $2\sqrt{2}\gamma \sqrt{t-s} \le \tfrac{\chi(x)}{2}$, thus, \eqref{CCCCC} becomes
\begin{align*}
    \left | \Gamma(t,x;s,y) \right | \leq \frac{D^{3/2} C^2B^2}{\sqrt{\omega_{t-s}(x)}\sqrt{\omega_{t-s}(y)}} \times e^{cP_{\infty,\infty}^2 (t-s)} e^{-\frac{\chi(x)}{2}+c \gamma^2(t-s)}.
\end{align*}
Remark also that $-\tfrac{\chi(x)}{2}+c\gamma^2(t-s)=-\tfrac{|x-y|^2}{16c(t-s)}$. Therefore, we obtain
\begin{align*}
    \left | \Gamma(t,x;s,y) \right | \leq \frac{D^{3/2} C^2B^2}{\sqrt{\omega_{t-s}(x)}\sqrt{\omega_{t-s}(y)}} \times e^{cP_{\infty,\infty}^2 (t-s)} e^{-\tfrac{|x-y|^2}{16c(t-s)}}.
\end{align*}
This concludes the proof in the case $R_0 = \infty$.
\subsubsection*{\textbf{\underline{Iteration in the case $R_0 < \infty$:}}}
We now suggest a way to iterate and thereby obtain \eqref{eq: iteration}. First, it is known that, for all $s<t$, the degenerate heat semigroup $e^{(t-s)\Delta_\omega}$ has a kernel $\mathcal{K}_{t-s}(x,y)$ satisfying two-sided Gaussian bounds \cite{cruz2014corrigendum, ataei2024fundamental, baadi2025degenerate}. More precisely, there exist constants $K_0 = K_0([\omega]_{A_2}, n) > 1$ and $k_0 = k_0([\omega]_{A_2}, n) > 1$ such that, for almost all $x,y \in \mathbb{R}^n$,
\begin{equation}\label{eq: semigroup}
   \frac{K_0^{-1}}{\sqrt{\omega_{t-s}(x)}\sqrt{\omega_{t-s}(y)}} e^{-\frac{k_0 |x-y|^2}{t-s}} \le \mathcal{K}_{t-s}(x,y) \le \frac{K_0}{\sqrt{\omega_{t-s}(x)}\sqrt{\omega_{t-s}(y)}}  e^{-\frac{|x-y|^2}{k_0(t-s)}}.
\end{equation}
The case $0 < t - s < 4R_0^2$ is covered by \eqref{pgb}, which has now been established. In this case, using the lower bound in \eqref{eq: semigroup}, we have, for almost all $x,y \in \mathbb{R}^n$,

\begin{equation}\label{eq: inegalite}
     \left | \Gamma(t,x;s,y) \right | \le \Tilde{C} \, K_0 \, e^{cP_{\infty,\infty}^2 (t-s)} \, \mathcal{K}_{\frac{ck_0}{16}(t-s)}(x,y) .
\end{equation}
Here, we have deliberately ignored the comparison between $\omega_{t-s}(x)$ and $\omega_{\frac{c k_0}{16}(t-s)}(x)$ for simplicity. Indeed, these quantities are comparable: there exist constants $\eta \in (0,1)$ and $\beta > 0$, depending only on $n$ and $[\omega]_{A_2}$, such that
\begin{equation*}
\beta^{-1}\left( \frac{|E|}{|B|} \right)^{\frac{1}{2\eta}} \leq \frac{\omega(E)}{\omega(B)} \leq \beta\left( \frac{|E|}{|B|} \right)^{2\eta},
\end{equation*}
whenever $B \subset \mathbb{R}^n$ is an Euclidean ball and for all measurable sets $E \subset B$. See \cite[Chapter 15.5]{heinonen2018nonlinear}. This will give another $K_0$ depending on $c$, which we still denote by $K_0$ for simplicity.

Fix $\ell \ge 1$ and let $s,t \in \mathbb{R}$ satisfy $\ell(4R_0^2) \le t-s < (\ell+1)(4R_0^2)$. For each $0 \le j \le \ell+1$, define
$$r_j := s + \frac{j(t-s)}{\ell+1},$$
so that $r_0 = s$, $r_{\ell+1} = t$, and $r_{j+1} - r_j = \frac{t-s}{\ell+1} \in (0, 4R_0^2)$. By the Chapman–Kolmogorov identities (see Theorem \ref{thm: representation}, Point (5)), we have
\begin{equation*}
    \Gamma(t,s)=\Gamma(t,r_\ell)\Gamma(r_\ell,r_{\ell-1})\dots\Gamma(r_1,s). 
\end{equation*}
In particular, $\Gamma(t,s)$ is an integral operator with kernel $\Gamma(t,x;s,y)$. Moreover, using the Chapman–Kolmogorov identities for $\Gamma$ and for the semigroup kernel (equivalently, the additivity of the semigroup), together with inequality \eqref{eq: inegalite}, we have, for almost every $x,y \in \mathbb{R}^n$, 
\begin{align*}
    \left | \Gamma(t,x;s,y) \right | &\leq \int_{(\mathbb{R}^n)^\ell} \left ( \prod_{j=0}^{\ell} \left | \Gamma(r_{j+1},z_{j+1};r_j,z_j) \right | \right ) \, \mathrm{d}\omega(z_\ell) \, \mathrm{d}\omega(z_{\ell-1})\dots\mathrm{d}\omega(z_1) \quad ( z_{\ell+1}:=x \ \text{and} \ z_0:=y) \\& \le (\Tilde{C} K_0)^{\ell+1}  \left ( e^{cP_{\infty,\infty}^2 \frac{t-s}{\ell+1}} \right )^{\ell+1} \int_{(\mathbb{R}^n)^\ell} \left ( \prod_{j=0}^{\ell} \mathcal{K}_{\frac{ck_0}{16}(r_{j+1}-r_j)}(z_{j+1},z_j) \right )  \, \mathrm{d}\omega(z_\ell)\, \mathrm{d}\omega(z_{\ell-1})\dots\mathrm{d}\omega(z_1)  \\&= (\Tilde{C} K_0)^{\ell+1} e^{cP_{\infty,\infty}^2 (t-s)} \mathcal{K}_{\frac{ck_0}{16}(t-s)}(x,y)
    \\& \le \frac{(\Tilde{C} K_0)^{\ell+1} K_0}{{\sqrt{\omega_{t-s}(x)}\sqrt{\omega_{t-s}(y)}}}e^{cP_{\infty,\infty}^2 (t-s)} e^{-\tfrac{16|x-y|^2}{c k_0^2 (t-s)}}
    \\&= \frac{\Tilde{C}^{\ell+1} K_0^{\ell+2}}{\sqrt{\omega_{t-s}(x)}\sqrt{\omega_{t-s}(y)}} e^{-\tfrac{16|x-y|^2}{c k_0^2 (t-s)}+cP_{\infty,\infty}^2 (t-s)},
\end{align*}
where, in the penultimate line, we used the upper bound in \eqref{eq: semigroup}.
\end{proof}

\begin{rem}
The factor $k_0^{-2}$ in \eqref{eq: iteration} arises from the lack of explicit formulas for the degenerate heat semigroup kernels, unlike in the unweighted case, where the kernels are explicit Gaussians and the factor $\frac{16}{c}$ in \eqref{pgb} remains unchanged under iteration.
\end{rem}

\begin{cor}\label{cor: cor}
Assume that we are in the setting of Theorem \ref{Thm: pointwise bounds}. Then, for all $t, s \in \mathbb{R}$ with $s<t$, the following properties hold.
\begin{enumerate}
\item (Adjointess property) For almost every $(x,y) \in \mathbb{R}^{2n}$, we have
$$ \Tilde{\Gamma}(s,y;t,x) = \overline{\Gamma(t,x;s,y)},$$
where $\Tilde{\Gamma}(s,y;t,x)$ is generalized fundamental solution of the adjoint operator $-\partial_t  + \mathcal{B}^\star_{\mathrm{inh}}$.
\item (Chapman-Kolmogorov identities) If $r\in (s,t)$, then for almost every $(x,y)\in \mathbb{R}^{2n}$, we have
 $$ \int_{\mathbb{R}^n} \Gamma(t,x;r,z) \Gamma(r,z;s,y) \ \mathrm{d}\omega(z)= \Gamma(t,x;s,y). $$
\end{enumerate}
\end{cor}
\begin{proof}
This follows directly from the adjointness property together with the Chapman--Kolmogorov identities, as stated in Theorem \ref{thm: representation}, Points (4) and (5).
\end{proof}

\subsubsection*{\textbf{Copyright}}
A CC-BY 4.0 \url{https://creativecommons.org/licenses/by/4.0/} public copyright license has been applied by the authors to the present document and will be applied to all subsequent versions up to the Author Accepted Manuscript arising from this submission.

\bibliographystyle{alpha}
\bibliography{references.bib}

@article{auscherbaadi2024fundamental,
    AUTHOR = {Auscher, P. and Baadi, K.},
     TITLE = {Fundamental solutions for parabolic equations and systems:
              universal existence, uniqueness, representation},
   JOURNAL = {J. Math. Anal. Appl.},
  FJOURNAL = {Journal of Mathematical Analysis and Applications},
    VOLUME = {552},
      YEAR = {2025},
    NUMBER = {1},
     PAGES = {Paper No. 129806, 48},
      ISSN = {0022-247X,1096-0813},
   MRCLASS = {35A08 (35A01 35A02 35A15 35K90 47D06)},
  MRNUMBER = {4924261},
       DOI = {10.1016/j.jmaa.2025.129806},
       URL = {https://doi.org/10.1016/j.jmaa.2025.129806},
}

@misc{auscherbaadi2025hardy,
  title={On {H}ardy-{L}ittlewood-{S}obolev estimates for degenerate {L}aplacians},
  author={Auscher, P. and Baadi, K.},
  year={2025},
  eprint={2506.20368},
  archivePrefix={arXiv},
  primaryClass={math.AP},
  note={To appear in Publicacions Matemàtiques. \href{https://doi.org/10.48550/arXiv.2506.20368}{ArXiv:2506.20368}}
}

@article{ataei2024fundamental,
    AUTHOR = {Ataei, A. and Nystr{\"o}m, K.},
     TITLE = {On fundamental solutions and {G}aussian bounds for degenerate
              parabolic equations with time-dependent coefficients},
   JOURNAL = {Potential Anal.},
  FJOURNAL = {Potential Analysis. An International Journal Devoted to the
              Interactions between Potential Theory, Probability Theory,
              Geometry and Functional Analysis},
    VOLUME = {62},
      YEAR = {2025},
    NUMBER = {3},
     PAGES = {465--483},
      ISSN = {0926-2601,1572-929X},
   MRCLASS = {35K08 (35K15 35K65)},
  MRNUMBER = {4877475},
       DOI = {10.1007/s11118-024-10143-7},
       URL = {https://doi.org/10.1007/s11118-024-10143-7},
}

@article {auscheregert2023universal,
    AUTHOR = {Auscher, P. and Egert, M.},
     TITLE = {A universal variational framework for parabolic equations and
              systems},
   JOURNAL = {Calc. Var. Partial Differential Equations},
  FJOURNAL = {Calculus of Variations and Partial Differential Equations},
    VOLUME = {62},
      YEAR = {2023},
    NUMBER = {9},
     PAGES = {Paper No. 249, 59},
      ISSN = {0944-2669,1432-0835},
   MRCLASS = {35K40 (26A33 35K41 35K51)},
  MRNUMBER = {4655794},
MRREVIEWER = {Bingchen\ Liu},
       DOI = {10.1007/s00526-023-02577-5},
       URL = {https://doi.org/10.1007/s00526-023-02577-5},
}

@article{AMP2019,
    AUTHOR = {Auscher, P. and Monniaux, S. and Portal, P.},
     TITLE = {On existence and uniqueness for non-autonomous parabolic
              {C}auchy problems with rough coefficients},
   JOURNAL = {Ann. Sc. Norm. Super. Pisa Cl. Sci. (5)},
  FJOURNAL = {Annali della Scuola Normale Superiore di Pisa. Classe di
              Scienze. Serie V},
    VOLUME = {19},
      YEAR = {2019},
    NUMBER = {2},
     PAGES = {387--471},
      ISSN = {0391-173X,2036-2145},
   MRCLASS = {35K15 (35B30)},
  MRNUMBER = {3973274},
MRREVIEWER = {S\'ebastien\ J.\ Boyaval},
}

@article{aronson1967bounds,
    AUTHOR = {Aronson, D. G.},
     TITLE = {Bounds for the fundamental solution of a parabolic equation},
   JOURNAL = {Bull. Amer. Math. Soc.},
  FJOURNAL = {Bulletin of the American Mathematical Society},
    VOLUME = {73},
      YEAR = {1967},
     PAGES = {890--896},
      ISSN = {0002-9904},
   MRCLASS = {35.63},
  MRNUMBER = {217444},
MRREVIEWER = {T.\ St\'ys},
       DOI = {10.1090/S0002-9904-1967-11830-5},
       URL = {https://doi.org/10.1090/S0002-9904-1967-11830-5},
}

@article{aronson1968non,
    AUTHOR = {Aronson, D. G.},
     TITLE = {Non-negative solutions of linear parabolic equations},
   JOURNAL = {Ann. Scuola Norm. Sup. Pisa Cl. Sci. (3)},
  FJOURNAL = {Annali della Scuola Normale Superiore di Pisa. Classe di
              Scienze. Serie III},
    VOLUME = {22},
      YEAR = {1968},
     PAGES = {607--694},
      ISSN = {0391-173X},
   MRCLASS = {35K20},
  MRNUMBER = {435594},
MRREVIEWER = {J.\ Chabrowski},
}

@article {baadi2025degenerate,
    AUTHOR = {Baadi, K.},
     TITLE = {Degenerate parabolic equations in divergence form: fundamental
              solution and {G}aussian bounds},
   JOURNAL = {J. Evol. Equ.},
  FJOURNAL = {Journal of Evolution Equations},
    VOLUME = {26},
      YEAR = {2026},
    NUMBER = {2},
     PAGES = {Paper No. 52},
      ISSN = {1424-3199,1424-3202},
   MRCLASS = {35A08 (35K08 35K10 35K15 35K65)},
  MRNUMBER = {5053972},
       DOI = {10.1007/s00028-026-01201-1},
       URL = {https://doi.org/10.1007/s00028-026-01201-1},
}

@article {MR772255,
    AUTHOR = {Chiarenza, F. and Serapioni, R.},
     TITLE = {Degenerate parabolic equations and {H}arnack inequality},
   JOURNAL = {Ann. Mat. Pura Appl. (4)},
  FJOURNAL = {Annali di Matematica Pura ed Applicata. Serie Quarta},
    VOLUME = {137},
      YEAR = {1984},
     PAGES = {139--162},
      ISSN = {0003-4622},
   MRCLASS = {35K65 (35B99)},
  MRNUMBER = {772255},
MRREVIEWER = {N.\ S.\ Trudinger},
       DOI = {10.1007/BF01789392},
       URL = {https://doi.org/10.1007/BF01789392},
}

@article {MR748366,
    AUTHOR = {Chiarenza, F. and Serapioni, R.},
     TITLE = {A {H}arnack inequality for degenerate parabolic equations},
   JOURNAL = {Comm. Partial Differential Equations},
  FJOURNAL = {Communications in Partial Differential Equations},
    VOLUME = {9},
      YEAR = {1984},
    NUMBER = {8},
     PAGES = {719--749},
      ISSN = {0360-5302,1532-4133},
   MRCLASS = {35K65},
  MRNUMBER = {748366},
MRREVIEWER = {R.\ E.\ Showalter},
       DOI = {10.1080/03605308408820346},
       URL = {https://doi.org/10.1080/03605308408820346},
}

@article{Chiarenza85,
    AUTHOR = {Chiarenza, F. and Serapioni, R.},
     TITLE = {A remark on a {H}arnack inequality for degenerate parabolic
              equations},
   JOURNAL = {Rend. Sem. Mat. Univ. Padova},
  FJOURNAL = {Rendiconti del Seminario Matematico della Universit\`a{} di
              Padova. The Mathematical Journal of the University of Padova},
    VOLUME = {73},
      YEAR = {1985},
     PAGES = {179--190},
      ISSN = {0041-8994},
   MRCLASS = {35K65 (35B45)},
  MRNUMBER = {799906},
MRREVIEWER = {Rouben\ Rostamian},
       URL = {http://www.numdam.org/item?id=RSMUP_1985__73__179_0},
}

@article {Ishige99,
    AUTHOR = {Ishige, K.},
     TITLE = {On the behavior of the solutions of degenerate parabolic
              equations},
   JOURNAL = {Nagoya Math. J.},
  FJOURNAL = {Nagoya Mathematical Journal},
    VOLUME = {155},
      YEAR = {1999},
     PAGES = {1--26},
      ISSN = {0027-7630,2152-6842},
   MRCLASS = {35K65 (35B05)},
  MRNUMBER = {1711391},
MRREVIEWER = {I.\ \L ojczyk-Kr\'olikiewicz},
       DOI = {10.1017/S0027763000006978},
       URL = {https://doi.org/10.1017/S0027763000006978},
}

@article{cruz2014corrigendum,
  title={Corrigendum to “{G}aussian bounds for degenerate parabolic equations”[{J}. {F}unct. {A}nal. 255 (2)(2008) 283--312]},
  author={Cruz-Uribe, D. and Rios, C.},
  journal={Journal of Functional Analysis},
  volume={267},
  number={9},
  pages={3507--3513},
  year={2014},
  publisher={Elsevier}
}

@article {Davies87,
    AUTHOR = {Davies, E. B.},
     TITLE = {Explicit constants for {G}aussian upper bounds on heat
              kernels},
   JOURNAL = {Amer. J. Math.},
  FJOURNAL = {American Journal of Mathematics},
    VOLUME = {109},
      YEAR = {1987},
    NUMBER = {2},
     PAGES = {319--333},
      ISSN = {0002-9327,1080-6377},
   MRCLASS = {58G11 (35K05)},
  MRNUMBER = {882426},
MRREVIEWER = {P.\ G\"unther},
       DOI = {10.2307/2374577},
       URL = {https://doi.org/10.2307/2374577},
}

@article{davies1992heat,
    AUTHOR = {Davies, E. B.},
     TITLE = {Heat kernel bounds, conservation of probability and the
              {F}eller property},
      NOTE = {Festschrift on the occasion of the 70th birthday of Shmuel
              Agmon},
   JOURNAL = {J. Anal. Math.},
  FJOURNAL = {Journal d'Analyse Math\'ematique},
    VOLUME = {58},
      YEAR = {1992},
     PAGES = {99--119},
      ISSN = {0021-7670,1565-8538},
   MRCLASS = {58G11 (47D07 47F05 58G32)},
  MRNUMBER = {1226938},
MRREVIEWER = {Kazuaki\ Taira},
       DOI = {10.1007/BF02790359},
       URL = {https://doi.org/10.1007/BF02790359},
}

@article{dunford1940linear,
    AUTHOR = {Dunford, N. and Pettis, B. J.},
     TITLE = {Linear operations on summable functions},
   JOURNAL = {Trans. Amer. Math. Soc.},
  FJOURNAL = {Transactions of the American Mathematical Society},
    VOLUME = {47},
      YEAR = {1940},
     PAGES = {323--392},
      ISSN = {0002-9947,1088-6850},
   MRCLASS = {46.3X},
  MRNUMBER = {2020},
MRREVIEWER = {T.\ H.\ Hildebrandt},
       DOI = {10.2307/1989960},
       URL = {https://doi.org/10.2307/1989960},
}

@article {Fabes1986,
    AUTHOR = {Fabes, E. B. and Stroock, D. W.},
     TITLE = {A new proof of {M}oser's parabolic {H}arnack inequality using
              the old ideas of {N}ash},
   JOURNAL = {Arch. Rational Mech. Anal.},
  FJOURNAL = {Archive for Rational Mechanics and Analysis},
    VOLUME = {96},
      YEAR = {1986},
    NUMBER = {4},
     PAGES = {327--338},
      ISSN = {0003-9527},
   MRCLASS = {35B45 (35K10 58C15)},
  MRNUMBER = {855753},
MRREVIEWER = {W.\ P.\ Ziemer},
       DOI = {10.1007/BF00251802},
       URL = {https://doi.org/10.1007/BF00251802},
}

@book{friedman2008partial,
    AUTHOR = {Friedman, A.},
     TITLE = {Partial differential equations of parabolic type},
 PUBLISHER = {Prentice-Hall, Inc., Englewood Cliffs, NJ},
      YEAR = {1964},
     PAGES = {xiv+347},
   MRCLASS = {35.00 (35.62)},
  MRNUMBER = {181836},
MRREVIEWER = {B.\ Frank\ Jones, Jr.},
}

@book{garcia2011weighted,
    AUTHOR = {{G}arc\'ia{-}Cuerva, J. and Francia, J.L. Rubio De },
     TITLE = {Weighted norm inequalities and related topics},
    SERIES = {North-Holland Mathematics Studies},
    VOLUME = {116},
      NOTE = {Notas de Matem\'atica, 104. [Mathematical Notes]},
 PUBLISHER = {North-Holland Publishing Co., Amsterdam},
      YEAR = {1985},
     PAGES = {x+604},
      ISBN = {0-444-87804-1},
   MRCLASS = {42B20 (42B25 46Exx 47B38)},
  MRNUMBER = {807149},
MRREVIEWER = {Kenneth\ F.\ Andersen},
}

@book{GrafakosBook2014,
	AUTHOR = {Grafakos, L.},
	TITLE = {Classical {F}ourier analysis},
	SERIES = {Graduate Texts in Mathematics},
	VOLUME = {249},
	EDITION = {Third},
	PUBLISHER = {Springer, New York},
	YEAR = {2014},
	PAGES = {xviii+638},
	ISBN = {978-1-4939-1193-6; 978-1-4939-1194-3},
	MRCLASS = {42-01 (42Bxx)},
	MRNUMBER = {3243734},
	MRREVIEWER = {Atanas\ G.\ Stefanov},
	DOI = {10.1007/978-1-4939-1194-3},
	URL = {https://doi.org/10.1007/978-1-4939-1194-3},
}

@book{haase2006functional,
    AUTHOR = {Haase, M.},
     TITLE = {The functional calculus for sectorial operators},
    SERIES = {Operator Theory: Advances and Applications},
    VOLUME = {169},
 PUBLISHER = {Birkh\"auser Verlag, Basel},
      YEAR = {2006},
     PAGES = {xiv+392},
      ISBN = {978-3-7643-7697-0; 3-7643-7697-X},
   MRCLASS = {47A60 (30E05 44A15 46B70 47A55 47D03 47E05 47F05)},
  MRNUMBER = {2244037},
MRREVIEWER = {Christian\ Le Merdy},
       DOI = {10.1007/3-7643-7698-8},
       URL = {https://doi.org/10.1007/3-7643-7698-8},
}

@book{heinonen2018nonlinear,
    AUTHOR = {Heinonen, J. and Kilpel\"ainen, T. and Martio, O.},
     TITLE = {Nonlinear potential theory of degenerate elliptic equations},
      NOTE = {Unabridged republication of the 1993 original},
 PUBLISHER = {Dover Publications, Inc., Mineola, NY},
      YEAR = {2006},
     PAGES = {xii+404},
      ISBN = {0-486-45050-3},
   MRCLASS = {31C45 (35J60 35J70)},
  MRNUMBER = {2305115},
MRREVIEWER = {Piotr\ Haj\l asz},
}

@article{hofmann2004gaussian,
    AUTHOR = {Hofmann, S. and Kim, S.},
     TITLE = {Gaussian estimates for fundamental solutions to certain
              parabolic systems},
   JOURNAL = {Publ. Mat.},
  FJOURNAL = {Publicacions Matem\`atiques},
    VOLUME = {48},
      YEAR = {2004},
    NUMBER = {2},
     PAGES = {481--496},
      ISSN = {0214-1493,2014-4350},
   MRCLASS = {35A08 (35B45 35K35)},
  MRNUMBER = {2091016},
MRREVIEWER = {Francesco\ Zirilli},
       DOI = {10.5565/PUBLMAT\_48204\_10},
       URL = {https://doi.org/10.5565/PUBLMAT_48204_10},
}

@book {hytonen2016analysis,
    AUTHOR = {Hyt\"onen, T. and van Neerven, J. and Veraar, M. and
              Weis, L.},
     TITLE = {Analysis in {B}anach spaces. {V}ol. {I}. {M}artingales and
              {L}ittlewood-{P}aley theory},
    SERIES = {Ergebnisse der Mathematik und ihrer Grenzgebiete. 3. Folge. A
              Series of Modern Surveys in Mathematics [Results in
              Mathematics and Related Areas. 3rd Series. A Series of Modern
              Surveys in Mathematics]},
    VOLUME = {63},
 PUBLISHER = {Springer, Cham},
      YEAR = {2016},
     PAGES = {xvi+614},
      ISBN = {978-3-319-48519-5; 978-3-319-48520-1},
   MRCLASS = {46-02 (42B35 46E30)},
  MRNUMBER = {3617205},
MRREVIEWER = {Adam\ Os\polhk ekowski},
}

@article{kaplan1966abstract,
    AUTHOR = {Kaplan, S.},
     TITLE = {Abstract boundary value problems for linear parabolic
              equations},
   JOURNAL = {Ann. Scuola Norm. Sup. Pisa Cl. Sci. (3)},
  FJOURNAL = {Annali della Scuola Normale Superiore di Pisa. Classe di
              Scienze. Serie III},
    VOLUME = {20},
      YEAR = {1966},
     PAGES = {395--419},
      ISSN = {0391-173X},
   MRCLASS = {35.65},
  MRNUMBER = {200593},
MRREVIEWER = {O.\ Hor\'a\v cek},
}

@article{kato1961abstract,
    AUTHOR = {Kato, T.},
     TITLE = {Abstract evolution equations of parabolic type in {B}anach and
              {H}ilbert spaces},
   JOURNAL = {Nagoya Math. J.},
  FJOURNAL = {Nagoya Mathematical Journal},
    VOLUME = {19},
      YEAR = {1961},
     PAGES = {93--125},
      ISSN = {0027-7630,2152-6842},
   MRCLASS = {34.95 (35.95)},
  MRNUMBER = {143065},
MRREVIEWER = {J.\ Gil de Lamadrid},
       URL = {http://projecteuclid.org/euclid.nmj/1118800864},
}

@article{kilpelainen1994weighted,
    AUTHOR = {Kilpel\"ainen, T.},
     TITLE = {Weighted {S}obolev spaces and capacity},
   JOURNAL = {Ann. Acad. Sci. Fenn. Ser. A I Math.},
  FJOURNAL = {Annales Academiae Scientiarum Fennicae. Series A I.
              Mathematica},
    VOLUME = {19},
      YEAR = {1994},
    NUMBER = {1},
     PAGES = {95--113},
      ISSN = {0066-1953},
   MRCLASS = {46E35 (31C15)},
  MRNUMBER = {1246890},
MRREVIEWER = {W.\ P.\ Ziemer},
}

@article {MR4387945,
    AUTHOR = {Kim, D. and Ryu, S. and Woo, K.},
     TITLE = {Parabolic equations with unbounded lower-order coefficients in
              {S}obolev spaces with mixed norms},
   JOURNAL = {J. Evol. Equ.},
  FJOURNAL = {Journal of Evolution Equations},
    VOLUME = {22},
      YEAR = {2022},
    NUMBER = {1},
     PAGES = {Paper No. 9, 40},
      ISSN = {1424-3199,1424-3202},
   MRCLASS = {35K10 (35R05 46E35)},
  MRNUMBER = {4387945},
       DOI = {10.1007/s00028-022-00761-2},
       URL = {https://doi.org/10.1007/s00028-022-00761-2},
}

@book{ladyzhenskaia1968linear,
    AUTHOR = {Ladyženskaja, O. A. and Solonnikov, V. A. and Ural’ceva, N. N},
     TITLE = {Linear and quasilinear equations of parabolic type},
    SERIES = {Translations of Mathematical Monographs},
    VOLUME = {23},
      NOTE = {Translated from the Russian by S. Smith},
 PUBLISHER = {American Mathematical Society, Providence, RI},
      YEAR = {1968},
     PAGES = {xi+648},
   MRCLASS = {35.62},
  MRNUMBER = {241822},
MRREVIEWER = {B.\ Frank\ Jones, Jr.},
}

@article{lions1957problemes,
    AUTHOR = {Lions, J.-L.},
     TITLE = {Sur les probl\`emes mixtes pour certains syst\`emes
              paraboliques dans des ouverts non cylindriques},
   JOURNAL = {Ann. Inst. Fourier (Grenoble)},
  FJOURNAL = {Universit\'e{} de Grenoble. Annales de l'Institut Fourier},
    VOLUME = {7},
      YEAR = {1957},
     PAGES = {143--182},
      ISSN = {0373-0956,1777-5310},
   MRCLASS = {35.00 (46.00)},
  MRNUMBER = {102666},
MRREVIEWER = {K.\ Yosida},
       URL = {http://www.numdam.org/item?id=AIF_1957__7__143_0},
}

@book{lions2013equations,
    AUTHOR = {Lions, J.-L.},
     TITLE = {\'Equations diff\'erentielles op\'erationnelles et probl\`emes
              aux limites},
    SERIES = {Die Grundlehren der mathematischen Wissenschaften},
    VOLUME = {Band 111},
 PUBLISHER = {Springer-Verlag, Berlin-G\"ottingen-Heidelberg},
      YEAR = {1961},
     PAGES = {ix+292},
   MRCLASS = {35.00 (34.95)},
  MRNUMBER = {153974},
MRREVIEWER = {S.\ Zaidman},
}

@incollection{McIntosh86,
    AUTHOR = {McIntosh, A.},
     TITLE = {Operators which have an {$H_\infty$} functional calculus},
 BOOKTITLE = {Miniconference on operator theory and partial differential
              equations ({N}orth {R}yde, 1986)},
    SERIES = {Proc. Centre Math. Anal. Austral. Nat. Univ.},
    VOLUME = {14},
     PAGES = {210--231},
 PUBLISHER = {Austral. Nat. Univ., Canberra},
      YEAR = {1986},
      ISBN = {0-86784-517-1},
   MRCLASS = {47A60},
  MRNUMBER = {912940},
MRREVIEWER = {Florian\ Horia\ Vasilescu},
}

@article{nash1958continuity,
    AUTHOR = {Nash, J.},
     TITLE = {Continuity of solutions of parabolic and elliptic equations},
   JOURNAL = {Amer. J. Math.},
  FJOURNAL = {American Journal of Mathematics},
    VOLUME = {80},
      YEAR = {1958},
     PAGES = {931--954},
      ISSN = {0002-9327,1080-6377},
   MRCLASS = {35.00},
  MRNUMBER = {100158},
MRREVIEWER = {C.\ B.\ Morrey, Jr.},
       DOI = {10.2307/2372841},
       URL = {https://doi.org/10.2307/2372841},
}

@book{reed1980methods,
    AUTHOR = {Reed, M. and Simon, B.},
     TITLE = {Methods of modern mathematical physics: Functional analysis},
   EDITION = {Second},
 PUBLISHER = {Academic Press, Inc. [Harcourt Brace Jovanovich, Publishers],
              New York},
      YEAR = {1980},
     PAGES = {xv+400},
      ISBN = {0-12-585050-6},
   MRCLASS = {46-01 (47-01)},
  MRNUMBER = {751959},
}

@book{Stein1993_HA,
    AUTHOR = {Stein, E.M.},
     TITLE = {Harmonic {A}nalysis: {R}eal-{V}ariable {M}ethods, {O}rthogonality, and {O}scillatory {I}ntegrals},
    SERIES = {Princeton Mathematical Series},
    VOLUME = {43},
      NOTE = {With the assistance of Timothy S. Murphy,
              Monographs in Harmonic Analysis, III},
 PUBLISHER = {Princeton University Press, Princeton, NJ},
      YEAR = {1993},
     PAGES = {xiv+695},
      ISBN = {0-691-03216-5},
   MRCLASS = {42-02 (35Sxx 43-02 47G30)},
  MRNUMBER = {1232192},
MRREVIEWER = {Michael\ Cowling},
}

@book{stein1971introduction,
    AUTHOR = {Stein, E.M. and Weiss, G.},
     TITLE = {Introduction to {F}ourier analysis on {E}uclidean spaces},
    SERIES = {Princeton Mathematical Series},
    VOLUME = {No. 32},
 PUBLISHER = {Princeton University Press, Princeton, NJ},
      YEAR = {1971},
     PAGES = {x+297},
   MRCLASS = {42A92 (31B99 32A99 46F99 47G05)},
  MRNUMBER = {304972},
MRREVIEWER = {Edwin\ Hewitt},
}

@article {MR1662313,
    AUTHOR = {Tartar, L.},
     TITLE = {Imbedding theorems of {S}obolev spaces into {L}orentz spaces},
   JOURNAL = {Boll. Unione Mat. Ital. Sez. B Artic. Ric. Mat. (8)},
  FJOURNAL = {Bollettino della Unione Matematica Italiana. Serie VIII.
              Sezione B. Articoli di Ricerca Matematica},
    VOLUME = {1},
      YEAR = {1998},
    NUMBER = {3},
     PAGES = {479--500},
      ISSN = {0392-4041},
   MRCLASS = {46E35 (46M35)},
  MRNUMBER = {1662313},
MRREVIEWER = {Lubo\v s\ Pick},
}

@article {MR742415,
    AUTHOR = {Baras, P. and Goldstein, J. A.},
     TITLE = {The heat equation with a singular potential},
   JOURNAL = {Trans. Amer. Math. Soc.},
  FJOURNAL = {Transactions of the American Mathematical Society},
    VOLUME = {284},
      YEAR = {1984},
    NUMBER = {1},
     PAGES = {121--139},
      ISSN = {0002-9947,1088-6850},
   MRCLASS = {35K05 (60J65)},
  MRNUMBER = {742415},
MRREVIEWER = {H.\ S.\ Bear},
       DOI = {10.2307/1999277},
       URL = {https://doi.org/10.2307/1999277},
}

@article {MR311856,
    AUTHOR = {Muckenhoupt, B.},
     TITLE = {Hardy's inequality with weights},
   JOURNAL = {Studia Math.},
  FJOURNAL = {Polska Akademia Nauk. Instytut Matematyczny. Studia
              Mathematica},
    VOLUME = {44},
      YEAR = {1972},
     PAGES = {31--38},
      ISSN = {0039-3223,1730-6337},
   MRCLASS = {26A84 (46E30)},
  MRNUMBER = {311856},
MRREVIEWER = {G.\ O.\ Okikiolu},
       DOI = {10.4064/sm-44-1-31-38},
       URL = {https://doi.org/10.4064/sm-44-1-31-38},
}

@article {MR1007530,
    AUTHOR = {Sinnamon, G.},
     TITLE = {A weighted gradient inequality},
   JOURNAL = {Proc. Roy. Soc. Edinburgh Sect. A},
  FJOURNAL = {Proceedings of the Royal Society of Edinburgh. Section A.
              Mathematics},
    VOLUME = {111},
      YEAR = {1989},
    NUMBER = {3-4},
     PAGES = {329--335},
      ISSN = {0308-2105,1473-7124},
   MRCLASS = {26D10},
  MRNUMBER = {1007530},
MRREVIEWER = {Kenneth\ F.\ Andersen},
       DOI = {10.1017/S0308210500018606},
       URL = {https://doi.org/10.1017/S0308210500018606},
}

\end{document}